\documentclass[10pt, reqno]{amsart}
\usepackage{amssymb}
\usepackage{graphicx}
\usepackage{xcolor} 
\usepackage{tensor}
\usepackage{tikz}

\usepackage[margin=1in, marginpar=.6in]{geometry} 

\usepackage{enumitem}

\allowdisplaybreaks				

\usepackage{marvosym}			

\usepackage[T1]{fontenc}

\usepackage{mathrsfs}

\usepackage[sort,nocompress]{cite}

\definecolor{green}{rgb}{0,0.8,0} 

\newtheorem{theorem}{Theorem}[section]

\newtheorem{lemma}[theorem]{Lemma}
\newtheorem{proposition}[theorem]{Proposition}
\theoremstyle{definition}
\newtheorem{definition}[theorem]{Definition}

\theoremstyle{remark}
\newtheorem{remark}[theorem]{Remark}
\numberwithin{equation}{section}
\newcommand{\relphantom}[1]{\mathrel{\phantom{#1}}}

\makeatletter
\newcommand{\nrm}{\@ifstar{\nrmb}{\nrmi}}
\newcommand{\nrmi}[1]{\Vert{#1}\Vert}
\newcommand{\nrmb}[1]{\left\Vert{#1}\right\Vert}
\newcommand{\abs}{\@ifstar{\absb}{\absi}}
\newcommand{\absi}[1]{\vert{#1}\vert}
\newcommand{\absb}[1]{\left\vert{#1}\right\vert}
\newcommand{\brk}{\@ifstar{\brkb}{\brki}}
\newcommand{\brki}[1]{\langle{#1}\rangle}
\newcommand{\brkb}[1]{\left\langle{#1}\right\rangle}
\newcommand{\set}{\@ifstar{\setb}{\seti}}
\newcommand{\seti}[1]{\{#1\}}
\newcommand{\setb}[1]{\left\{ #1\right\}}
\makeatother

\newcommand{\td}[1]{\widetilde{#1}}
\newcommand{\br}[1]{\overline{#1}}

\newcommand{\VERT}[1]{{\left\vert\kern-0.25ex\left\vert\kern-0.25ex\left\vert #1
    \right\vert\kern-0.25ex\right\vert\kern-0.25ex\right\vert}}

\DeclareMathOperator{\dist}{dist}
\DeclareMathOperator{\sgn}{sgn}
\DeclareMathOperator{\supp}{supp}
\DeclareMathOperator{\tr}{tr}
\let\div\relax
\DeclareMathOperator{\div}{div}

\newcommand{\aeq}{\simeq}
\newcommand{\aleq}{\lesssim}
\newcommand{\ageq}{\gtrsim}

\newcommand{\uD}{\mathrm{D}}
\newcommand{\ud}{\mathrm{d}}
\newcommand{\rd}{\partial}
\newcommand{\nb}{\nabla}

\newcommand{\imp}{\Rightarrow}

\newcommand{\impmi}{\Leftrightarrow}

\newcommand{\bb}{\Big}

\newcommand{\0}{\emptyset}

\newcommand{\peq}{\relphantom{=}}			

\newcommand{\alp}{\alpha}
\newcommand{\bt}{\beta}
\newcommand{\gmm}{\gamma}
\newcommand{\Gmm}{\Gamma}
\newcommand{\dlt}{\delta}
\newcommand{\Dlt}{\Delta}
\newcommand{\eps}{\epsilon}

\newcommand{\kpp}{\kappa}
\newcommand{\lmb}{\lambda}
\newcommand{\Lmb}{\Lambda}

\newcommand{\sgm}{\sigma}
\newcommand{\Sgm}{\Sigma}
\newcommand{\tht}{\theta}
\newcommand{\Tht}{\Theta}

\newcommand{\omg}{\omega}
\newcommand{\Omg}{\Omega}

\newcommand{\bfe}{{\bf e}}
\newcommand{\bff}{{\bf f}}
\newcommand{\bfg}{{\bf g}}

\newcommand{\bfk}{{\bf k}}

\newcommand{\bfC}{{\bf C}}
\newcommand{\bfD}{{\bf D}}
\newcommand{\bfE}{{\bf E}}
\newcommand{\bfF}{{\bf F}}
\newcommand{\bfG}{{\bf G}}
\newcommand{\bfH}{{\bf H}}

\newcommand{\bfJ}{{\bf J}}

\newcommand{\bfM}{{\bf M}}

\newcommand{\bfP}{{\bf P}}
\newcommand{\bfQ}{{\bf Q}}

\newcommand{\bfT}{{\bf T}}

\newcommand{\bfX}{{\bf X}}
\newcommand{\bfY}{{\bf Y}}

\newcommand{\bfalp}{\boldsymbol{\alpha}}

\newcommand{\bfdlt}{\boldsymbol{\delta}}

\newcommand{\bfeta}{\boldsymbol{\eta}}

\newcommand{\bfnu}{\boldsymbol{\nu}}
\newcommand{\bfxi}{\boldsymbol{\xi}}

\newcommand{\bfPsi}{\boldsymbol{\Psi}}
\newcommand{\bfomg}{\boldsymbol{\omega}}
\newcommand{\bfOmg}{\boldsymbol{\Omega}}

\newcommand{\bbJ}{\mathbb J}

\newcommand{\bbN}{\mathbb N}

\newcommand{\bbP}{\mathbb P}

\newcommand{\bbR}{\mathbb R}
\newcommand{\bbS}{\mathbb S}

\newcommand{\bbU}{\mathbb U}

\newcommand{\bbZ}{\mathbb Z}

\newcommand{\calE}{\mathcal E}

\newcommand{\calH}{\mathcal H}

\newcommand{\calL}{\mathcal L}
\newcommand{\calM}{\mathcal M}

\newcommand{\calO}{\mathcal O}

\newcommand{\calQ}{\mathcal Q}

\newcommand{\calU}{\mathcal U}

\newcommand{\bm}{{\rm b}}
\newcommand{\Hb}{H_{\bm}}
\newcommand{\Ric}{\mathrm{Ric}}

\begin{document}

\title[]{Initial data gluing in the asymptotically flat regime \\ via solution operators with prescribed support properties}
\author{Yuchen Mao}%
\address{Department of Mathematics, UC Berkeley, Berkeley, CA, USA 94720}%
\email{yuchen\_mao@berkeley.edu}%

\author{Sung-Jin Oh}%
\address{Department of Mathematics, UC Berkeley, Berkeley, CA, USA 94720 and Korea Institute for Advanced Study, Seoul, Korea 02455}%
\email{sjoh@math.berkeley.edu}%

\author{Zhongkai Tao}%
\address{Department of Mathematics, UC Berkeley, Berkeley, CA, USA 94720}%
\email{ztao@math.berkeley.edu}%

%

\begin{abstract}
We give new proofs of general relativistic initial data gluing results on unit-scale annuli based on explicit solution operators for the linearized constraint equation around the flat case with prescribed support properties. These results retrieve and optimize -- in terms of positivity, regularity, size and/or spatial decay requirements -- a number of known theorems concerning asymptotically flat initial data, including Kerr exterior gluing by Corvino--Schoen and Chru\'{s}ciel--Delay, interior gluing (or ``fill-in'') by Bieri--Chru\'{s}ciel, and obstruction-free gluing by Czimek--Rodnianski. In particular, our proof of the strengthened obstruction-free gluing theorem relies on purely spacelike techniques, rather than null gluing as in the original approach.
\end{abstract}
\maketitle
\section{Introduction and main results} \label{sec:intro}
Let $(\calM^{3}, g)$ be a Riemannian manifold and $k$ be a symmetric (contravariant) $2$-tensor on $\calM^{3}$. The \emph{Einstein vacuum constraint equation} reads
\begin{equation} \label{eq:constraint}
\left\{
\begin{aligned}R[g] + (\tr_{g} k)^{2} - \abs{k}_{g}^{2} &= 0, \\
	\div_{g} k - \ud \tr_{g} k &= 0,
\end{aligned}
\right.
\end{equation}
where $R[g]$ is the scalar curvature of $g$, $\div_{g} k_{j} = D_{i} k_{i' j} (g^{-1})^{i i'}$ with $D_{i}$ the Levi-Civita connection associated with $g$, $\tr_{g} k = (g^{-1})^{ij} k_{ij}$ and $\abs{k}_{g}^{2} = k_{ij} k_{i' j'} (g^{-1})^{i i'} (g^{-1})^{j j'}$.
By the fundamental theorem of Choquet-Bruhat \cite{ChBr}, to any (sufficiently regular) triple $(\calM^{3}, g, k)$ solving \eqref{eq:constraint} corresponds a (geometrically) unique spacetime $(\calM^{1+3}, \bfg)$ solving the Einstein vacuum equation
\begin{equation} \label{eq:eve}
\Ric[\bfg] - \frac{1}{2} \bfg \tr_{\bfg} \Ric[\bfg] = 0,
\end{equation}
into which $(\calM^{3}, g)$ isometrically embeds with $k$ as the second fundamental form. In this sense, a pair $(g, k)$ solving \eqref{eq:constraint} constitutes an initial data set for the Einstein vacuum equation.

While underdetermined, the nonlinear nature of \eqref{eq:constraint} imposes nontrivial constraints on the class of initial data sets for the Einstein vacuum equation, as exemplified by the celebrated positive mass theorem \cite{SchYau, Wit}. Understanding the flexibility of initial data solving \eqref{eq:constraint} is often an indispensable part of the study of solutions to the Einstein vacuum equation.

The subject of this paper is \emph{initial data gluing}, which asks: given two initial data sets, find -- if possible -- another initial data set which contains the two. We focus on initial data sets $(\calM^{3}, g, k)$ such that either (1)~$(\calM^{3}, g)$ is bounded and $(g, k)$ is almost flat (i.e., close to the flat data $(\dlt, 0)$) or (2)~$\calM^{3}$ is unbounded but $(g, k)$ is asymptotically flat (i.e., $\calM^{3}$ is diffeomorphic to the exterior of a ball in $\bbR^{3}$ and $(g, k) \to (\dlt, 0)$ as $\abs{x} \to \infty$). These two cases are closely related as the latter is typically reduced to the former via rescaling. Since the pioneering work of Corvino \cite{Cor}, many results have appeared on initial data gluing in this setting, such as (1)~gluing a general asymptotically flat initial data set to the exterior of initial data for one of the Kerr spacetimes \cite{ChrCorIse, CorSch, ChrDel}, and (2)~gluing an interior region to a general asymptotically flat initial data defined in the exterior of a ball to produce a globally defined initial data \cite{Chr2, BieChr}, and so on. See also the recent works \cite{Hin1, Hin2, Hin3} adopting a geometric microlocal approach. For further discussion and additional references, we refer the reader to the excellent review article \cite{Car}.

The first contribution of this paper is a new short proof of a basic gluing result on an almost flat annulus at unit scale (Theorem~\ref{thm:gluing-annulus}), which leads to the above gluing results (1) and (2) for asymptotically flat initial data after suitable rescaling arguments; see also Theorem~\ref{thm:corvino-schoen} and Remark~\ref{rem:obs-free-af}. The main simplification comes from the use of explicit Bogovskii-type solution operators to solve the linearization of \eqref{eq:constraint} around the flat case $(g_{ij}, k_{ij}) = (\dlt_{ij}, 0)$ on an annular domain with zero boundary condition. As a byproduct of its simple and explicit nature, our proof readily handles the optimal (modulo the endpoint) Sobolev regularity.

Another contribution of this paper is a new approach to (a strengthening of) \emph{obstruction-free gluing} \`a la Czimek--Rodnianski \cite{CziRod}; see Theorems~\ref{thm:obs-free-unit} and~\ref{thm:obs-free-af} for the precise formulation. This novel type of gluing theorem, which relies on remarkable properties of the nonlinearity of \eqref{eq:constraint}, was originally proved in \cite{CziRod} using the theory of \emph{null} (or \emph{characteristic}) \emph{data gluing}, pioneered by Aretakis \cite{Are1, Are2} and Aretakis--Czimek--Rodnianski \cite{AreCziRod1, AreCziRod2, AreCziRod3} (see also \cite{ChrCon, KehUng1, KehUng2}).
Our direct proof, independent of null gluing, optimizes the positivity, regularity, size and spatial decay requirements in obstruction-free gluing for asymptotically flat initial data; see Remarks~\ref{rem:obs-free-unit} and~\ref{rem:obs-free-af}. It is based on the first main theorem (Theorem~\ref{thm:gluing-annulus}), as well as the construction of multi-bump initial data sets with prescribed charges and support properties that relies on several ingredients:
(1)~conic-type solution operators for the linearized constraint equations around the flat case as in \cite{MaoTao}, which can be used to construct initial data sets localized in conic regions (cf.~ Carlotto--Schoen \cite{CarSch} as well as \cite{AreCziRod1, AreCziRod2, AreCziRod3}); 
(2)~a \emph{nonlinear} computation similar to that of Bartnik in \cite{Bar} concerning the positivity of mass in the time-symmetric, almost flat case; and
(3)~a localized boost argument based on the fundamental theorem of Choquet-Bruhat \cite{ChBr} and the computation of Chru\'{s}ciel \cite{Chr1} (see also \cite{ChrDel}). A more detailed sketch of the proof is in Section~\ref{subsec:obs-free-overview} below, which can be read after Sections~\ref{sec:intro} and \ref{sec:div}. 

\begin{remark}[Solution operators to divergence-type equations with prescribed support]
As indicated above, the approach in this paper is based on the use of solution operators for the linearized constraint equations around the flat case whose support properties are prescribed (in a bounded set or in a cone); see Lemmas~\ref{lem:bogovskii-0} and \ref{lem:conic} below. In \cite{MaoTao}, a similar approach was employed to construct of localized initial data sets \`a la Carlotto--Schoen \cite{CarSch}. In fact, construction of such solution operators may be extended to other divergence-type equations with variable coefficients and/or on different manifolds; this will be addressed in an upcoming work \cite{IMOT}.
\end{remark}

\subsection{Notation and conventions}
Before stating the main results, we introduce some notation.
\begin{itemize}
\item As usual, $A \aleq B$ is the shorthand for $\abs{A} \leq C B$ for some $C > 0$ (implicit constant), which may differ from line to line. We write $A \aeq B$ if $A \aleq B$ and $B \aleq A$.
\item $\bbN$ is the set of (positive) natural integers, $\bbN_{0} = \bbN \cup \set{0}$.
\item We denote by $(x^{1}, x^{2}, x^{3})$ the rectangular coordinates on $\bbR^{3}$. We equip $\bbR^{3}$ with the Euclidean metric, given by the identity matrix $\dlt_{ij}$ with respect to $(x^{1}, x^{2}, x^{3})$. We use latin indices $i, j,\ldots$ for tensors on $\bbR^{3}$; all latin indices are raised and lowered using $\dlt$. Also, $\tr_{\dlt} T = \sum_{i} T_{ii}$.
\item We denote by $(x^{0}, x^{1}, x^{2}, x^{3})$ the rectangular coordinates on $\bbR^{1+3}$. We equip $\bbR^{1+3}$ with the Minkowski metric, given by the diagonal matrix $\bfeta_{\mu \nu}$ with entries $-1, +1, +1, +1$ with respect to $(x^{0}, x^{1}, x^{2}, x^{3})$. We use greek indices $\mu, \nu, \lmb, \ldots$ for tensors on $\bbR^{1+3}$; all greek indices are raised and lowered using $\bfeta$. Also, $\tr_{\bfeta} \bfT = \bfeta^{\mu \nu} \bfT_{\mu \nu}$.
\item We define $B_{r}(\xi)$ to be the ball of radius $r$ centered at $\xi$, $A_{r}(\xi) = B_{2r}(\xi) \setminus \br{B_{r}(\xi)}$ for an annulus of radii $\aeq r$ and $\td{A}_{r}(\xi) = B_{4r}(\xi) \setminus \br{B_{\frac{r}{2}}(\xi)}$ for an enlargement of $A_{r}$.  When $\xi = 0$, we shall omit $(\xi)$ and write $B_{r} = B_{r}(0)$, $A_{r} = A_{r}(0)$ etc. We introduce two pieces of notation for cones in $\bbR^{3}$: first, given $\omg \subseteq \bbS^{2}$, we define $C_{\omg} = \set{x \in \bbR^{3} : \tfrac{x}{r} \in \omg}$, and second, given $\tht \in (0, \pi)$ and $\bfomg \in \bbS^{2}$, we define $C_{\tht}(\bfomg) = \set{x \in \bbR^{3} : \angle(x, \bfomg) < \tht}$. We write $\Omg^{c}$ for the complement of $\Omg$ in $\bbR^{3}$, $-\Omg = \set{x \in \bbR^{3} : - x \in \Omg}$, and $\Omg + \xi = \set{x + \xi \in \bbR^{3} : x \in \Omg}$.
\item For $F_{ij}$ defined on a symmetric subset $\Omg$ of $\bbR^{3}$ (i.e., $\Omg = - \Omg$), we denote its decomposition into the even and odd parts by $F_{ij} = F_{ij}^{+} + F_{ij}^{-}$, where $F^{\pm}(x) = \frac{1}{2}(F(x) \pm F(-x))$.
\item  We introduce $\bfe_{i} = \rd_{i}$, $\bfnu = \frac{x^{j}}{\abs{x}} \bfe_{j}$ (the outward unit normal to $\rd B_{r}$) and $\bfY_{i} = \tensor{\eps}{_{i j}^{k}} x^{j} \bfe_{k}$ (the infinitesimal generator of rotation about the $x^{i}$-axis), where $\eps_{ijk}$ is the Levi-Civita symbol.
\item Given a Banach space $X = X(\bbR^{3})$ of (possibly vector-valued) functions on $\bbR^{3}$ and an open subset $\Omg \subseteq \bbR^{3}$, we define the space of extendible functions on $\Omg$ as $X(\Omg) = X(\bbR^{3}) / \set{u \in X(\bbR^{3}) : \left. u \right|_{\Omg} = 0}$, equipped with the $\nrm{u}_{X(\Omg)} = \inf_{\td{u} \in X(\bbR^{3}) : \left. \td{u} \right|_{\Omg} = u} \nrm{\td{u}}_{X}$. We also define the space of functions supported in $\Omg$, denoted by $X_{0}(\Omg)$, to be the completion of $C^{\infty}_{c}(\Omg)$ with respect to  $\nrm{\cdot}_{X(\Omg)}$.
\end{itemize}
\subsection{Scaling invariance and charges in the almost flat regime}
In this paper, we shall consider solutions $(g, k)$ to \eqref{eq:constraint} equipped with global coordinates $x^{1}, x^{2}, x^{3}$ -- or more concretely, defined on a subset $\Omg$ of $\bbR^{3}$. For any such initial data $(g, k)$ and any $r > 0$, define $(g^{(r)}, k^{(r)})$ by
\begin{equation} \label{eq:scaling}
	(g, k) \mapsto (g^{(r)}(x), k^{(r)}(x)) := (g(r x), r k(r x)).
\end{equation}
Observe that \eqref{eq:constraint} is invariant under \eqref{eq:scaling}, which we shall call the \emph{invariant scaling transformation}.

We also introduce the following charges of $(g, k)$ measured on the sphere $\rd B_{r}$, which are conserved (i.e., independent of $r$) if $(g, k)$ solves the linearization of \eqref{eq:constraint} around $(\dlt, 0)$ (see Lemma~\ref{lem:charge} below):
\begin{align}
\bfE[(g, k); \rd B_{r}] &= \frac{1}{2} \int_{\rd B_{r}} \sum_{i} \left( \rd_{i} g_{ij} - \rd_{j} g_{ii} \right) \bfnu^{j}\, \ud S , \label{eq:charge-E} \\
\bfP_{i}[(g, k); \rd B_{r}] &= \int_{\rd B_{r}} \left( k_{ij} -  \dlt_{ij} \tr_{\dlt} k \right) \bfnu^{j}\, \ud S, \label{eq:charge-P} \\
\bfC_{k} [(g, k); \rd B_{r}] &= \frac{1}{2} \int_{\rd B_{r}} \sum_{i} \left( x_{k} \rd_{i} g_{ij} - x_{k} \rd_{j} g_{ii} - \dlt_{ik} (g_{ij} - \dlt_{ij}) + \dlt_{jk} (g_{ii} - \dlt_{ii})   \right) \bfnu^{j}\, \ud S, \label{eq:charge-C} \\
\bfJ_{k} [(g, k); \rd B_{r}] &= \int_{\rd B_{r}} \left( k_{ij} -  \dlt_{ij} \tr_{\dlt} k \right) Y^{i}_{k} \bfnu^{j}\, \ud S. \label{eq:charge-J}
\end{align}
We put these together to form a $10$-vector,
\begin{equation*}
	\bfQ[(g, k); \rd B_{r}] = (\bfE, \bfP_{1}, \bfP_{2}, \bfP_{3}, \bfC_{1}, \bfC_{2}, \bfC_{3}, \bfJ_{1}, \bfJ_{2}, \bfJ_{3})^{\dagger} [(g, k); \rd B_{r}].
\end{equation*}
Taking (formally) the limit as $r \to \infty$ leads to the usual ADM energy, linear momenta, center of mass and angular momenta under a suitable asymptotic flatness assumption (see Definition~\ref{def:af} and Lemma~\ref{lem:af} below). For these quantities, we introduce the notation
\begin{equation*}
	\bfQ^{ADM}[(g, k)] = (\bfE^{ADM}, \ldots, \bfJ^{ADM}_{3})^{\dagger} [(g, k)] := (\lim_{r \to \infty} \bfE[(g, k); \rd B_{r}], \ldots, \lim_{r \to \infty} \bfJ_{3}[(g, k); \rd B_{r}])^{\dagger}.
\end{equation*}
Since we will be working with annuli, it is also convenient to introduce the following (smoothly) \emph{averaged charges}. Fix $\eta \in C^{\infty}_{c}(0, \infty)$ such that $\supp \eta \subseteq (1, 2)$, $\int \eta(r') \, \ud r' = 1$ and $\eta_{r}(r') := r^{-1} \eta(r^{-1} r')$. We define
\begin{align}
\bfQ[(g, k); A_{r}] = (\bfE, \ldots, \bfJ_{3})[(g, k); A_{r}] = \int \eta_{r}(r')(\bfE, \ldots, \bfJ_{3})[(g, k); \rd B_{r'}] \, \ud r'. \label{eq:charge-avg}
\end{align}
In accordance with this notation, we shall denote the components of $Q \in \bbR^{10}$ by
\begin{equation*}
	Q = (\bfE(Q), \bfP_{1}(Q), \bfP_{2}(Q), \bfP_{3}(Q), \bfC_{1}(Q), \bfC_{2}(Q), \bfC_{3}(Q), \bfJ_{1}(Q), \bfJ_{2}(Q), \bfJ_{3}(Q))^{\dagger}.
\end{equation*}
These quantities are conserved for solutions to the linearization of \eqref{eq:constraint} around $(\dlt, 0)$; see Lemma~\ref{lem:charge}. They are \emph{not} invariant under \eqref{eq:scaling}, but transform as follows:
\begin{equation*}
	(\bfE, \bfP)[(g^{(r)}, k^{(r)}); A_{R_{0}}] = r^{-1} (\bfE, \bfP)[(g, k); A_{r R_{0}}], \quad
	(\bfC, \bfJ)[(g^{(r)}, k^{(r)}); A_{R_{0}}] = r^{-2} (\bfC, \bfJ)[(g, k); A_{r R_{0}}],
\end{equation*}
and similarly for $\bfQ[(g, k); \rd B_{r}]$.

\subsection{Gluing up to linear obstructions} \label{subsec:obs-gluing}
Suppose that we are given two disjoint dyadic annuli $A_{R_{in}}$ and $A_{R_{out}}$ (i.e., $2 R_{in} < R_{out}$) and initial data sets $(g_{in}, k_{in})$ and $(g_{out}, k_{out})$ on the annuli $A_{R_{in}}$ and $A_{R_{out}}$, respectively, that are almost flat (i.e., close to $(\dlt, 0)$ in some suitable sense). Consider the following gluing problem:~{\it find an almost flat initial data set $(g, k)$ on $B_{2 R_{out}} \setminus \br{B_{R_{in}}}$ that agrees with $(g_{in}, k_{in})$ and $(g_{out}, k_{out})$ on $A_{R_{in}}$ and $A_{R_{out}}$, respectively.}

If all the initial data sets solve not \eqref{eq:constraint} but rather its linearization around $(\dlt, 0)$, then an obvious necessary condition for the solvability of this problem is $\bfQ[(g_{in}, k_{in}); A_{R_{in}}] = \bfQ[(g_{out}, k_{out}); A_{R_{out}}]$, in view of the conservation laws (see Lemma~\ref{lem:charge} below). In fact, this condition turns out to be sufficient as well\footnote{Sufficiency of $\bfQ[(g_{in}, k_{in}); A_{R_{in}}] = \bfQ[(g_{out}, k_{out}); A_{R_{out}}]$ for linearized gluing can be established using Lemma~\ref{lem:bogovskii} along with simple extension procedures; we leave the details to the interested reader.}; in this sense, (the failure of) the condition $\bfQ[(g_{in}, k_{in}); A_{R_{in}}] = \bfQ[(g_{out}, k_{out}); A_{R_{out}}]$ is precisely the obstruction for linearized initial data gluing.

Our first main result generalizes the above linear considerations to the nonlinear setting.
To state it, we first need to formulate the notion of \emph{admissible initial data sets} (cf.~\cite{CorSch}) on an annulus.

\begin{definition}[$\calQ$-admissible initial data sets] \label{def:adm-extr}
Let a bounded open subset $\calQ \subseteq \bbR^{10}$ and $s \in \bbR$ be given.
We say that a $10$-parameter family $\set{(g_{Q}, k_{Q})}_{\calQ}$ is a \emph{$\calQ$-admissible family of annular initial data sets} on $\td{A}_{r}$ with Sobolev regularity $s$ if:
\begin{enumerate}
\item for each $Q \in \calQ$, $(g_{Q}, k_{Q}) \in (H^{1} \cap C^{0}) \times L^{2} (\td{A}_{r})$ and solves \eqref{eq:constraint} in $\td{A}_{r}$;
\item for each $Q \in \calQ$, $(g_{Q}, k_{Q}) \in H^{s} \times H^{s-1} (\td{A}_{r})$;
\item the map $\calQ \to H^{s}\times H^{s-1}(\td{A}_{r})$ defined by $Q \mapsto (g_{Q}, k_{Q})$ is Lipschitz;
\item for each $Q  \in \calQ \subseteq \bbR^{10}$, we have $Q = \bfQ[(g_{Q}, k_{Q}); A_{r}]$.
\end{enumerate}
\end{definition}

\begin{theorem} [Gluing up to linear obstructions at unit scale]\label{thm:gluing-annulus}
For each $s > \frac{3}{2}$, there exist $\eps_{c} = \eps_{c}(s) > 0$ and $M_{c} = M_{c}(s) > 0$ such that the following holds. Let $(\mathring{g}, \mathring{k})$ be a solution to \eqref{eq:constraint} in $H^{s} \times H^{s-1}(\td{A}_{1})$ such that
\begin{equation} \label{eq:gluing-epss}
\nrm{(\mathring{g} - \dlt, \mathring{k})}_{H^{s} \times H^{s-1}(\td{A}_{1})} \leq \eps.
\end{equation}
Define $\mathring{Q} \in \bbR^{10}$ by $\mathring{Q} = \bfQ[(\mathring{g}, \mathring{k}); A_{1}]$. Let $\calQ \subseteq \bbR^{10}$ be a bounded open set such that
\begin{equation} \label{eq:gluing-q}
	\mathring{Q} \in \calQ, \quad B_{M_{c} \eps^{2}}(\mathring{Q}) \subseteq \calQ.
\end{equation}
and consider an $\calQ$-admissible family of annular initial data sets  $\set{(g_{Q}, k_{Q})}_{Q \in \calQ}$ on $\td{A}_{1}$ with Sobolev regularity $s$ such that, for all $Q, Q' \in \calQ$,
\begin{align}
	\nrm{(g_{Q} - \dlt, k_{Q})}_{H^{s} \times H^{s-1}(\td{A}_{1})} &\leq \eps, \label{eq:gluing-adm-epss} \\
	\nrm{(g_{Q} - g_{Q'}, k_{Q} - k_{Q'})}_{H^{s} \times H^{s-1}(\td{A}_{1})} &\leq K \abs{Q-Q'}. \label{eq:gluing-adm-lip}
\end{align}
Then the following holds:
\begin{enumerate}
\item {\bf Exterior gluing.} If $\eps < \eps_{c}$ and $K \eps < \eps_{c}$, then there exists $(g, k) \in H^{s} \times H^{s-1}(\td{A}_{1})$ and $Q \in \calQ$ such that $\nrm{(g-\dlt, k)}_{H^{s} \times H^{s-1}(\td{A}_{1})} \aleq \eps$, $\abs{Q - \mathring{Q}} \leq M_{c} \eps^{2}$, and
\begin{equation*}
	(g, k) = \begin{cases}
	(\mathring{g}, \mathring{k}) & \hbox{ in } A_{\frac{1}{2}}, \\
	(g_{Q}, k_{Q}) & \hbox{ in } A_{2}.
	\end{cases}
\end{equation*}

\item {\bf Interior gluing.} If $\eps < \eps_{c}$ and $K \eps < \eps_{c}$, then there exists $(g, k) \in H^{s} \times H^{s-1}(\td{A}_{1})$ and $Q \in \calQ$ such that $\nrm{(g-\dlt, k)}_{H^{s} \times H^{s-1}(\td{A}_{1})} \aleq \eps$, $\abs{Q - \mathring{Q}} \leq M_{c} \eps^{2}$, and
\begin{equation*}
	(g, k) = \begin{cases}
	(g_{Q}, k_{Q}) & \hbox{ in } A_{\frac{1}{2}}, \\
	(\mathring{g}, \mathring{k}) & \hbox{ in } A_{2}.
	\end{cases}
\end{equation*}
\end{enumerate}

In both cases, the following additional statements hold as well:
\begin{itemize}
\item (Lipschitz continuity) The map $(\mathring{g}, \mathring{k}) \mapsto (g, k, Q)$ is Lipschitz as a map from the subset of $H^{s} \times H^{s-1}$ restricted by \eqref{eq:gluing-epss}--\eqref{eq:gluing-q} into $H^{s} \times H^{s-1} \times \bbR^{10}$.

\item (Persistence of regularity) If $(\mathring{g}, \mathring{k}) \in H^{s+m} \times H^{s+m-1}(\td{A}_{1})$ and the family $\set{(g_{Q}, k_{Q})}_{Q \in \calQ}$ is of Sobolev regularity $s+m$ for $m \in \bbN$, then $(g, k) \in H^{s+m} \times H^{s+m-1}(\td{A}_{1})$ and
\begin{equation*}
\begin{aligned}
	\nrm{(g-\dlt, k)}_{H^{s+m} \times H^{s+m-1}(\td{A}_{1})} &\leq C \nrm{(\mathring{g}-\dlt, \mathring{k})}_{H^{s+m} \times H^{s+m-1}(\td{A}_{1})}\\
	&\peq + C \nrm{(g_{Q}-\dlt, k_{Q})}_{H^{s+m} \times H^{s+m-1}(\td{A}_{1})},
\end{aligned}\end{equation*}
where $C = C(s, m, \nrm{(\mathring{g}-\dlt, \mathring{k})}_{H^{s+1} \times H^{s}(\td{A}_{1})}, \nrm{(g_{Q}-\dlt, k_{Q})}_{H^{s+1} \times H^{s}(\td{A}_{1})})$.
\end{itemize}
\end{theorem}
Theorem~\ref{thm:gluing-annulus} is proved in Section~\ref{sec:gluing}.

\begin{remark}
\begin{enumerate}[leftmargin=*]
\item In view of the fact that $\dot{H}^{\frac{3}{2}} \times \dot{H}^{\frac{1}{2}}$ is invariant under \eqref{eq:scaling}, the regularity requirement $s > \frac{3}{2}$ is optimal modulo the endpoint. Moreover, using different Moser and product estimates on $\bbR^{3}$ in lieu Lemma~\ref{lem:moser-product} below, it is not difficult to extend the result to the case when $H^{s}$ is replaced by $B^{\frac{3}{2}, 2}_{1}$ (scaling invariant Besov space) or $W^{s, p}$ with $s > \frac{3}{p}$ and $1 < p < 3$.
\item We formulated Theorem~\ref{thm:gluing-annulus} for initial data sets on $\td{A}_{1}$ for technical convenience. Nonetheless, by the extension lemmas proved below (Lemmas~\ref{lem:ext-outgoing} and \ref{lem:ext-ingoing}), the theorem may be immediately reformulated in terms of initial data in two annuli $A_{\frac{1}{2}}$ and $A_{2}$ with averaged charges measured in the respective annuli. As alluded to before, this formulation provides a sufficient condition for solving the gluing problem with $R_{in} = \frac{1}{2}$ and $R_{out} = 2$ analogous to the linear problem. The result may also be easily extended to any other pair of concentric annuli whose closures are disjoint.
\end{enumerate}
\end{remark}

Using Theorem~\ref{thm:gluing-annulus}, we may retrieve the celebrated gluing result of Corvino--Schoen \cite{CorSch} (see also Chru\'{s}ciel--Delay \cite{ChrDel} and Chru\'{s}ciel--Corvino--Isenberg \cite{ChrCorIse}) for asymptotically flat initial data sets, under (essentially) optimal assumptions on the regularity and the spatial decay. We adopt the following definition of asymptotic flatness:
\begin{definition}\label{def:af}
Let $(g, k) \in H^{s}_{loc} \times H^{s-1}_{loc}(B_{R_{0}}^{c})$ for $s > \frac{3}{2}$ and $R_{0} \in 2^{\bbN}$. We say that $(g, k)$ is \emph{$\alp$-asymptotically flat} if, for some $\alp \in \bbR$ and $D_{0} > 0$, we have
\begin{equation} \label{eq:w-af}
	\nrm{(g-\dlt, k)}_{\dot{\calH}^{s} \times \dot{\calH}^{s-1}(\td{A}_{r})} \leq D_{0} r^{-s+\frac{3}{2}-\alp} \quad \hbox{ for all } r \in 2^{\bbN} \cap B_{R_{0}}^{c},
\end{equation}
where\footnote{When $s \in \bbN$, observe that $\nrm{u}_{\dot{\calH}^{s}(\td{A}_{R})} \aeq \nrm{\rd^{(s)}_{x} u}_{L^{2}(\td{A}_{R})} + R^{-s} \nrm{u}_{L^{2}(\td{A}_{R})}$. } $\nrm{u}_{\dot{\calH}^{s}(\td{A}_{R})} := R^{-s+\frac{3}{2}} \nrm{u (R x) }_{H^{s}(\td{A}_{1})}$.
\end{definition}

By the Sobolev embedding, $\alp$ corresponds to the pointwise decay rate of $g - \dlt$ in $\abs{x}^{-1}$.

\begin{theorem}[Corvino--Schoen \cite{CorSch} and Chru\'{s}ciel--Delay \cite{ChrDel}; see also \cite{ChrCorIse}] \label{thm:corvino-schoen}
Let $s > \frac{3}{2}$ and $\alp > \frac{1}{2}$. Given any $\alp$-asymptotically flat initial data $(g_{in}, k_{in}) \in H_{loc}^{s} \times H_{loc}^{s-1} (B_{R_{0}}^{c})$ solving \eqref{eq:constraint} and satisfying $\abs{\bfE^{ADM}[(g, k)]} > \abs{\bfP^{ADM}[(g, k)]}$, there exist $(g, k) \in H_{loc}^{s} \times H_{loc}^{s-1}(B_{R_{0}}^{c})$ solving \eqref{eq:constraint} and $r \geq R_{0}$ such that $(g, k) = (g_{in}, k_{in})$ on $B_{r}$ and $(g, k)$ equals an initial data set for one of the Kerr spacetimes on $B_{2r}^{c}$.
\end{theorem}

Theorem~\ref{thm:corvino-schoen} is proved in Section~\ref{sec:af}. We remark that $\alp > \frac{1}{2}$ is the sharp (up to endpoint) threshold for the ADM energy-momentum to be well-defined \cite{Bar}; see also Lemma~\ref{lem:af} for sufficiency.

 \subsection{Obstruction-free gluing} \label{subsec:obs-free-results}

We return to the problem of gluing two initial data sets $(A_{R_{in}}, g_{in}, k_{in})$ and $(A_{R_{out}}, g_{out}, k_{out})$ that are almost flat. Recall from the discussion in Section~\ref{subsec:obs-gluing} that the \emph{linearized} gluing problem is solvable (if and) only if the $10$ charges $\bfQ$ are identical. Remarkably, the original \emph{nonlinear} problem turns out to be solvable if these identities among charges (obstructions) are replaced by mere positivity conditions! This novel type of gluing -- \emph{obstruction-free gluing} -- has been recently introduced and established by Czimek--Rodnianski\footnote{We note that, instead of annular data, \cite{CziRod} work with data on two spheres.} \cite{CziRod}.

We provide a new purely spacelike proof of obstruction-free gluing (the original approach of \cite{CziRod} involves null gluing), and furthermore sharpen the positivity, regularity and size requirements -- see Remark~\ref{rem:obs-free-unit} below. Our main result for almost flat data on two annuli reads as follows.

\begin{theorem}[Obstruction-free gluing for almost flat data on annuli] \label{thm:obs-free-unit}
Given $s > \frac{3}{2}$ and $\Gmm > 1$, there exist $\eps_{o} = \eps_{o}(s, {\Gmm}) > 0$, $\mu_{o} = \mu_{o}(s, \Gmm) > 0$ and $C_{o} = C_{o}(s, \Gmm) > 0$ such that the following holds. Let $(g_{in}, k_{in}) \in H^{s} \times H^{s-1} (A_{1})$, $(g_{out}, k_{out}) \in H^{s} \times H^{s-1} (A_{32})$ be pairs solving \eqref{eq:constraint}. Define $\Dlt Q = (\Dlt \bfE, \ldots, \Dlt \bfJ_{3}) \in \bbR^{10}$ by
\begin{equation} \label{eq:Dlt-Q}
	\Dlt Q = \bfQ[(g_{out}, k_{out}); A_{32}] - \bfQ[(g_{in}, k_{in}; A_{1})],
\end{equation}
and assume that
\begin{gather}
	\Dlt \bfE > \abs{\Dlt \bfP}, \label{eq:obs-free-unit:EP} \\
	\tfrac{\Dlt \bfE}{\sqrt{(\Dlt \bfE)^{2} - \abs{\Dlt \bfP}^{2}}} < \Gmm, \label{eq:obs-free-unit:gmm} \\
	\Dlt \bfE < \eps_{o}^{2}, \label{eq:obs-free-unit:eps} \\
	\abs{\Dlt \bfC} + \abs{\Dlt \bfJ} < \mu_{o} \Dlt \bfE, \label{eq:obs-free-unit:CJ}\\
	\nrm{(g_{in} - \dlt, k_{in})}_{H^{s} \times H^{s-1}(A_{1})}^{2} + \nrm{(g_{out} - \dlt, k_{out})}_{H^{s} \times H^{s-1}(A_{32})}^{2} < \mu_{o} \Dlt \bfE \label{eq:obs-free-unit:data}.
\end{gather}
Then there exists $(g, k) \in H^{s} \times H^{s-1}(B_{64} \setminus \br{B_{1}})$ solving \eqref{eq:constraint} such that
\begin{gather}
(g, k) = (g_{in}, k_{in}) \quad \hbox{in } A_{1}, \quad (g, k) = (g_{out}, k_{out}) \quad \hbox{in } A_{32}, \label{eq:obs-free-gluing} \\
\nrm{(g - \dlt, k)}_{H^{s} \times H^{s-1}(B_{64} \setminus \br{B_{1}})}^{2} < C_{o} \Dlt \bfE. \label{eq:obs-free-conc}
\end{gather}
\end{theorem}

\begin{remark} [Sharpness of positivity, regularity and size assumptions] \label{rem:obs-free-unit}
As observed in \cite{CziRod}, the positivity requirement $\Dlt \bfE > \abs{\Dlt \bfP}$ is a key necessary condition for the validity of obstruction-free gluing in general (see also Remark~\ref{rem:obs-free-af} below). In Theorem~\ref{thm:obs-free-unit}, $\abs{\Dlt \bfP}$ is allowed to be arbitrarily close to $\Dlt \bfE$ (i.e., $\Gmm$ is arbitrarily large in view of \eqref{eq:obs-free-unit:gmm}) provided that the bounds \eqref{eq:obs-free-unit:eps}--\eqref{eq:obs-free-unit:data} hold with sufficiently small $\eps_{o}$ and $\mu_{0}$ (depending on $\Gmm$).

As before, the regularity requirement $s > \frac{3}{2}$ is sharp (up to endpoint) in relation to the scaling critical exponent. In view of the conservation law for $\bfE$ (see Lemma~\ref{lem:charge} below), for a solution $(g, k) \in H^{s} \times H^{s-1} (B_{64 \Gmm} \setminus \br{B_{1}})$ ($s > \frac{3}{2}$) satisfying \eqref{eq:constraint}, \eqref{eq:Dlt-Q} and \eqref{eq:obs-free-gluing}, it is necessary that
\begin{equation*}
	\Dlt \bfE \aleq_{s, \Gmm} \nrm{(g - \dlt, k)}_{H^{s} \times H^{s-1}(B_{64 \Gmm} \setminus \br{B_{1}})}^{2}.
\end{equation*}
Thus \eqref{eq:obs-free-conc} is sharp up to a constant, and hence so are the assumptions \eqref{eq:obs-free-unit:CJ} and \eqref{eq:obs-free-unit:data}. Finally, we remark that the choice of the radii $R_{in} = 1$ and $R_{out} = 32$ is arbitrary, and the result may be easily extended to any other pair of concentric annuli whose closures are disjoint.
\end{remark}

\begin{remark}
By a minor modification of the proof of Theorem~\ref{thm:obs-free-unit}, the following statements may also be established (where the domains are omitted for simplicity):
\begin{itemize}
\item (Persistence of regularity) If $(g_{in}, k_{in}, g_{out}, k_{out}) \in (H^{s+m} \times H^{s+m-1}) \times (H^{s+m} \times H^{s+m-1})$ for $m \in \bbN$, then $(g, k) \in H^{s+m} \times H^{s+m-1}$,

\item (Lipschitz continuity) The correspondence $(g_{in}, k_{in}, g_{out}, k_{out}) \mapsto (g, k)$ in Theorem~\ref{thm:obs-free-unit} may be put together to define a locally Lipschitz map from the subset of $(H^{s} \times H^{s-1}) \times (H^{s} \times H^{s-1})$ restricted by \eqref{eq:obs-free-unit:EP}--\eqref{eq:obs-free-unit:data} into $H^{s} \times H^{s-1}$.
\end{itemize}
\end{remark}

The corresponding result for asymptotically flat initial data sets has a more elegant hypothesis, at the expense of performing the gluing procedure in an annulus sufficiently afar (as in Theorem~\ref{thm:corvino-schoen}).

\begin{theorem} [Obstruction-free gluing for asymptotically flat data] \label{thm:obs-free-af}
Let $s > \frac{3}{2}$ and $\alp > \frac{1}{2}$. Let $(g_{in}, k_{in})$, $(g_{out}, k_{out}) \in H_{loc}^{s} \times H_{loc}^{s-1} (B_{R_{0}}^{c})$ be $\alp$-asymptotically flat pairs solving \eqref{eq:constraint}. Define, for those components that are well-defined,
\begin{equation*}
	\Dlt Q := \bfQ^{ADM}[(g_{out}, k_{out})] - \bfQ^{ADM}[(g_{in}, k_{in})],
\end{equation*}
and assume that
\begin{equation*}
	\Dlt \bfE  > \abs{\Dlt \bfP}.
\end{equation*}
Then there exists an asymptotically flat pair $(g, k) \in H_{loc}^{s} \times H_{loc}^{s-1}(B^{c}_{R_{0}})$ solving \eqref{eq:constraint} and $r \geq R_{0}$ such that $(g, k) = (g_{in}, k_{in})$ on $B_{2 r}$ and $(g, k) = (g_{out}, k_{out})$ on $B_{32 r}^{c}$.
\end{theorem}
Theorems~\ref{thm:obs-free-unit} and~\ref{thm:obs-free-af} are proved in Section~\ref{sec:obs-free}.
\begin{remark}\label{rem:obs-free-af}
Taking $(g_{in}, k_{in}) = (\dlt, 0)$ shows that the condition $\Dlt \bfE > \abs{\Dlt \bfP}$ is sharp in view of the positive mass theorem \cite{SchYau, Wit}. Note that this case recovers the interior gluing of Bieri--Chru\'{s}ciel \cite{BieChr}.

The requirements on the regularity and spatial decay exponents $s > \frac{3}{2}$ and $\alp > \frac{1}{2}$ are sharp (up to endpoints) as discussed in Section~\ref{subsec:obs-gluing}; in particular, $\Dlt \bfE$ and $\Dlt \bfP$ are well-defined. On the contrary, $\bfC^{ADM}, \bfJ^{ADM}$ need not be well-defined for either $(g_{in}, k_{in})$ or $(g_{out}, k_{out})$ to apply Theorem~\ref{thm:obs-free-af}.
\end{remark}

\subsection{Structure of the paper}
The remainder of the paper is structured as follows. {\bf Section~\ref{sec:div}} collects some preliminary facts concerning the linearization of \eqref{eq:constraint} and \eqref{eq:eve} around the flat case. More specifically, after rewriting \eqref{eq:constraint} as a quasilinear perturbation of divergence-type equations on $(\bbR^{3}, \dlt)$ (i.e., the linearization of \eqref{eq:constraint} around the flat case) in {\bf Section~\ref{subsec:linearize}}, we write down Bogovskii- and conic-type operators for these equations (see also \cite{IMOT}) in {\bf Sections~\ref{subsec:bogovskii}} and {\bf \ref{subsec:conic}}, respectively. We discuss the conservation laws for the linearization of \eqref{eq:constraint} (which involve $\bfQ$) in {\bf Section~\ref{subsec:charge}}, and for the linearization of \eqref{eq:eve} (which will be used in the proof of obstruction free gluing) in {\bf Section~\ref{subsec:charge-spt}}. Then in {\bf Section~\ref{sec:gluing}}, we prove Theorem~\ref{thm:gluing-annulus}. In {\bf Section~\ref{sec:af}}, we collect some basic facts about asymptotic flatness and establish Theorem~\ref{thm:corvino-schoen}, with some technical details concerning Kerr initial data sets deferred to {\bf Appendix~\ref{sec:kerr-extr}}.
Finally, in {\bf Section~\ref{sec:obs-free}}, we prove Theorems~\ref{thm:obs-free-unit} and~\ref{thm:obs-free-af}; an outline of the proof is provided in {\bf Section~\ref{subsec:obs-free-overview}}.

\subsection*{Acknowledgements}
The authors would like to thank Igor Rodnianski for enlightening discussions regarding obstruction-free gluing.
Y.~Mao was partially supported by a National Science Foundation CAREER Grant under NSF-DMS-1945615. S.-J.~Oh was partially supported by a Sloan Research Fellowship and a National Science Foundation CAREER Grant under NSF-DMS-1945615. Z. Tao was partially supported by the NSF grant DMS-1952939 and by the Simons Targeted Grant Award No. 896630.

\section{Linearization around the flat case} \label{sec:div}
\subsection{Schematic notation and reformulation of \eqref{eq:constraint}} \label{subsec:linearize}
We begin by reformulating \eqref{eq:constraint} in a sufficiently flat region. We introduce new variables $(h, \pi)$, defined as follows:
\begin{align}
(h_{ij}, \pi_{ij}) &= (g_{ij} - \dlt_{ij} - \dlt_{ij} \tr_{\dlt} (g - \dlt), k_{ij} - \dlt_{ij} \tr_{\dlt} k) .\label{eq:hpi-def}
\end{align}
Observe that the transformation is obviously invertible with the formulae
\begin{align}
(g_{ij}, k_{ij}) &= (\dlt_{ij} + h_{ij} - \frac{1}{2} \dlt_{ij} \tr_{\dlt} h, \pi_{ij} - \frac{1}{2} \dlt_{ij} \tr_{\dlt} \pi). \label{eq:hpi2gk}
\end{align}
With respect to the new variables, we write the left-hand sides of \eqref{eq:constraint} schematically as
\begin{align}
	R[g] &= \rd_{i} \rd_{j} h^{ij} - M_{h}^{(2)}(h, \rd^{2} h) - M_{h}^{(1)}(\rd h, \rd h), \\
	(\tr_{g} k)^{2} - \abs{k}_{g}^{2} &= - M_{h}^{(0)}(\pi, \pi), \\
	g^{jj'} (g^{ii'}D_{g; i} k_{i'j'} - \rd_{j'} \tr_{g} k )&= \rd_{i} \pi^{ij} - N_{h}^{(1) j}(h, \rd \pi) - N_{h}^{(0) j}(\rd h, \pi).
\end{align}
The schematic notation we use is defined as follows:
\begin{definition}[Schematic notation] \label{def:sch-notation}
Each of $M^{(\ast)}_{h}(u_{1}, \ldots, u_{k})$ and $N^{(\ast)}_{h}(u_{1}, \ldots, u_{k})$ is a linear combination of contraction of $u_{1}, \ldots, u_{k}$ with a smooth tensor field $A$ (of the appropriate rank) on $\bbR^{3}$ that depends only on $h$. Moreover, we assume that  $\abs{\rd_{h}^{(n)} A (h)} \aleq_{M, n} 1$ as long as $\abs{h} \leq M$ for all $n \in \bbN \cup \set{0}$.
\end{definition}

In conclusion, \eqref{eq:constraint} takes the form
\begin{align}
	\rd_{i} \rd_{j} h^{ij} &= M_{h}^{(2)}(h, \rd^{2} h) + M_{h}^{(1)}(\rd h, \rd h) + M_{h}^{(0)}(\pi, \pi), \label{eq:constraint-h} \\
	\rd_{i} \pi^{ij} &= N_{h}^{(1) j}(h, \rd \pi) + N_{h}^{(0) j}(\rd h, \pi), \label{eq:constraint-pi}
\end{align}
where the indices are raised and lowered using the Euclidean metric $\dlt$. Introducing
\begin{equation} \label{eq:constraint-abbrev}
	\vec{P}(h, \pi) = (\rd_{i} \rd_{j} h^{ij}, \rd_{i} \pi^{ij}), \quad
	\vec{N}(h, \pi) = (M(h, \pi), N(h, \pi)) := (\hbox{RHS of }\eqref{eq:constraint-h}, \hbox{RHS of }\eqref{eq:constraint-pi}),
\end{equation}
we may abbreviate \eqref{eq:constraint-h}, \eqref{eq:constraint-pi} as $\vec{P}(h, \pi) = \vec{N}(h, \pi)$.

\subsection{Bogovskii-type operators} \label{subsec:bogovskii}
We now state our main tool for inverting the LHS of \eqref{eq:constraint-h}--\eqref{eq:constraint-pi} while preserving the annular support property.
\begin{lemma} \label{lem:bogovskii}
Let $\Omg = A_{1}$. There exists a linear operator $S$ defined for $f \in C_{c}^{\infty}(\Omg)$ and taking values in symmetric $2$-tensor fields (i.e., symmetric $3 \times 3$ matrix-valued functions) such that
\begin{enumerate}[label=$(S\arabic*)$]
\item $\supp S f \subseteq \Omg$ (recall $f \in C_{c}^{\infty}(\Omg)$);
\item $\rd_{i} \rd_{j} (S f)^{ij} = f$ if $\int_{\Omg} f \cdot (1, x_{1}, x_{2}, x_{3})^{\dagger} \, \ud x = 0$;
\item $\nrm{S f}_{H^{s'}} \aleq_{s'} \nrm{f}_{H^{s'-2}}$ for any $s' \in \bbR$;
\item $\nrm{[S, \rd_{j}] f}_{H^{s'}} \aleq_{s'} \nrm{f}_{H^{s'-2}}$ for any $s' \in \bbR$ and $j = 1, \ldots, d$.
\end{enumerate}
Moreover, there exists a linear operator $T$ defined for $\bff \in C_{c}^{\infty}(\Omg; \bbR^{3})$ and taking values in symmetric $2$-tensor fields such that
\begin{enumerate}[label=$(T\arabic*)$]
\item $\supp T \bff \subseteq \Omg$ (recall $\bff \in C_{c}^{\infty}(\Omg; \bbR^{3})$);
\item $\rd_{i} (T \bff)^{ij} = \bff^{j}$ if $\int_{\Omg} \bff \cdot  (\bfe_{1}, \bfe_{2}, \bfe_{3}, \bfY_{1}, \bfY_{2}, \bfY_{3})^{\dagger} \, \ud x = 0$;
\item $\nrm{T \bff}_{H^{s'}} \aleq_{s'} \nrm{\bff}_{H^{s'-1}}$ for any $s' \in \bbR$;
\item $\nrm{[T, \rd_{j}] \bff}_{H^{s'}} \aleq_{s'} \nrm{\bff}_{H^{s'-1}}$ for any $s' \in \bbR$ and $j = 1, \ldots, d$.
\end{enumerate}
\end{lemma}
Concerning the integral conditions in $(S2)$ (resp.~$(T2)$), observe that $1$, $x_{1}$, $x_{2}$, $x_{3}$ span the kernel of the formal $L^{2}$-adjoint of the double divergence operator $h \mapsto \rd_{i} \rd_{j} h^{ij}$ (resp.~$\bfe_{1}$, $\ldots$, $\bfY_{3}$ span the kernel of the formal $L^{2}$-adjoint of the symmetric divergence operator $\pi \mapsto \rd_{i} \pi^{ij}$).

The basic ingredient for the proof of Lemma~\ref{lem:bogovskii} is the following explicit operators in the case $\Omg$ is star-shaped with respect to a ball (see Lemma~\ref{lem:bogovskii-0} for the definition), which is analogous to the classical Bogovskii operator \cite{Bog} for the divergence operator $u \mapsto \rd_{j} u^{j}$.

\begin{lemma} [Bogovskii-type operators] \label{lem:bogovskii-0}
Let $\Omg$ be an open set in $\bbR^{3}$ that is star-shaped with respect to a ball $B \subset \Omg$, i.e., for every $x \in U$ and $y \in B$, the line segment $[\br{x, y}]$ is contained in $\Omg$. Fix $C^{\infty}_{c}(B)$ such that $\int \eta \, \ud x= 1$, and $\chi_{\Omg} \in C^{\infty}_{c}(\bbR^{3})$ such that $\chi_{\Omg} = 1$ on $\Omg$.

Then the operator defined for all $f \in C^{\infty}_{c}(\bbR^{3})$ by
\begin{equation} \label{eq:d-div-eq-bogovskii}
(S f)^{ij}(x) = \int \bfPsi_{\eta}^{ij}(x, y) f(y) \, \ud y,
\end{equation}
\begin{equation} \label{eq:d-div-eq-bogovskii-ker}
\bfPsi_{\eta}^{ij} (z+y, y)
= \left(\int_{\abs{z}}^{\infty} \eta\left( r \frac{z}{\abs{z}} + y\right) r^{2} \, \ud r \right) \frac{z^{i} z^{j}}{\abs{z}^{3}}.
\end{equation}
satisfies $(S1)$ (when $\supp f \subseteq \Omg$) and $(S2)$. Moreover, $f \mapsto S (\chi_{\Omg} f)$ is a classical pseudodifferential operator of order $-2$, which also implies $(S3)$ and $(S4)$ when $\supp f \subseteq \Omg$.

Moreover, the operator defined by
\begin{equation} \label{eq:symm-div-eq-bogovskii}
	(T \bff)^{ij} (x) = \int (\bfPsi_{\eta})^{ij}_{k}(x, y) \bff^{k}(y) \, \ud y,
\end{equation}
\begin{equation} \label{eq:symm-div-eq-bogovskii-ker}
\begin{aligned}
(\bfPsi_{\eta})^{ij}_{k}(z+y, y)
&= - \frac{1}{2} \left( \int_{\abs{z}}^{\infty} \eta\left( r \frac{z}{\abs{z}}  + y \right)  r^{2} \, \ud r \right)\abs{z}^{-3} \left( z^{i} \bfdlt^{j}_{k} + \bfdlt^{i}_{k} z^{j} \right) \\
&\peq + \frac{1}{2} \rd_{z^{m}}\left( \left( \int_{\abs{z}}^{\infty} \eta\left( r \frac{z}{\abs{z}}  + y \right)  r^{2} \, \ud r \right)\abs{z}^{-3} \left( z^{i} \bfdlt^{j}_{k} z^{m} + \bfdlt^{i}_{k} z^{j} z^{m} \right)  \right)  \\
&\peq - \rd_{z^{k}} \left( \left( \int_{\abs{z}}^{\infty} \eta\left(r \frac{z}{\abs{z}} + y \right) r^{2} \, \ud r \right)\abs{z}^{-3}  z^{i} z^{j} \right).
\end{aligned}
\end{equation}
satisfies $(T1)$ (when $\supp \bff \subseteq \Omg$) and $(T2)$. Moreover, $\bff \mapsto T (\chi_{\Omg} \bff)$ is a classical pseudodifferential operator of order $-1$, which also implies $(T3)$ and $(T4)$ when $\supp \bff \subseteq \Omg$.
\end{lemma}

\begin{proof}
That $(S1)$ and $(T1)$ hold may be verified by computation using the explicit formulae. To prove $(S1)$, we first recall the computation $\rd_{z^{i}} \left[ \frac{z^{i}}{\abs{z}^{3}} \int_{\abs{z}}^{\infty} \eta ( r \frac{z}{\abs{z}} + y) r^{2} \, \ud r \right]  = \dlt_{0}(z) - \eta(z+y)$,
where the expression inside the parentheses is the classical Bogovskii operator for the divergence operator \cite{Bog}. We may then compute
\begin{align*}
	\rd_{z^{i}} \rd_{z^{j}} \bfPsi_{\eta}(z+y, y)
	&= \rd_{z^{i}} \left[ \rd_{z^{j}} \left( \frac{z^{i} z^{j}}{\abs{z}^{3}}\right) \int_{\abs{z}}^{\infty} \eta \left( r \frac{z}{\abs{z}} + y\right) r^{2} \, \ud r
	+ \frac{z^{i} z^{j}}{\abs{z}^{3}} \rd_{z^{j}} \int_{\abs{z}}^{\infty} \eta\left( r \frac{z}{\abs{z}} + y\right) r^{2} \, \ud r  \right]\\
	&= \rd_{z^{i}} \left[ \frac{z^{i}}{\abs{z}^{3}} \int_{\abs{z}}^{\infty} \eta \left( r \frac{z}{\abs{z}} + y\right) r^{2} \, \ud r \right]
	- \rd_{z^{i}} \left[ z^{i} \eta\left( z + y\right)  \right]\\
	&= \dlt_{0}(z) - 4 \eta(z+y) - z^{i} (\rd_{i} \eta)(z + y),
\end{align*}
where we used the preceding identity on the last line. From this computation, $(S2)$ follows.
On the other hand, a proof of $(T2)$ can be found in \cite{IseOh}, where the same operator was introduced and used (alternatively, $(T2)$ can also be verified by computation as above). Next, following an argument similar to \cite{OhTat}, it may be verified that $f \mapsto S (\chi_{\Omg} f)$ and $\bff \mapsto T (\chi_{\Omg} \bff)$ are pseudodifferential operators of order $-2$ and $-1$, respectively. Then $(S3)$, $(S4)$, $(T3)$ and $(T4)$ follow.  \qedhere
\end{proof}

\begin{proof}[Proof of Lemma~\ref{lem:bogovskii}]
We shall call $S$ (resp.~$T$) satisfying $(S1)$--$(S4)$ (resp.~$(T1)$--$(T4)$) a \emph{Bogovskii-type} operator on $\Omg$. Lemma~\ref{lem:bogovskii-0} says that such operators exist for any open set star-shaped with respect to a ball. We claim that a Bogovskii-type operator can be defined on the union of any two open sets on each of which such an operator is defined. By a simple recursion argument, this observation would allow us to define a Bogovskii-type operator on any finite union of such sets. Since $A_{1}$ is clearly a finite union of open star-shaped sets with respect to balls, Lemma~\ref{lem:bogovskii} would then follow.

It remains to verify the claim. We focus on the case of $S$, the case of $T$ being similar. Let $U_{1}$ and $U_{2}$ be open sets on which Bogovskii-type operators $S_{1}$ and $S_{2}$, respectively, exist. When $U_{1} \cap U_{2} \neq \0$, let $\set{\chi_{1}, \chi_{2}}$ be a smooth partition of unity on $U$ subordinate to $\set{U_{1}, U_{2}}$ and fix $\eta \in C^{\infty}_{c}(U_{1} \cap U_{2})$. Write $g_{0} = 1$, $g_{j} = x_{j}$ ($j=1, 2, 3$) and define $G_{\mu \nu} = \int g_{\mu} g_{\nu} \eta^{2} \, \ud x$. If $\eta \neq 0$, note that $G_{\mu \nu}$ is positive-definite and in particular invertible. Now, we define $\tht^{\mu} = (G^{-1})^{\mu \nu} g_{\nu} \eta^{2}$, so that $\int \tht^{\mu} g_{\mu'} \, \ud x = \dlt^{\mu}_{\mu'}$ and $\tht^{\mu} \in C^{\infty}_{c}(U_{1} \cap U_{2})$.
We decompose $f$ into
\begin{align*}
	f &= f_{1} + f_{2} + \left(\int f \, \ud x\right) \, \tht^{0} -\sum_{j=1}^{3} \left(\int f x_{j} \, \ud x\right) \tht^{j},
\end{align*}
where $f_{k} = f  \chi_{k} - \left(\int f \chi_{k} \, \ud x\right) \tht^{0} -\sum_{j=1}^{3} \left(\int f \chi_{k} x_{j} \, \ud x\right) \tht^{j}$. Then $f = f_{1} + f_{2}$ if and only if $\int f (1, x_{1}, x_{2}, x_{3})^{\dagger} \, \ud x = 0$. Moreover, $\supp f_{k} \subseteq U_{k}$ with $\int f_{k} (1, x_{1}, x_{2}, x_{3})^{\dagger} \, \ud x = 0$. We claim that the following defines a desired Bogovskii-type operator on $U$:
\begin{equation*}
	S f := S_{1} f_{1} + S_{2} f_{2}.
\end{equation*}
That $(S1)$--$(S3)$ hold is straightforward. To check $(S4)$, note that
\begin{equation} \label{eq:bogovskii-comm}
	[S, \rd_{j}] f = [S_{1}, \rd_{j}] f_{1} + [S_{2}, \rd_{j}] f_{2} + S_{1}((\rd_{j} f)_{1} - \rd_{j} f_{1}) + S_{2}((\rd_{j} f)_{2} - \rd_{j} f_{2}).
\end{equation}
That the $H^{s'}$ norm of the first two terms on the RHS is bounded by $\nrm{f}_{H^{s'-2}}$ is clear by $(S4)$ for $S_{k}$ and the definition of $f_{k}$. To estimate the $H^{s'}$ norm of the third term, in view of $(S3)$ for $S_{1}$, we write
\begin{equation*}
	\nrm{S_{1}((\rd_{j} f)_{1} - \rd_{j} f_{1})}_{H^{s'}}
	\aleq \nrm{(\rd_{j} f)_{1} - \rd_{j} f_{1}}_{H^{s'-2}}
	\aleq \nrm{\rd_{j} f \chi_{1} - \rd_{j} (f \chi_{1})}_{H^{s'-2}} + \nrm{f}_{H^{s'-2}},
\end{equation*}
where, for the last inequality, we used the observation that all terms involving $\tht^{\mu}$ may be bounded by $\nrm{f}_{H^{s'-2}}$. But $\nrm{\rd_{j} f \chi_{1} - \rd_{j} (f \chi_{1})}_{H^{s'-2}} = \nrm{f \rd_{j} \chi_{1}}_{H^{s'-2}} \aleq\nrm{f}_{H^{s'-2}}$, which is acceptable. The fourth term in \eqref{eq:bogovskii-comm} is handled similarly. \qedhere

\qedhere
\end{proof}

\subsection{Conic operators} \label{subsec:conic}
Here we state the main tool for inverting the LHS of \eqref{eq:constraint-h}--\eqref{eq:constraint-pi} while preserving the conic support property. We first define the relevant weighted Sobolev space.
\begin{definition} \label{def:Hb}
For $s \in \bbN_0$, the $\bm$-Sobolev space $\Hb^{s}(\bbR^{3})$ is defined by the norm
\begin{align*}
    \|u\|_{\Hb^s}^2:=\sum\limits_{k\leq s}\|\langle x\rangle^k \nabla^k u\|_{L^2(\bbR^3)}^2.
\end{align*}
We extend the definition to $s\in\bbR$ by duality and (complex) interpolation. We further define for $\delta\in\bbR$, $\Hb^{s,\delta}:=\langle x\rangle^{-\delta}\Hb^s$.
\end{definition}

In \cite{MaoTao} (see also \cite{OhTat}), the following was proved.
\begin{lemma} \label{lem:conic}
Let $\omg \subseteq \bbS^{2}$ be a convex open subset, and consider $\kpp \in C^{\infty}(\bbS^{2})$ with $\int \kpp = 1$ and $\supp \kpp \subseteq \omg$. Consider the linear translation-invariant operator $S_{c}$ defined for $f \in C_{c}^{\infty}(C_{\omg})$ and taking values in symmetric $2$-tensor fields (i.e., symmetric $3 \times 3$ matrix-valued functions) given by
\begin{equation*}
(S_{c} f)^{ij} = \int K_{\kpp}^{ij}(x - y) f(y) \, \ud y, \quad K_{\kpp}^{ij}(z) = \kpp(\tfrac{z}{\abs{z}}) \frac{ z^{i} z^{j}}{\abs{z}^{3}}.
\end{equation*}
Then $S_{c}$ satisfies
\begin{enumerate}[label=$(S_{c}\arabic*)$]
\item $\supp S_{c} f \subseteq C_{\omg}$;
\item $\rd_{i} \rd_{j} (S_{c} f)^{ij} = f$;
\item $\nrm{S f}_{H^{s', \dlt}_{\bm}} \aleq_{s', \dlt} \nrm{f}_{H^{s'-2, \dlt+2}_{\bm}}$ for any $s' \in \bbR$ and $\dlt < -\frac{1}{2}$.
\end{enumerate}
Moreover, consider the linear translation-invariant operator $T_{c}$ defined for $\bff \in C_{c}^{\infty}(C_{\omg}; \bbR^{3})$ and taking values in symmetric $2$-tensor fields given by
\begin{equation*}
T_{c} \bff = \int (K_{\kpp})^{ij}_{k}(x - y) \bff^{k} (y) \, \ud y, \quad (K_{\kpp})^{ij}_{k}(z) = \rd_{m} \left(\kpp(\tfrac{z}{\abs{z}}) \frac{z^{i} \bfdlt^{j}_{k} z^{m} + \bfdlt^{i}_{k} z^{j} z^{m} - z^{i} z^{j} \bfdlt^{m}_{k}}{\abs{z}^{3}}\right).
\end{equation*}
Then $T_{c}$ satisfies
\begin{enumerate}[label=$(T_{c}\arabic*)$]
\item $\supp T_{c} \bff \subseteq C_{\omg}$;
\item $\rd_{i} (T_{c} \bff)^{ij} = \bff^{j}$;
\item $\nrm{T_{c} \bff}_{H^{s', \dlt}_{\bm}} \aleq_{s'} \nrm{\bff}_{H^{s'-1, \dlt+1}_{\bm}}$ for any $s' \in \bbR$ and $\dlt < \frac{1}{2}$.
\end{enumerate}
\end{lemma}

\begin{remark}
    The definition of $K_{\kappa}$ is not exactly the same as \cite{MaoTao}, but the boundedness assertion $(S_c3)$ for the solution operator was proved for any homogeneous distribution with outgoing property, which clearly applies here. The fact $(S_c1)$ is obvious from the defintion and convexity of the cone and $(S_c2)$ can be checked by direct computation:
    \begin{align*}
        &\partial_i\partial_j \left(\kappa\left(\frac{z}{|z|}\right)\frac{z^iz^j}{|z|^3}\right)=\partial_i\left(\partial_j\kappa\left(\frac{z}{|z|}\right)\frac{z^iz^j}{|z|^3}+\kappa\left(\frac{z}{|z|}\right)\partial_j\left(\frac{z^iz^j}{|z|^3}\right)\right)=\partial_i\left(\kappa\left(\frac{z}{|z|}\right)\frac{z^i}{|z|^3}\right)=\delta_0.
    \end{align*}
\end{remark}

\subsection{Conservation laws for the linearized constraint equation} \label{subsec:charge}
We now state a tool for controlling the variation of charges $\bfQ$ (both averaged or unaveraged). For $0 < r_{0} < r_{1}$, and $\eta_{r}$ is as in \eqref{eq:charge-avg}, we introduce
\begin{equation} \label{eq:chi}
\chi_{r_{0}, r_{1}} (r)= \int_{-\infty}^{r} \left( \eta_{r_{0}}(r') - \eta_{r_{1}}(r') \right) \, \ud r'.
\end{equation}
\begin{lemma} \label{lem:charge}
Let $(g, k)$ and $(h, \pi)$ be related by \eqref{eq:hpi-def}. For any $0 < r_{0} < r_{1}$, we have
\begin{equation*}
	\int_{B_{r_{1}} \setminus B_{r_{0}}} \tfrac{1}{2} \rd_{i} \rd_{j} h^{ij} \begin{pmatrix} 1 \\ x_{1} \\ x_{2} \\ x_{3} \end{pmatrix} \, \ud x = \begin{pmatrix} \bfE \\ \bfC_{1} \\ \bfC_{2} \\ \bfC_{3} \end{pmatrix}[(g, k); \rd B_{r_{1}}]
	- \begin{pmatrix} \bfE \\ \bfC_{1} \\ \bfC_{2} \\ \bfC_{3} \end{pmatrix}[(g, k); \rd B_{r_{0}}],
\end{equation*}
\begin{equation*}
	\int_{B_{r_{1}} \setminus B_{r_{0}}} \sum_{j} \rd_{i} \pi^{ij} \begin{pmatrix} \bfe_{1}^{j} \\ \bfe_{2}^{j} \\ \bfe_{3}^{j} \\ \bfY_{1}^{j} \\ \bfY_{2}^{j} \\ \bfY_{3}^{j} \end{pmatrix} \, \ud x
	=
	\begin{pmatrix} \bfP_{1} \\ \bfP_{2} \\ \bfP_{3} \\ \bfJ_{1} \\ \bfJ_{2} \\ \bfJ_{3} \end{pmatrix}[(g, k); \rd B_{r_{1}}]
	- \begin{pmatrix} \bfP_{1} \\ \bfP_{2} \\ \bfP_{3} \\ \bfJ_{1} \\ \bfJ_{2} \\ \bfJ_{3} \end{pmatrix}[(g, k); \rd B_{r_{0}}],
\end{equation*}
and in terms of averaged charges, we have
\begin{equation*}
	\int \chi_{r_{0}, r_{1}}(r) \tfrac{1}{2}  \rd_{i} \rd_{j} h^{ij} \begin{pmatrix} 1 \\ x_{1} \\ x_{2} \\ x_{3} \end{pmatrix} \, \ud x = \begin{pmatrix} \bfE \\ \bfC_{1} \\ \bfC_{2} \\ \bfC_{3} \end{pmatrix}[(g, k); A_{r_{1}}]
	- \begin{pmatrix} \bfE \\ \bfC_{1} \\ \bfC_{2} \\ \bfC_{3} \end{pmatrix}[(g, k); A_{r_{0}}],
\end{equation*}
\begin{equation*}
	\int \sum_{j}  \chi_{r_{0}, r_{1}}(r) \rd_{i} \pi^{ij} \begin{pmatrix} \bfe_{1}^{j} \\ \bfe_{2}^{j} \\ \bfe_{3}^{j} \\ \bfY_{1}^{j} \\ \bfY_{2}^{j} \\ \bfY_{3}^{j} \end{pmatrix} \, \ud x
	=
	\begin{pmatrix} \bfP_{1} \\ \bfP_{2} \\ \bfP_{3} \\ \bfJ_{1} \\ \bfJ_{2} \\ \bfJ_{3} \end{pmatrix}[(g, k); A_{r_{1}}]
	- \begin{pmatrix} \bfP_{1} \\ \bfP_{2} \\ \bfP_{3} \\ \bfJ_{1} \\ \bfJ_{2} \\ \bfJ_{3} \end{pmatrix}[(g, k); A_{r_{0}}].
\end{equation*}
\end{lemma}
This lemma follows immediately by integration by parts, \eqref{eq:hpi-def} and the fact that $1, x_{1}, \ldots, x_{3}$ (resp.~$\bfe_{1}, \ldots, \bfY_{3}$) belong to the kernel of the $L^{2}$-adjoint of $h \mapsto \rd_{i} \rd_{j} h^{ij}$ (resp.~$\pi \mapsto \rd_{i} \pi^{ij}$).

\subsection{Conservation laws for the linearized Einstein vacuum equations} \label{subsec:charge-spt}

In this subsection, we work in an oriented spacetime domain $\bfOmg$ equipped with a flat (i.e., Minkowski) metric $\bfeta$. We use greek indices $\alp, \bt, \ldots$ to refer to spacetime tensorial indices. We always raise and lower indices using the flat metric $\bfeta$, and denote by $\nb$ the (trivial) covariant derivative with respect to $\bfeta$.

Consider the Einstein tensor $\bfG[\bfg] = \Ric[\bfg] - \frac{1}{2} \bfg \tr_{\bfg} \Ric[\bfg]$. Its linearization around the flat metric $\bfeta$, which trivially satisfies $\bfG[\bfeta] = 0$, takes the form
\begin{equation} \label{eq:lin-bfG}
	\uD_{\bfeta} \bfG [\dot{\bfg}]_{\alp \bt} = \frac{1}{2} \left( - \nb^{\gmm} \nb_{\gmm} \bfH_{\alp \bt} + \nb_{\alp} \nb^{\gmm} \bfH_{\gmm \bt} + \nb_{\bt} \nb^{\gmm} \bfH_{\gmm \alp} - \bfeta_{\alp \bt} \nb^{\gmm} \nb^{\dlt} \bfH_{\gmm \dlt} \right),
\end{equation}
where $\bfH = \dot{\bfg} - \frac{1}{2} \bfeta \tr_{\bfeta} \dot{\bfg}$. Note that, by linearizing the second Bianchi identity $\bfD^{\bt} \bfG[\bfg]_{\alp \bt} = 0$, we have $\nb^{\bt} \uD_{\bfeta} \bfG[\dot{\bfg}]_{\alp \bt} = 0$.



Let $\bfX$ be a Killing vector field with respect to $\bfeta$, i.e., $(\calL_{\bfX} \bfeta)_{\alp \bt} = \nb_{\alp} \bfX_{\bt} + \nb_{\bt} \bfX_{\alp} = 0$. By the linearized second Bianchi identity, observe that $\nb^{\bt} (\uD_{\bfeta} \bfG[\dot{\bfg}]_{\alp \bt} \bfX^{\alp}) = 0$; hence, the $1$-form $\uD_{\bfeta} \bfG[\dot{\bfg}]_{\alp \bt} \bfX^{\alp}$ is co-closed. In fact, there exists $2$-form whose co-differential is this $1$-form; following \cite[Appendix~E]{ChrDel} (see also \cite{Chr1}), we introduce 
\begin{equation} \label{eq:U-X}
	{}^{(\bfX)}\bbU_{\alp \bt}[\dot{\bfg}]
	= \frac{1}{2} \left[ \left(- \nb_{\alp} \bfH_{\gmm \bt} + \nb_{\bt} \bfH_{\gmm \alp}    + \bfeta_{\gmm \alp} \nb^{\dlt} \bfH_{\bt \dlt} - \bfeta_{\gmm \bt} \nb^{\dlt} \bfH_{\alp \dlt} \right)\bfX^{\gmm} + \bfH_{\gmm \alp} \nb^{\gmm} \bfX_{\bt} - \bfH_{\gmm \bt} \nb^{\gmm} \bfX_{\alp} \right].
\end{equation}
By a straightforward computation, it may be verified that
\begin{equation} \label{eq:U-X-div}
	\nb^{\alp} ({}^{(\bfX)} \bbU_{\alp \bt}[\dot{\bfg}]) = \uD_{\bfeta} \bfG [\dot{\bfg}]_{\alp \bt} \bfX^{\alp}.
\end{equation}
Equivalently, $\ud (\star {}^{(\bfX)} \bbU[\dot{\bfg}]) = - \star \uD_{\bfeta} \bfG[\dot{\bfg}](\bfX, \cdot)$ using the Hodge star operator $\star$ associated to $\bfeta$. By the Stokes theorem,
\begin{equation} \label{eq:U-X-stokes}
	\int_{\rd U} \star {}^{(\bfX)} \bbU[\dot{\bfg}] =  \int_{U} \ud (\star {}^{(\bfX)} \bbU[\dot{\bfg}]) = - \int_{U} \star \uD_{\bfeta} \bfG[\dot{\bfg}](\bfX, \cdot).
\end{equation}


Consider a system of coordinates $\set{x^{0}, \ldots, x^{3}}$ with respect to which $\bfeta = - (\ud x^{0})^{2} + (\ud x^{1})^{2} + (\ud x^{2})^{2} + (\ud x^{3})^{2}$ and $\ud x^{0} \wedge \ud x^{1} \wedge \ud x^{2} \wedge \ud x^{3}$ has positive orientation -- such coordinates shall be referred to as \emph{canonical}. On $\Sgm_{0} = \set{x^{0} = 0}$, define
\begin{equation} \label{eq:induced-data-lin}
\dot{g}_{i j} = \dot{\bfg}_{ij} \bb|_{\set{x^{0} = 0}}, \quad
\dot{k}_{i j} = \frac{1}{2} \left( \rd_{0} \dot{\bfg}_{i j} - \rd_{i} \dot{\bfg}_{j 0} - \rd_{j} \dot{\bfg}_{i 0} \right) \bb|_{\set{x^{0} = 0}}.
\end{equation}
It may be checked that (cf.~\eqref{eq:hpi-def} and \eqref{eq:constraint-abbrev})
\begin{equation} \label{eq:lin-constraint}
	\uD_{\bfeta}\bfG[\dot{\bfg}]_{0 0} =  \tfrac{1}{2} \rd_{j} \rd_{k} (\dot{g}^{jk} - \dlt^{jk} \tr_{\dlt} \dot{g}) , \quad
	\uD_{\bfeta}\bfG[\dot{\bfg}]_{0 j} =  \rd^{\ell} (\dot{k}_{j \ell} - \dlt_{j \ell} \tr_{\dlt} \dot{k}_{j \ell}).
\end{equation}
Indeed, these may be proved by linearizing the nonlinear relations $\bfG[\bfg]_{00} = \frac{1}{2} (R[\br{g}] - (\tr_{\br{g}} \br{k})^{2} + \abs{\br{k}}_{\br{g}}^{2})$, $\bfG[\bfg]_{0j} = \br{D}^{\ell} \br{k}_{j \ell} - \rd_{j} \tr \br{k}$ on $\Sgm_{0}$, where $\br{g}$ is the induced metric, $\br{D}$ is the induced connection and $\br{k}$ is the second fundamental form on $\Sgm_{0}$ (see \eqref{eq:induced-data} below for the formulae for $(\br{g}, \br{k})$).

The space of Killing vector fields with respect to $\bfeta$ are spanned by $\rd_{x^{\mu}}$ and $x_{\mu} \rd_{x^{\nu}} - x_{\nu} \rd_{x^{\mu}}$ ($\mu, \nu = 0, 1, 2, 3$). For $r > 0$, we define the associated charges by
\begin{equation} \label{eq:charge-spt}
\begin{aligned}
	\bbP_{\mu}[\dot{\bfg}; \set{x^{0} = 0, \, \abs{x} = r}] &=  \int_{\set{x^{0} = 0, \, \abs{x} = r}} \star {}^{(\rd_{x^{\mu}})} \bbU[\dot{\bfg}], \\
	\bbJ_{\mu \nu}[\dot{\bfg}; \set{x^{0} = 0, \, \abs{x} = r}] &=  \int_{\set{x^{0} = 0, \, \abs{x} = r}} \star {}^{(x_{\mu} \rd_{x^{\nu}} - x_{\nu} \rd_{x^{\mu}})} \bbU[\dot{\bfg}],
\end{aligned}
\end{equation}
where $\set{x^{0} = 0, \, \abs{x} = r}$ is oriented with $\star (\ud x^{0} \wedge \ud \abs{x})$. We have the following relationship between $\bbP_{\mu}$, $\bbJ_{\mu \nu}$ and the charges for $(g, k)$.

\begin{lemma} \label{lem:charge-spt}
For any $r > 0$ and canonical coordinates $x^{\mu}$ on $\bfOmg$ containing $\set{x^{0} = 0, \, \abs{x} < r}$,
\begin{align*}
( \bbP_{0}, \bbP_{i}, \bbJ_{i 0}, \bbJ_{j k})[\dot{\bfg}; \set{x^{0} = 0, \, \abs{x} = r}]
 = (\bfE, \bfP_{i}, \bfC_{i}, \tensor{\eps}{^{i}_{jk}} \bfJ_{i})[(\dlt + \dot{g}, \dot{k}); \rd B_{r}],
\end{align*}
where $(\dot{g}, \dot{k})$ is given in terms of $\dot{\bfg}$ by \eqref{eq:induced-data-lin} in the coordinates $(x^{0}, x^{1}, x^{2}, x^{3})$.
\end{lemma}

Indeed, these identities may be quickly verified by comparing \eqref{eq:U-X-stokes} with $U = \set{x^{0} = 0, \, \abs{x} < r}$ and \eqref{eq:lin-constraint} with Lemma~\ref{lem:charge}.

We conclude this discussion with some formulae concerning Poincar\'e transformations. Let $\tensor{\Lmb}{^{\mu}_{\nu}} \in SO^{+}(1, 3)$, i.e., $\bfeta_{\mu \nu} \tensor{\Lmb}{^{\mu}_{\mu'}} \tensor{\Lmb}{^{\nu}_{\nu'}} = \bfeta_{\mu' \nu'}$ (isometry), $\det \Lmb = 1$ (proper) and $\tensor{\Lmb}{^{0}_{0}} > 0$ (orthochronous) and $\bfxi \in \bbR^{1+3}$. If $\set{x^{\mu}}$ is a system of canonical coordinates, then so is $\set{y^{\mu} = \tensor{\Lmb}{^{\mu}_{\mu'}} x^{\mu'} + \bfxi^{\mu}}$. Observe the following transformation laws for the Killing vector fields:
\begin{equation} \label{eq:LT-killing-vf}
\begin{aligned}
	\rd_{y^{\mu}} &= \tensor{\Lmb}{_{\mu}^{\mu'}} \rd_{x^{\mu'}}, \\
	y_{\mu} \rd_{y^{\nu}} - y_{\nu} \rd_{y^{\mu}} &= \tensor{\Lmb}{_{\mu}^{\mu'}} \tensor{\Lmb}{_{\nu}^{\nu'}} (x_{\mu'} \rd_{x^{\nu'}} - x_{\nu'} \rd_{x^{\mu'}}) + \bfxi_{\mu}\tensor{\Lmb}{_{\nu}^{\nu'}} \rd_{x^{\nu'}} - \bfxi_{\nu}\tensor{\Lmb}{_{\mu}^{\mu'}} \rd_{x^{\mu'}},
\end{aligned}\end{equation}
where $\tensor{\Lmb}{_{\mu}^{\nu}}$ (defined via index raising and lowering using $\bfeta$) is identical to $\tensor{(\Lmb^{-1})}{^{\nu}_{\mu}}$ in view of the isometry property. Using \eqref{eq:U-X-stokes}, \eqref{eq:charge-spt}, \eqref{eq:LT-killing-vf} and Lemma~\eqref{lem:charge-spt}, as well as the linearity of ${}^{(\bfX)} \bbU$ in $\bfX$, transformation properties of the charges $(\bfE, \bfP, \bfC, \bfJ)$ under isometries of $\bfeta$ may be derived.


\section{Gluing up to linear obstructions} \label{sec:gluing}
\begin{proof}[Proof of Theorem~\ref{thm:gluing-annulus}]
\noindent
{\bf Step~1.}
We begin by forming the first trial for the desired initial data set. Let $\chi(x)$ be a smooth radial function which equals $0$ for $\abs{x} < 1$ and $1$ for $\abs{x} > 2$. We introduce
\begin{equation*}
	(g_{in; Q}, k_{in; Q}, g_{out; Q}, k_{out; Q})
	= \begin{cases}
	(\mathring{g}, \mathring{k}, g_{Q}, k_{Q}) & \hbox{ for Statement~(1),} \\
	(g_{Q}, k_{Q}, \mathring{g}, \mathring{k}) & \hbox{ for Statement~(2).}
	\end{cases}
\end{equation*}
Let $(h_{in; Q}, \pi_{in; Q})$ and $(h_{out; Q}, \pi_{out; Q})$ be defined by \eqref{eq:hpi-def} with $(g, k)$ replaced by $(g_{in; Q}, k_{in; Q})$ and $(g_{out; Q}, k_{out; Q})$, respectively. We introduce the first guess for the glued initial data, namely,
\begin{equation} \label{eq:hpi-first}
	(\br{h}_{Q}, \br{\pi}_{Q})
	= (1-\chi) (h_{in; Q}, \pi_{in; Q}) + \chi (h_{out; Q}, \pi_{out; Q}).
\end{equation}
Of course, $(\br{h}_{Q}, \br{\pi}_{Q})$ would not solve the constraint equations \eqref{eq:constraint-h}--\eqref{eq:constraint-pi}; our aim is to find a correction $(\td{h}_{Q}, \td{\pi}_{Q})$ supported in $A_{1}$ and a choice of $Q \in \calQ$ such that $(h, \pi) = (\br{h}_{Q} + \td{h}_{Q}, \br{\pi}_{Q} + \td{\pi}_{Q})$ solves \eqref{eq:constraint-h}--\eqref{eq:constraint-pi}.

A quick algebraic computation shows that we want $(\td{h}_{Q}, \td{\pi}_{Q})$ to satisfy
\begin{align}
&
\begin{aligned}
	\rd_{i} \rd_{j} \td{h}_{Q}^{ij}
	&= F_{Q} + M_{\br{h}_{Q}+\td{h}_{Q}}^{(2)} (\br{h}_{Q}+\td{h}_{Q}, \rd^{2}(\br{h}_{Q}+\td{h}_{Q}))
	- M_{\br{h}_{Q}}^{(2)} (\br{h}_{Q}, \rd^{2} \br{h}_{Q}) \\
	&\peq + M_{\br{h}_{Q}+\td{h}_{Q}}^{(1)} (\rd (\br{h}_{Q}+\td{h}_{Q}), \rd (\br{h}_{Q}+\td{h}_{Q}))
	- M_{\br{h}_{Q}}^{(1)} (\rd \br{h}_{Q}, \rd \br{h}_{Q})  \\
	&\peq + M_{\br{h}_{Q}+\td{h}_{Q}}^{(0)}(\br{\pi}_{Q}+\td{\pi}_{Q}, \br{\pi}_{Q}+\td{\pi}_{Q})
	- M_{\br{h}_{Q}}^{(0)}(\br{\pi}_{Q}, \br{\pi}_{Q}),
\end{aligned}	\label{eq:tdh} \\
&
\begin{aligned}
\rd_{i} \td{\pi}_{Q}^{ij}
&= G_{Q}^{j}
+ N^{(1)j}_{\br{h}_{Q} + \td{h}_{Q}}(\br{h}_{Q} + \td{h}_{Q}, \rd (\br{\pi}_{Q} + \td{\pi}_{Q}))
- N^{(1)j}_{\br{h}_{Q}}(\br{h}_{Q}, \rd \br{\pi}_{Q}) \\
&\peq + N^{(0)j}_{\br{h}_{Q} + \td{h}_{Q}}(\rd(\br{h}_{Q} + \td{h}_{Q}), \br{\pi}_{Q} + \td{\pi}_{Q})
- N^{(0)j}_{\br{h}_{Q}}(\rd \br{h}_{Q}, \rd \br{\pi}_{Q}),
\end{aligned} \label{eq:tdpi}
\end{align}
where $F_{Q}$ and $G_{Q}$ are the errors incurred by plugging $(\br{h}_{Q}, \br{\pi}_{Q})$ into \eqref{eq:constraint-h}--\eqref{eq:constraint-pi}, i.e.,
\begin{align*}
	F_{Q} &= - \rd_{i} \rd_{j} \br{h}_{Q}^{ij} + M_{\br{h}_{Q}}^{(2)}(\br{h}_{Q}, \rd^{2} \br{h}_{Q}) + M_{\br{h}_{Q}}^{(1)}(\rd \br{h}_{Q}, \rd \br{h}_{Q}) + M_{\br{h}_{Q}}^{(0)}(\br{\pi}_{Q}, \br{\pi}_{Q}), \\
	G_{Q}^{j} &= -\rd_{i} \br{\pi}_{Q}^{ij} + N_{\br{h}_{Q}}^{(1)j}(\br{h}_{Q}, \rd \br{\pi}_{Q}) + N_{h}^{(0)j}(\rd \br{h}_{Q}, \br{\pi}_{Q}).
\end{align*}

There is, of course, considerable flexibility in specifying $(\td{h}_{Q}, \td{\pi}_{Q})$ due to the underdetermined nature of the problem, which we now fix. Let $(S, T)$ be defined as in Lemma~\ref{lem:bogovskii} with $\Omg = A_{1}$. For each $Q \in \calQ$, we look for $(\td{h}_{Q}, \td{\pi}_{Q})$ solving the fixed point problems
\begin{align}
	(\td{h}_{Q}, \td{\pi}_{Q}) = \vec{S}(\td{M}_{Q}, \td{N}_{Q}), \label{eq:tdhpi-fixed-pt}
\end{align}
where $\td{M}_{Q}$ and $\td{N}^{j}_{Q}$ are our shorthands for the RHS of \eqref{eq:tdh} and \eqref{eq:tdpi}, respectively, and $\vec{S}(F, G) = (SF, TG)$. Next, we look for $Q \in \calQ$ such that
\begin{align}
	\int \tfrac{1}{2} \td{M}_{Q} ( 1, x_{1}, x_{2}, x_{3})^{\dagger} \, \ud x &= 0, \label{eq:tdh-charge} \\
	\int \td{N}_{Q} \cdot (\bfe_{1}, \bfe_{2}, \bfe_{3}, \bfY_{1}, \bfY_{2}, \bfY_{3})^{\dagger} \, \ud x&= 0, \label{eq:tdpi-charge}
\end{align}
which would ensure that $\td{h}_{Q}$, $\td{\pi}_{Q}$ solve \eqref{eq:tdh} and \eqref{eq:tdpi} by $(S2)$, $(T2)$.

\smallskip \noindent
{\bf Step~2: Finding $(\td{h}_{Q}, \td{\pi}_{Q})$.}
We shall use the following standard estimates:
\begin{lemma} \label{lem:moser-product}
Let $u, v$ be (possibly vector-valued) Schwartz functions on $\bbR^{3}$.
\begin{enumerate}
\item {\bf Moser estimates.} Let $F$ be a $C^{\infty}$ function with bounded derivatives. Then for $s > 0$, we have
\begin{align}
\nrm{F(u) - F(0)}_{H^{s}} \aleq_{s, F, \nrm{u}_{L^{\infty}}} \nrm{u}_{H^{s}}. \label{eq:moser}
\end{align}
\item {\bf Moser difference estimates.} Let $F$ be a $C^{\infty}$ function with bounded derivatives. Then for $s > 0$, we have
\begin{align}
\nrm{F(u) - F(v)}_{H^{s}} \aleq_{s, F, \nrm{u}_{L^{\infty}}, \nrm{v}_{L^{\infty}}} \nrm{u-v}_{H^{s}} + \nrm{u-v}_{L^{\infty}} \left(\nrm{u}_{H^{s}} + \nrm{v}_{H^{s}}\right). \label{eq:moser-diff}
\end{align}

\item {\bf Product estimates.} For $s_{0}, s_{1}, s_{2}$ such that $s_{0} + s_{1} + s_{2} \geq \frac{3}{2}$, $s_{0} + s_{1} + s_{2} \geq \max\set{s_{0}, s_{1}, s_{2}}$ with at least one of the inequalities strict, we have
\begin{align}
\nrm{u \cdot v}_{H^{-s_{0}}} \aleq_{s_{0}, s_{1}, s_{2}} \nrm{u}_{H^{s_{1}}} \nrm{v}_{H^{s_{2}}}. \label{eq:product}
\end{align}
\end{enumerate}
\end{lemma}
See, for instance, \cite[Sec.~2.8]{BCD}; the results therein easily imply Lemma~\ref{lem:moser-product}. In what follows, we shall apply \eqref{eq:moser} and \eqref{eq:moser-diff} with the given $s$, and \eqref{eq:product} with $(s_{0}, s_{1}, s_{2}) = (2-s, s, s-2)$ and $(2-s, s-1, s-1)$, all of which are valid thanks to $s > \frac{3}{2}$.

For each $Q \in \calQ$, we now find $(\td{h}_{Q}, \td{\pi}_{Q})$ satisfying \eqref{eq:tdhpi-fixed-pt}. We first claim that,
\begin{equation} \label{eq:hpi-source}
	\supp (F_{Q}, G_{Q}) \subseteq A_{1}, \quad \nrm{F_{Q}}_{H^{s-2}}+\nrm{G_{Q}}_{H^{s-2}} \aleq \eps.
\end{equation}
Indeed, from \eqref{eq:hpi-first} it is clear that $(\br{h}_{Q}, \br{\pi}_{Q})$ solves \eqref{eq:constraint-h}--\eqref{eq:constraint-pi} outside $A_{1}$, from which the support property follows. Using the definition of localized norms, the support property and Lemma~\ref{lem:moser-product}, the Sobolev norm bound follows\footnote{We note that, in fact, $F_{Q}$ and $G_{Q}$ enjoy better Sobolev regularities, but this gain is useless in view of \eqref{eq:hpi-nonlin}, which is sharp.}.

Next, we claim that if $\supp (\td{h}_{Q}, \td{\pi}_{Q}) \subseteq A_{1}$ and $\nrm{(\td{h}_{Q}, \td{\pi}_{Q})}_{H^{s} \times H^{s-1}} \leq M_{c} \eps$ for $M_{c} > 0$, then
\begin{equation} \label{eq:hpi-nonlin}
	\supp (\td{M}_{Q} - F_{Q}, \td{N}_{Q} - G_{Q}) \subseteq A_{1}, \quad \nrm{\td{M}_{Q} - F_{Q}}_{H^{s-2}} + \nrm{\td{N}_{Q} - G_{Q}^{j}}_{H^{s-2}} \aleq_{M_{c}} \eps^{2}.
\end{equation}
Indeed, the support property follows from that of $(\td{h}_{Q}, \td{\pi}_{Q})$ and the structure of the terms $\td{M}_{Q} - F_{Q}$ and $\td{N}_{Q} - G_{Q}$. The Sobolev norm bound follows from the assumptions on $(\td{h}_{Q}, \td{\pi}_{Q})$ and $(\br{h}_{Q}, \br{\pi}_{Q})$, as well as Lemma~\ref{lem:moser-product}. By the same lemma, if, in addition to the previous assumptions for $(\td{h}_{Q}, \td{\pi}_{Q})$, we have $\supp (\td{h}'_{Q}, \td{\pi}'_{Q}) \subseteq A_{1}$ and $\nrm{(\td{h}'_{Q}, \td{\pi}'_{Q})}_{H^{s} \times H^{s-1}} \leq M_{c} \eps$, then
\begin{equation} \label{eq:hpi-nonlin-lip}
\begin{aligned}
	&\nrm{\td{M}_{q}[(\td{h}_{Q}, \td{\pi}_{Q})] - \td{M}_{Q}[(\td{h}'_{Q}, \td{\pi}'_{Q})]}_{H^{s-2}} + \nrm{\td{N}_{Q}[(\td{h}_{Q}, \td{\pi}_{Q})] - \td{N}_{Q}[(\td{h}'_{Q}, \td{\pi}'_{Q})]}_{H^{s-2}} \\
	&\aleq \eps \nrm{(\td{h}_{Q} - \td{h}_{Q}', \td{\pi}_{Q} - \td{\pi}_{Q}')}_{H^{s} \times H^{s-1}}.
\end{aligned}
\end{equation}

Therefore, by a standard Picard iteration argument, if $0 < \eps < \eps_{c}$ and $\eps_{c} > 0$ is sufficiently small, then for every $Q \in \calQ$ we may find a unique $(\td{h}_{Q}, \td{\pi}_{Q}) \in H^{s} \times H^{s-1}$. By the support-preserving property of $S$ and $T$, it follows that $(\td{h}_{Q}, \td{\pi}_{Q})$ vanish outside of $A_{1}$. Finally, by varying $Q$ and using a similar argument as before, one may verify that $Q \mapsto (\td{h}_{Q}, \td{\pi}_{Q}) \in H^{s} \times H^{s-1}$ is locally Lipschitz, whose details we omit.

\smallskip \noindent
{\bf Step~3: Finding $Q$.} To conclude the proof of existence, it only remains to find $Q \in \calQ$ such that \eqref{eq:tdh-charge}--\eqref{eq:tdpi-charge} hold. Since $(\br{h}_{Q}, \br{\pi}_{Q})$ and $(\td{h}_{Q}, \td{\pi}_{Q})$ are $\eps$-small on $\td{A}_{1}$, it is sensible to isolate the linear terms $-\rd_{i} \rd_{j} \br{h}_{Q}^{ij}$ and $-\rd_{i} \br{\pi}_{Q}^{ij}$ on the RHS of \eqref{eq:tdh} and \eqref{eq:tdpi}, respectively. Introduce $\chi_{\frac{1}{2}, 2}$ as in \eqref{eq:chi}.
Observe that $\supp \chi_{\frac{1}{2}, 2}(r) \subseteq \td{A}_{1}$, and $\chi_{\frac{1}{2}, 2} = 1$ on $A_{1}$, which contains the support of the RHS of \eqref{eq:tdh}. Hence, by Lemma~\ref{lem:charge}, the LHS of \eqref{eq:tdh-charge} equals
\begin{align*}
&\int - \chi_{\frac{1}{2}, 2}(r) \tfrac{1}{2}  \rd_{i} \rd_{j} \br{h}_{Q}^{ij} \begin{pmatrix} 1 \\ x_{1} \\ x_{2} \\ x_{3} \end{pmatrix} \, \ud x
+ \int_{\td{A}_{1}} \chi_{\frac{1}{2}, 2}(r) \tfrac{1}{2} \td{M}^{nonlin}_{Q} \begin{pmatrix} 1 \\ x_{1} \\ x_{2} \\ x_{3} \end{pmatrix} \, \ud x \\
&= \begin{pmatrix} \bfE \\ \bfC_{1} \\ \bfC_{2} \\ \bfC_{3} \end{pmatrix} [(g_{in; Q}, k_{in; Q}); A_{\frac{1}{2}}]
- \begin{pmatrix} \bfE \\ \bfC_{1} \\ \bfC_{2} \\ \bfC_{3} \end{pmatrix} [(g_{out; Q}, k_{out; Q}); A_{2}]
 + \int_{\td{A}_{1}} \chi_{\frac{1}{2}, 2}(r) \tfrac{1}{2} \td{M}^{nonlin}_{Q} \begin{pmatrix} 1 \\ x_{1} \\ x_{2} \\ x_{3} \end{pmatrix} \, \ud x,
\end{align*}
where $\td{M}^{nonlin}_{Q} = (\hbox{RHS of \eqref{eq:tdh}}) + \rd_{i} \rd_{j} \br{h}_{Q}^{ij}$. To compute further the first two terms on the last line, we multiply \eqref{eq:constraint-h} for $(h_{in}, \pi_{in})$ by $\chi_{\frac{1}{2}, 1}(r)$ and integrate by parts; by which we obtain
\begin{equation*}
	\begin{pmatrix} \bfE \\ \vdots \\ \bfC_{3} \end{pmatrix} [(g_{in; Q}, k_{in; Q}); A_{\frac{1}{2}}]
	=
	\begin{pmatrix} \bfE \\ \vdots \\ \bfC_{3} \end{pmatrix} [(g_{in; Q}, k_{in; Q}); A_{1}]
	+ \int_{\td{A}_{1}} \chi_{\frac{1}{2}, 1}(r) \tfrac{1}{2} M_{in; Q}^{nonlin} \begin{pmatrix} 1 \\ \vdots \\ x_{3} \end{pmatrix} \, \ud x,
\end{equation*}
where $M_{in; Q}^{nonlin} = (\hbox{RHS of \eqref{eq:constraint-h}})$ with $(h, \pi) = (h_{in; Q}, \pi_{in; Q})$. Carrying out a similar computation for $(h_{out; Q}, \pi_{out; Q})$ using $\chi_{1, 2}$ and $M_{out; Q}^{nonlin} = (\hbox{RHS of \eqref{eq:constraint-h}})$ with $(h, \pi) = (h_{out; Q}, \pi_{out; Q})$, then going back to the previous computation, we arrive at
\begin{equation} \label{eq:tdh-charge-f}
\begin{aligned}
(\hbox{LHS of \eqref{eq:tdh-charge}})
&= \begin{pmatrix} \bfE \\ \vdots \\ \bfC_{3} \end{pmatrix} [(g_{in; Q}, k_{in; Q}); A_{1}]
- \begin{pmatrix} \bfE \\ \vdots \\ \bfC_{3} \end{pmatrix} [(g_{out; Q}, k_{out; Q}); A_{1}]
& + \begin{pmatrix} n[Q]_{\bfE} \\ \vdots \\ n[Q]_{\bfC_{3}} \end{pmatrix}  \end{aligned}
\end{equation}
where
\begin{equation} \label{eq:n0-3}
\begin{pmatrix} n[Q]_{\bfE} \\ \vdots \\ n[Q]_{\bfC_{3}} \end{pmatrix}
=
\int_{\td{A}_{1}} \tfrac{1}{2} \left( \chi_{\frac{1}{2}, 2}(r) \td{M}_{Q}^{nonlin} + \chi_{\frac{1}{2}, 1}(r) M_{in; Q}^{nonlin} + \chi_{1, 2}(r) M_{out; Q}^{nonlin} \right) \begin{pmatrix} 1 \\ \vdots \\ x_{3} \end{pmatrix} \, \ud x.
\end{equation}
Similarly, working with \eqref{eq:tdpi}, we may show that
\begin{equation} \label{eq:tdpi-charge-f}
\begin{aligned}
(\hbox{LHS of \eqref{eq:tdpi-charge}})
&= \begin{pmatrix} \bfP_{1} \\ \vdots \\ \bfJ_{3} \end{pmatrix} [(g_{in; Q}, k_{in; Q}); A_{1}]
- \begin{pmatrix} \bfP_{1} \\ \vdots \\ \bfJ_{3} \end{pmatrix} [(g_{out; Q}, k_{out; Q}); A_{1}]
+\begin{pmatrix} n[Q]_{\bfP_{1}} \\ \vdots \\ n[Q]_{\bfJ_{3}} \end{pmatrix}
\end{aligned}
\end{equation}
with
\begin{equation} \label{eq:n4-9}
\begin{pmatrix} n[Q]_{\bfP_{1}} \\ \vdots \\ n[Q]_{\bfJ_{3}} \end{pmatrix}
=
\int_{\td{A}_{1}} \left( \chi_{\frac{1}{2}, 2}(r) \td{N}_{Q}^{nonlin} + \chi_{\frac{1}{2}, 1}(r) N_{in; Q}^{nonlin} + \chi_{1, 2}(r) N_{out; Q}^{nonlin} \right) \cdot  \begin{pmatrix} \bfe_{1} \\ \vdots \\ \bfY_{3} \end{pmatrix} \, \ud x,
\end{equation}
where $\td{N}^{nonlin}_{Q} = \td{N}^{nonlin}[(\br{h}_{Q}, \br{\pi}_{Q}), (\td{h}_{Q}, \td{\pi}_{Q})]$ and $N_{\Box; Q}^{nonlin} = (\hbox{RHS of \eqref{eq:constraint-pi}})$ with $(h, \pi) = (h_{\Box; Q}, \pi_{\Box; Q})$ with $\Box = in$ or $out$.

For the nonlinear contribution $n[Q]$, we claim that
\begin{align*}
\abs*{n[Q]} \aleq \eps^{2}, \quad
\nrm*{n[\cdot]}_{Lip} \aleq K \eps.
\end{align*}
We begin by recalling the expressions \eqref{eq:n0-3} and \eqref{eq:n4-9}. The terms that do not involve $\rd_{i'} \rd_{j'} h_{Q}^{ij}$ and $\rd_{i'} \pi_{Q}^{i j}$, where $(h_{Q}, \pi_{Q})$ may be $(\td{h}_{Q}, \td{\pi}_{Q})$, $(\br{h}_{Q}, \br{\pi}_{Q})$, $(h_{in; Q}, \pi_{in; Q})$, or $(h_{out; Q}, \pi_{out; Q})$, can be handled using the $L^{\infty}$ and $H^{1} \times L^{2}$ norms of $(h_{Q}, \pi_{Q})$ (which are all $\eps$-small), as well as the trivial observation that $1, x^{1}, \ldots, \bfe_{1}, \ldots, \bfY_{3}$ and their derivatives are uniformly bounded on $\td{A}_{1}$. For the terms containing such a factor, observe that this factor is always linear and we may integrate one derivative by parts off of it, after which the $L^{\infty}$ and $H^{1}$ norms of $(h_{Q}, \pi_{Q})$ again suffice. The Lipschitz bound is proved similarly, using in addition \eqref{eq:gluing-adm-lip}.

We are now ready to conclude this step. For the sake of concreteness, we focus on the case of exterior gluing from this point on; the case of interior gluing is similar. Then $\bfQ[(g_{in; Q}, k_{in; Q}); A_{1}] = \mathring{Q}$, whereas $\bfQ[(g_{out; Q}, k_{out; Q}); A_{1}] = Q$. Hence, \eqref{eq:tdh-charge-f}--\eqref{eq:tdpi-charge-f} reduce to a fixed point problem
\begin{equation} \label{eq:q-eq}
	Q = \mathring{Q} + n[Q].
\end{equation}
Moreover, we have shown that $\abs{n[Q]} \aleq \eps^{2}$ and $\nrm{n[Q]}_{Lip} \aleq K \eps$. {\bf We take $M_{c}$ to be the implicit constant in the first inequality} so that $\abs{n[Q]} \leq M_{c} \eps^{2}$. By hypothesis, $\dist(\mathring{Q}, \rd \calQ) > M_{c} \eps^{2}$. Hence, for $\eps_{c}$ sufficiently small (recall that $\eps < \eps_{c}$), the RHS of \eqref{eq:q-eq} is thus a contraction on $\set{Q : \abs{Q - \mathring{Q}} \leq M_{c} \eps^{2}} \subseteq \calQ$. It follows that a unique $Q \in \calQ$ satisfying \eqref{eq:q-eq} and $\abs{Q - \mathring{Q}} \leq M_{c} \eps^{2}$ exists by the Banach fixed point theorem.

\smallskip \noindent
{\bf Step~4: Conclusion of the proof.}
It remains to verify that, taking $\eps_{c} > 0$ smaller if necessary, the Lipschitz dependence and persistence of regularity properties hold. The Lipschitz dependence property is immediate from the Picard iteration (or Banach fixed point) schemes above; we omit the details. In case of persistence of regularity, we need to differentiate \eqref{eq:tdhpi-fixed-pt} and estimate $(\rd^{\alp} \td{h}_{Q}, \rd^{\alp} \td{\pi}_{Q})$ in $H^{s} \times H^{s-1}$ for $\abs{\alp} \leq m$. While achieving this bound is straightforward if $\eps_{c}$ is allowed to depend on $m$, we need to show that a single choice of $\eps_{c}$ works for all $m$. We sketch the necessary argument below.

From now on, we omit the subscript $Q$ since it remains fixed. We first consider a-priori bounds under the assumption that all objects are smooth. We take $\rd^{\alp}$ of \eqref{eq:tdhpi-fixed-pt} and rearrange the equations as
\begin{equation}\label{eq:persist-reg-hpi}
\begin{aligned}
	&(\rd^{\alp} \td{h}, \rd^{\alp} \td{\pi}) - \vec{S} \left( \uD_{(\td{h}, \td{\pi})} (\td{M}, \td{N})[(\rd^{\alp} \td{h}, \rd^{\alp} \td{\pi})] \right) \\ &= \vec{S}[\rd^{\alp} (F, G)] + [\rd^{\alp}, \vec{S}] (\td{M}, \td{N})(\td{h}, \td{\pi}) + (E_{\alp, \td{M}}, E_{\alp, \td{N}})[(\td{h}, \td{\pi})],
\end{aligned}
\end{equation}
where $\uD_{(\td{h}, \td{\pi})} (\td{M}, \td{N})[(\dot{h}, \dot{\pi})]$ is the linearization of $\td{M}$ around $(\td{h}, \td{\pi})$ applied to $(\dot{h}, \dot{\pi})$. (Here, $\br{h}_{ij}(x)$, $\br{\pi}_{ij}(x)$ are regarded as coefficients). Observe that we have put all highest order derivatives of $(\td{h}, \td{\pi})$ on the LHS. Proceeding as in Step~2, but also using $(S4)$ and $(T4)$, as well as Gagliardo--Nirenberg, we may show that
\begin{align*}
	\nrm{\hbox{RHS of \eqref{eq:persist-reg-hpi}}}_{H^{s} \times H^{s-1}}
	&\aleq (1+\eps)\nrm{(\br{h}, \br{\pi})}_{H^{s+m}\times H^{s+m-1}}  \\
	&\peq + \left[ \nrm{(\br{h}, \br{\pi})}_{H^{s+1}\times H^{s}}+\nrm{(\td{h}, \td{\pi})}_{H^{s+1}\times H^{s}} \right] \nrm{(\td{h}, \td{\pi})}_{H^{s+m-1}\times H^{s+m-2}},
\end{align*}
where $m = \abs{\alp}$. On the other hand, observe that the LHS of \eqref{eq:persist-reg-hpi} defines the same linear operator for $(\rd^{\alp} \td{h}, \rd^{\alp} \td{\pi})$ independent of $\alp$. Proceeding as in Step~2, for $\eps$ sufficiently small \emph{independent of $\alp$}, we obtain
\begin{align*}
	\nrm{\hbox{LHS of \eqref{eq:persist-reg-hpi}}}_{H^{s} \times H^{s-1}}
	\ageq \nrm{(\rd^{\alp} \td{h}, \rd^{\alp} \td{\pi})}_{H^{s} \times H^{s-1}}.
\end{align*}
{\bf We fix $\eps_{c} > 0$ so that this estimate holds.} At this point, it is straightforward to set up an induction scheme involving difference quotients to prove the persistence of regularity. \qedhere
\end{proof}

\section{Asymptotic flatness} \label{sec:af}
In this section, we collect some facts concerning asymptotic flat initial data sets, which will be useful in the remainder of this paper. As an application, we also establish Theorem~\ref{thm:corvino-schoen}.

\begin{lemma}[Annular restriction of asymptotically flat data] \label{lem:af}
Let $(g, k)$ be an $\alp$-asymptotically flat pair on $B_{R_{0}}^{c}$ solving \eqref{eq:constraint}, and let $r \in 2^{\bbZ}$.
Then the following holds.
\begin{enumerate}
\item The pair $(g^{(r)}, k^{(r)})$ solves \eqref{eq:constraint} on $\set{\abs{x} > r^{-1} R_{0}}$ and obeys
\begin{equation*}
\nrm{(g^{(r)} - \dlt, k^{(r)})}_{H^{s} \times H^{s-1}(\td{A}_{1})} \leq D_{0} r^{-\alp},
\end{equation*}
where $s$, $D_{0}$ and $\alp$ are from \eqref{eq:w-af}.

\item For $\alp > \frac{1}{2}$, $\bfE^{ADM}$ and $\bfP^{ADM}$ are well-defined for $(g, k)$. Moreover, for $r \geq R_{0}$ such that $D_{0} r^{-\alp} \leq 1$, we have
\begin{align}
	\abs{\bfE[(g^{(r)}, k^{(r)}); A_{1}] - r^{-1} \bfE^{ADM}} + \abs{\bfP[(g^{(r)}, k^{(r)}); A_{1}] - r^{-1} \bfP^{ADM}} &\aleq D_{0}^{2} r^{-2\alp}, \label{eq:af-ep}
\end{align}
where $\bfE^{ADM}$ and $\bfP^{ADM}$ are evaluated for $(g, k)$. For the remaining averaged charges, for $r \geq r_{0} \geq R_{0}$ such that $D_{0} r_{0}^{-\alp} \leq 1$, we have
\begin{equation}
\begin{aligned}
	\abs{(\bfC, \bfJ)[(g^{(r)}, k^{(r)}); A_{1}]} \aleq r^{-2} \abs{(\bfC^{r_{0}}, \bfJ^{r_{0}})}+ D_{0}^{2} r^{-2 \min\set{\alp, 1}} r_{0}^{2 \min\set{0, 1-\alp}}\log^{\dlt_{1}(\alp)} (\tfrac{r}{r_{0}}),
\end{aligned}
\end{equation}
where $\dlt_{1}(\alp) = 1$ when $\alp = 1$ and $0$ otherwise, and $(\bfC^{r_{0}}, \bfJ^{r_{0}}) := (\bfC, \bfJ)[(g, k); A_{r_{0}}]$.

\item Assume furthermore that $(g, k)$ also satisfies the $\alp_{-}$-\emph{parity} (or Regge--Teitelbaum) \emph{condition},
\begin{equation} \label{eq:parity}
\begin{aligned}
	\nrm{g_{ij}^{-}(x)}_{\dot{\calH}^{s}(\td{A}_{r})}
	+\nrm{\rd_{i} (g_{jk}^{-}(x))}_{\dot{\calH}^{s-1}(\td{A}_{r})}
	+\nrm{k_{ij}^{+}(x)}_{\dot{\calH}^{s-1}(\td{A}_{r})} & \leq D_{0} r^{-s+\frac{3}{2}-\alp_{-}}
\end{aligned}\end{equation}
for all $r \in 2^{\bbZ} \cap B_{R_{0}}^{c}$. If $\alp_{-} + \alp > 2$, then $\bfC_{i}^{ADM}$ and $\bfJ^{ADM}_{i}$ $(i=1, 2, 3)$ are well-defined for $(g, k)$. Moreover, for $r \geq R_{0}$ such that $D_{0} r^{-\alp} \leq 1$, we have
\begin{align*}
	\sum_{i} \abs{\bfC_{i}[(g^{(r)}, k^{(r)}); A_{1}] - r^{-2} \bfC^{ADM}_{i}} + \sum_{i} \abs{\bfJ_{i}[(g^{(r)}, k^{(r)}); A_{1}] - r^{-2} \bfJ^{ADM}_{i}} &\aleq D_{0}^{2} r^{-\alp-\alp_{-}} .
\end{align*}
where $\bfC^{ADM}_{i}$ and $\bfJ^{ADM}_{i}$ are evaluated for $(g, k)$.
\end{enumerate}
\end{lemma}
Part~(1) is trivial, Part~(2) is proved using Lemma~\ref{lem:charge} and proceeding as in Step~3 in Section~\ref{sec:gluing}, and Part~(3) follows from the observation that the contribution of $(h^{-}, \pi^{+})$ in Lemma~\ref{lem:charge} for $\bfC, \bfJ$ vanishes due to parity considerations. We omit the details as they are straightforward.

We also state the following quantitative facts concerning (exterior) Kerr initial data sets.
\begin{lemma} \label{lem:kerr-extr}
Let $\calE := \set{Q \in \bbR^{10} : \abs{\bfE(Q)} > \abs{\bf P(Q)}}$. For each $Q \in \calE$, there exists an exterior region of an initial data set for one of the Kerr spacetimes, denoted by $(g_{Q}^{Kerr}, k_{Q}^{Kerr})$, such that
\begin{align}
	\bfQ^{ADM}[(g_{Q}^{Kerr}, k_{Q}^{Kerr})] &= Q, \\
	\abs{x}^{n} \abs{\rd^{(n)}(g_{Q}^{Kerr} - \dlt)} +  \abs{x}^{n+1} \abs{\rd^{(n)} k_{Q}^{Kerr}} & \leq C_{D}(n, \gmm) \abs{\bfM} \abs{x}^{-1} , \label{eq:Kerr-af} \\
	\abs{x}^{n} \abs{\rd^{(n)} g_{Q}^{Kerr, -}} + \abs{x}^{n+1} \abs{\rd^{(n)} k_{Q}^{Kerr, +}} &\leq C_{D}(n, \gmm) \abs{\bfM} \abs{x}^{-2}, \label{eq:Kerr-parity} \\
	\abs{x}^{n} \abs{\rd^{(n)}\rd_{\bfE, \bfP} g_{Q}^{Kerr}} +  \abs{x}^{n+1} \abs{\rd^{(n)} \rd_{\bfE, \bfP} k_{Q}^{Kerr}} & \leq C_{D}(n, \gmm) \abs{x}^{-1}, \label{eq:Kerr-Lip-EP} \\
	\abs{x}^{n} \abs{\rd^{(n)}\rd_{\bfE, \bfP} g_{Q}^{Kerr, -}} +  \abs{x}^{n+1} \abs{\rd^{(n)} \rd_{\bfE, \bfP} k_{Q}^{Kerr, +}}  & \leq C_{D}(n, \gmm) (\abs{\bfM}^{-1} \abs{(\bfC, \bfJ)} + 1) \abs{x}^{-2}, \label{eq:Kerr-Lip-EP-parity} \\
	\abs{x}^{n} \abs{\rd^{(n)}\rd_{\bfC, \bfJ} g_{Q}^{Kerr}} +  \abs{x}^{n+1} \abs{\rd^{(n)} \rd_{\bfC, \bfJ} k_{Q}^{Kerr}} & \leq C_{D}(n, \gmm) \abs{x}^{-2}, \label{eq:Kerr-Lip-CJ}
\end{align}
for $\abs{x} \geq C_{R}(\gmm) (\abs{\bfM} + \abs{\bfM}^{-1} \abs{(\bfC, \bfJ)})$, where $\bfM = \sgn \bfE \sqrt{\bfE^{2} - \abs{\bfP}^{2}}$ and $\gmm = \tfrac{\abs{\bfE}}{\abs{\bfM}}$.
\end{lemma}

An elegant construction of such a family has been given in \cite{ChrDel} (via application of isometries of the background Minkowski metric to Kerr initial data sets), although the bounds \eqref{eq:Kerr-Lip-EP}--\eqref{eq:Kerr-Lip-CJ} are not explicitly established there. We sketch the proof of Lemma~\ref{lem:kerr-extr} in Appendix~\ref{sec:kerr-extr}.

\begin{proof}[Proof of Theorem~\ref{thm:corvino-schoen}]
Assume, without loss of generality, that $\frac{1}{2} < \alp < 1$. Let $(\mathring{g}, \mathring{k})= (g_{in}^{(r)}, k_{in}^{(r)})$, where $r > 0$ will be fixed at the end. By Lemma~\ref{lem:af}.(1)--(2),
\begin{equation} \label{eq:extr-r}
\begin{aligned}
	\nrm{(g_{in}^{(r)} - \dlt, k_{in}^{(r)})}_{H^{s} \times H^{s-1}(\td{A}_{1})} &\leq D_{0} r^{-\alp}, \\
	(\bfE, \bfP)[(g_{in}^{(r)}, k_{in}^{(r)}); A_{1}] &= r^{-1} (\bfE^{ADM}, \bfP^{ADM})[(g_{in}, k_{in})] + \calO(D_{0}^{2} r^{-2\alp}), \\
	\abs{(\bfC, \bfJ)[(g_{in}^{(r)}, k_{in}^{(r)}); A_{1}]} &\aleq D_{0} r^{-2} r_{0}^{2-\alp} + D_{0}^{2} r^{-2 \alp},
\end{aligned}
\end{equation}
where $r_{0} = D_{0}^{\frac{1}{\alp}}$ (so that $D_{0} r_{0}^{-\alp} = 1$).
Define $\eps = D_{0} r^{-\alp}$ and $\calQ = B_{M_{c} \eps^{2}}(\mathring{Q})$ with $\mathring{Q} = \bfQ[(g^{(r)}_{in}, k^{(r)}_{in}); A_{1}]$ so that \eqref{eq:gluing-epss} and \eqref{eq:gluing-q} trivially hold. We shall construct a $\calQ$-admissible family from Lemma~\ref{lem:kerr-extr} by rescaling and reparametrizing $Q$. Observe that, by \eqref{eq:Kerr-af}--\eqref{eq:Kerr-parity} and Lemma~\ref{lem:af}.(1)--(3),
\begin{equation} \label{eq:extr-kerr-Q}
\begin{aligned}
	\nrm{(g_{Q}^{Kerr (r)} - \dlt, k_{Q}^{Kerr (r)})}_{H^{s} \times H^{s-1}(\td{A}_{1})} &\leq C_{s, \gmm} \abs{\bfM(Q)} r^{-1}, \\
	(\bfE, \bfP)[(g_{Q}^{Kerr (r)}, k_{Q}^{Kerr (r)}); A_{1}] &= r^{-1} (\bfE, \bfP)(Q)   + \calO_{\gmm}(\abs{\bfM(Q)}^{2} r^{-2}), \\
	(\bfC, \bfJ)[(g_{Q}^{Kerr (r)}, k_{Q}^{Kerr (r)}); A_{1}] &= r^{-2} (\bfC, \bfJ)(Q) + \calO_{\gmm}(\abs{\bfM(Q)}^{2} r^{-3}),
\end{aligned}
\end{equation}
where $\gmm = \frac{\abs{\bfE(Q)}}{\abs{\bfM(Q)}}$, as long as
\begin{equation} \label{eq:extr-kerr-r}
r \geq C_{\gmm} (\abs{\bfM(Q)} + \abs{\bfM(Q)}^{-1} \abs{(\bfC, \bfJ)(Q)}),
\end{equation}
Similarly, by \eqref{eq:Kerr-af}--\eqref{eq:Kerr-Lip-CJ} and Lemma~\ref{lem:af}.(1)--(3),
\begin{align}
	&\bfQ[(g_{Q}^{Kerr (r)}, k_{Q}^{Kerr (r)}); A_{1}]
	- \bfQ[(g_{Q'}^{Kerr (r)}, k_{Q'}^{Kerr (r)}); A_{1}] \label{eq:extr-kerr-DltQ} \\
	&=(r^{-1} (\bfE, \bfP)(Q - Q'), r^{-2} (\bfC, \bfJ)(Q - Q')) + \calO_{\gmm_{\max}} \left( M_{\max} r^{-1} \abs{(r^{-1} (\bfE, \bfP)(Q - Q'), r^{-2} (\bfC, \bfJ)(Q - Q'))} \right) \notag
\end{align}
for $\gmm_{\max} \geq \max\set{\gmm(Q), \gmm(Q')}$, $M_{\max} \geq \max \set{\abs{\bfM(Q)}, \abs{\bfM(Q')}}$ and $r$ satisfying \eqref{eq:extr-kerr-r} with respect to $Q, Q'$.

Consider
\begin{equation*}
	\calE_{1}^{(r)} = \set*{Q \in \calE : \tfrac{1}{2} \bfE(Q) < \bfE^{ADM}[(g_{in}, k_{in})] < 2 \bfE(Q), \, \gmm(Q) < 2 \gmm^{ADM}[(g_{in}, k_{in})], \, \abs{(\bfC, \bfJ)}(Q) \leq C D_{0}^{2} r^{2-2\alp}},
\end{equation*}
where $\gmm^{ADM} = ((\bfE^{ADM})^{2} - \abs{\bfP^{ADM}}^{2})^{-\frac{1}{2}}\bfE^{ADM}$ and $C$ is larger than the implicit constant in the last bound in \eqref{eq:extr-r}. Since $\alp > \frac{1}{2}$, there exists $R_{1} \geq R_{0}$ such that \eqref{eq:extr-kerr-r} is satisfied for all $r \geq R_{1}$ and $Q \in \calE_{1}^{(r)}$. In particular, by \eqref{eq:extr-kerr-DltQ}, the map $T^{(r)} : \calE_{1}^{(r)} \to \bbR^{10}$, $Q \mapsto \bfQ[(g_{Q}^{Kerr(r)}, k_{Q}^{Kerr(r)}); A_{1}]$ is one-to-one and bi-Lipschitz if $r \geq R_{1}$. Moreover, in view of \eqref{eq:extr-r} and \eqref{eq:extr-kerr-Q}, as well as the inverse function theorem, we may ensure that $T^{(r)}(\calE_{1}^{(r)})$ contains $\calQ = B_{M_{c} \eps^{2}}(\mathring{Q})$ if $r$ is sufficiently large. Hence, inverting $T^{(r)}$ and composing with $(g_{Q}^{Kerr (r)}, k_{Q}^{Kerr (r)})$, we produce a $\calQ$-admissible family on $A_{1}$. Finally, taking $r$ even larger if necessary, we obtain $\eps < \eps_{c}$ and $C_{D}(\lceil s \rceil, \gmm) \eps < \eps_{c}$, so Theorem~\ref{thm:gluing-annulus} can be applied. \qedhere
\end{proof}

\section{Obstruction-free gluing} \label{sec:obs-free}
In this section, we prove the obstruction-free gluing theorems stated in Section~\ref{subsec:obs-free-results}.

\subsection{Outline of the proof} \label{subsec:obs-free-overview}
Our proof consists of the following steps.

\subsubsection{Construction of a single localized bump} \label{subsubsec:single-bump}
At the heart of our proof is the following construction. Given $\tht \in (0, \frac{\pi}{2})$, $\bfomg \in \bbS^{2}$, $b > 0$, $t > 0$, $\ell \in \bbR^{3}$ and $\xi \in \bbR^{3}$ such that $\br{B_{b}(\xi)} \subseteq C_{\tht}(\bfomg) \cap A_{8}$, we construct a smooth solution $(g_{t, \ell, \xi}, k_{t, \ell, \xi})$ to \eqref{eq:constraint} on $\bbR^{3}$ that
\begin{enumerate}[label=(\roman*)]
\item equals $(\dlt, 0)$ outside $C_{\tht}(\bfomg) \cap \set{\abs{x} > 8}$, and
\item attains the following (averaged) charges $\Dlt Q \in \bbR^{10}$ in an outer annulus (say, $A_{16}$):
\begin{equation*}
	(\bfE, \bfP_{i}, \bfC_{i}, \bfJ_{i})[(g_{t, \ell, \xi}, k_{t, \ell, \xi}); A_{16}] = (\gmm(\ell) t^{2}, \gmm(\ell) t^{2} \ell_{i}, \gmm(\ell) t^{2} \xi_{i}, \gmm(\ell) t^{2} \tensor{\eps}{_{i}^{j k}} \xi_{j} \ell_{k})  + \cdots,
\end{equation*}
where $\gmm(\ell) = (1-\abs{\ell}^{2})^{-\frac{1}{2}}$ and we omitted terms that are smaller than $t^{2}$ provided that $t \ll b \ll 1$.
\end{enumerate}
Intuitively, $(g_{t, \ell, \xi}, k_{t, \ell, \xi})$ may be thought of as bump (or particle) with the mass, velocity and position $t^{2}$, $\ell$ and $\xi$, respectively, with characteristic scale $b$.

A key difficulty in this construction is the fact that any smooth solution $(\alp, \bt)$ on $\bbR^{3}$ to the linearization of \eqref{eq:constraint} around $(\dlt, 0)$ has zero charges (measured on any $A_{r}$) due to the conservation laws, Lemma~\ref{lem:charge}. Hence, (ii) must arise from the properties of the nonlinearity of \eqref{eq:constraint}. Another difficulty lies in ensuring the support property (i) while achieving (ii).

With these difficulties in mind, let us first consider the zero velocity (or time-symmetric) case $\ell = 0$. Let $\mathring{\alp}_{ij}$ be a smooth symmetric $2$-tensor that satisfies $\mathring{\alp} = \calO(t)$, $\supp \mathring{\alp} \subseteq B_{b}(0)$ and
\begin{equation*}
\rd^{i} (\mathring{\alp}_{ij} - \dlt_{ij} \tr_{\dlt} \mathring{\alp})= 0,
\end{equation*}
which is stronger than $(\mathring{\alp}, 0)$ solving the linearized constraint equation (i.e., $\rd^{i} \rd^{j} (\mathring{\alp}_{ij} - \dlt_{ij} \tr_{\dlt} \mathring{\alp})= 0$). Suppose that $(g_{t, 0, 0}, k_{t, 0, 0})$ is a solution to \eqref{eq:constraint} such that $(g_{t, 0, 0}, k_{t, 0, 0}) = (\dlt + \mathring{\alp}, 0) + \calO(t^{2})$. A key \emph{nonlinear} computation, which is similar to that of Bartnik in \cite{Bar}, shows that the charges of $(g_{t, 0, 0}, k_{t, 0, 0})$ are determined by an explicit expression in $\mathring{\alp}$ up to leading order:
\begin{equation*}
(\bfE, \bfP_{i}, \bfC_{i}, \bfJ_{i})[(g_{t, 0, 0}, k_{t, 0, 0}); A_{r_{0}}(\xi_{0})] = \bb( \tfrac{1}{16} \int \sum_{j, k, \ell} (\rd_{k} \mathring{\alp}_{j \ell} - \rd_{\ell} \mathring{\alp}_{j k})^{2} \, \ud x, 0, 0, 0 \bb) + \cdots,
\end{equation*}
where $\supp \mathring{\alp} \subseteq B_{r_{0}}(\xi_{0})$. By normalizing $\mathring{\alp}$, we may ensure that $\bfE[(g_{t, 0, 0}, k_{t, 0, 0}); A_{r}] = t^{2}$.

To construct such a solution $(g_{t, 0, 0}, k_{t, 0, 0})$ with a good support property, we use the conic solution operator $\vec{S}_{c}$ from \cite{MaoTao} (see Section~\ref{subsec:conic}) to set up a Picard iteration argument. By selecting the convolution kernel of $\vec{S}_{c}$ to be supported in $C_{\tht}(\bfomg)$, we may ensure that $\supp (g_{t, 0, 0}, k_{t, 0, 0}) \subseteq \cup_{x \in \supp \mathring{\alp}} (C_{\tht}(\bfomg) + x)$. Applying the spatial translation $x \mapsto x - \xi$ (with $b$ and $\xi$ as above), we obtain $(g_{t, 0, \xi}, k_{t, 0, \xi})$ with properties (i) and (ii).

To handle the case $\ell \neq 0$, the idea is to utilize Lorentz boosts, which now requires thinking about objects defined on spacetime (as opposed to spacelike hypersurface). For solutions to the linearization of \eqref{eq:eve} around $\bfeta$, an elegant proof of the transformation properties of the charges under Lorentz boosts can be given using the Hamiltonian \eqref{eq:U-X} associated with (continuous) isometries of $\bfeta$ introduced by Chru\'sciel \cite{Chr1}. One way to extend this property to our nonlinear setting is to first consider the Cauchy development of $(g_{t, 0, \xi}, k_{t, 0, \xi})$ with respect to \eqref{eq:eve} using the Choquet-Bruhat theorem \cite{ChBr}, then controlling the nonlinear error terms in the proof of the transformation property (cf.~\cite[Appendix~E]{ChrDel} in the case of the ADM charges). The upside is that this procedure produces a boosted pair $(\td{g}_{t, \ell, \xi}, \td{k}_{t, \ell, \xi})$ attaining the desired charges (ii). The downside, however, is that the support property (i) is difficult to keep, particularly for $\abs{\ell}$ close to $1$.

To avoid this issue, we instead consider the Cauchy development $\bfalp$ of $(\mathring{\alp}, 0)$ with respect to the \emph{linearized} Einstein vacuum equation (around $\bfeta$), obtain the boosted pair $(\alp_{t, \ell, \xi}, \bt_{t, \ell, \xi})$ solving the linearized constraint equation (around $(\dlt, 0)$), then upgrade it using the conic solution operator $\vec{S}_{c}$ to a solution $(g_{t, \ell, \xi}, k_{t, \ell, \xi})$ to \eqref{eq:constraint} with the desired support property (i). Next, we apply the above Lorentz boost argument in reverse for the Cauchy development of $(g_{t, \ell, \xi}, k_{t, \ell, \xi})$ with respect to \eqref{eq:eve}. The key observation is that the resulting un-boosted (also un-translated by $\xi$) pair still obeys $(\td{g}_{t, 0, 0}, \td{k}_{t, 0, 0}) = (\dlt + \mathring{\alp}, 0) + \calO(t^{2})$, which is sufficient for the charge computations to go through.

\subsubsection{Multi-bump configurations with prescribed charges} \label{subsubsec:multi-bumps}
The next step of the proof is to put together multiple bumps obtained in \S \ref{subsubsec:single-bump} -- which may be arranged to have pairwise disjoint supports -- to construct a smooth solution $(g_{\Dlt Q}^{{\rm bump}; \Gmm}, k_{\Dlt Q}^{{\rm bump}; \Gmm})$ to \eqref{eq:constraint} on $\bbR^{3}$ that attains the (averaged) charges $\Dlt Q \in \bbR^{10}$ as measured in the annulus $A_{16}$ for any $\Dlt Q =  (\Dlt \bfE, \ldots, \Dlt \bfJ_{3})$ satisfying
\begin{equation*}
	\Dlt \bfE > \abs{\Dlt \bfP}, \,
	\tfrac{\Dlt \bfE}{\sqrt{(\Dlt \bfE)^{2} - \abs{\Dlt \bfP}^{2}}} < 2 \Gmm, \,
	\Dlt\bfE < \eps_{b}^{2}, \,
	\abs{\Dlt \bfC} + \abs{\Dlt \bfJ} < \mu_{b} \Dlt \bfE,
\end{equation*}
with $\eps_{b}$, $\mu_{b}$ sufficiently small depending on $\Gmm$. Importantly, this range is larger than \eqref{eq:obs-free-unit:EP}--\eqref{eq:obs-free-unit:CJ} ($\eps_{o}, \mu_{o}$ will be chosen small compared to $\eps_{b}, \mu_{b}$).

There are many possible ways to achieve this goal. Our configuration consists of six bumps -- see Figure~\ref{fig:bump}. Two of these bumps are placed symmetrically (relative to $0$) on the $x^{3}$ axis and carry most of $\bfE$ and all of $\bfP$ while contributing zero $\bfC$ and $\bfJ$. The remaining four are placed roughly symmetrically (relative to $0)$ on the $x^{1}$ and $x^{2}$ axes, are almost time-symmetric (i.e., small $\abs{\ell}$) and attain the desired values of $\bfC$ and $\bfJ$ while contributing zero $\bfP$. We remark that the $\bfP$, $\bfC$ and $\bfJ$ of the latter quadruple is well-approximated by the total linear momentum, center of mass and angular momentum of the system of 4 point particles in Newtonian mechanics in the same configuration.

\subsubsection{Extension procedures and proof of Theorem~\ref{thm:obs-free-unit}} \label{subsubsec:obs-free-unit}
Extending $(g_{out}, k_{out})$ to $\td{A}_{16}$ (which is straightforward; see Lemma~\ref{lem:ext-ingoing}) and applying Theorem~\ref{thm:gluing-annulus} (gluing up to linear obstructions), the proof of Theorem~\ref{thm:obs-free-unit} is reduced to constructing an admissible family of extensions $(g_{\Dlt Q}, k_{\Dlt Q})$ (in the outward direction) of $(g_{in}, k_{in})$ with
\begin{equation*}
\bfQ[(g_{\Dlt Q}, k_{\Dlt Q}]; A_{16}) \approx \bfQ[(g_{in}, k_{in}]; A_{1}) + \Dlt Q
\approx \bfQ[(g_{out}, k_{out}]; A_{32})
\end{equation*}
up to errors of order $\calO(c \Dlt \bfE)$ with $c = c(\Gmm)$ small. To achieve this goal, we first extend $(g_{in} - \dlt, k_{in})$ outward to $(\td{g}_{in} - \dlt, \td{k}_{in})$ using Carlotto--Schoen-type techniques (but based on conic- and Bogovskii-type solution operators as in \cite{MaoTao}), which can be arranged to have disjoint supports from $(g_{\Dlt Q}^{{\rm bump}; \Gmm}, k_{\Dlt Q}^{{\rm bump}; \Gmm})$; see Figure~\ref{fig:bump}. We then simply superpose $(\td{g}_{in} - \dlt, \td{k}_{in})$ on $(g_{\Dlt Q}^{{\rm bump}; \Gmm}, k_{\Dlt Q}^{{\rm bump}; \Gmm})$ to produce the desired pair $(g_{\Dlt Q}, k_{\Dlt Q})$.

\subsubsection{Theorem~\ref{thm:obs-free-unit} $\imp$ Theorem~\ref{thm:obs-free-af}} \label{subsubsec:obs-free-af} This is a simple rescaling argument (cf.~Section~\ref{sec:af}).

\begin{remark}[Comparison with \cite{CziRod}]
Apart from the apparent differences -- such as null vs.~spacelike gluing, spherical vs.~annular data and the use of a Lorentz boost to optimize the range of $\Dlt \bfP$ -- perhaps the biggest difference between our approach and \cite{CziRod} is the absence of extra oscillations in the bumps, which correspond to, in the context of \cite{CziRod}, the ``nonlinear corrections'' for ``gluing'' the charges. Roughly speaking, instead of separation in frequencies as in \cite{CziRod}, we rely on separation in physical space to preclude dangerous interactions. Our use of conic solution operators (also Bogovskii operators for extension procedures) to ensure that the multi-bump configuration and $(\td{g}_{in} -\dlt, \td{k}_{in})$ have disjoint supports is an instance of this idea. Placing the gluing annulus $A_{16}$ outside the support of the leading order part $(\alp_{t, \ell, \xi}, \bt_{t, \ell, \xi})$ of all bumps is another (at a more technical level, note the improved bounds \eqref{eq:size-extr-bump-Gmm}, \eqref{eq:Lip-extr-bump-Gmm} in this region).
\end{remark}

In the rest of this section, we carry out the proof of Theorems~\ref{thm:obs-free-unit} and \ref{thm:obs-free-af} outlined above.

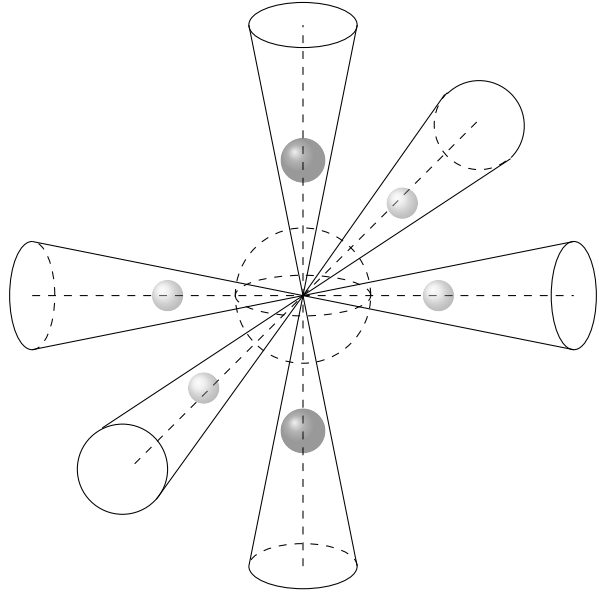
\begin{figure}
    \centering
    \begin{tikzpicture}[scale=0.6]
    \coordinate (O) at (0,0,0);
    \draw[dashed] (xyz cs:x=-6) -- (xyz cs:x=6);
    \draw[dashed] (xyz cs:y=-6) -- (xyz cs:y=6);
    \draw[dashed] (xyz cs:z=-10) -- (xyz cs:z=10);

\draw plot[variable=\t, domain=0:6](\t,0.2*\t);
\draw plot[variable=\t, domain=0:6](\t,-0.2*\t);
\draw(6,0) ellipse (0.5 and 1.2);

  \shade[ball color = gray!40, opacity = 0.4] (3,0) circle (0.35);

\draw plot[variable=\t, domain=0:6](-\t,0.2*\t);
\draw plot[variable=\t, domain=0:6](-\t,-0.2*\t);
\draw(-6,1.2) arc (90:270:0.5 and 1.2);
\draw[dashed] (-6,-1.2) arc (-90:90:0.5 and 1.2);

  \shade[ball color = gray!40, opacity = 0.4] (-3,0) circle (0.35);

\draw plot[variable=\t, domain=0:6](0.2*\t,\t);
\draw plot[variable=\t, domain=0:6](-0.2*\t,\t);
\draw(0,6) ellipse (1.2 and 0.5);

  \shade[ball color = black, opacity = 0.4] (0,3) circle (0.5);

\draw plot[variable=\t, domain=0:6](0.2*\t,-\t);
\draw plot[variable=\t, domain=0:6](-0.2*\t,-\t);
\draw (-1.2,-6) arc (180:360:1.2 and 0.5);
\draw[dashed] (1.2,-6) arc (0:180:1.2 and 0.5);

  \shade[ball color = black, opacity = 0.4] (0,-3) circle (0.5);

\draw plot[variable=\t, domain=0:4.6](\t,0.66*\t);
\draw plot[variable=\t, domain=0:3.14](\t,1.42*\t);
\draw[dashed] (3.2,4.5) arc (135:315: 1);
\draw (4.61,3.06) arc (-45:135: 1);

  \shade[ball color = gray!40, opacity = 0.4] (2.2,2.05) circle (0.35);

\draw plot[variable=\t, domain=0:4.45](-\t,-0.66*\t);
\draw plot[variable=\t, domain=0:3.25](-\t,-1.39*\t);
\draw(-4,-3.85) circle (1);

  \shade[ball color = gray!40, opacity = 0.4] (-2.2,-2.05) circle (0.35);

\draw[dashed] (0,0) circle (1.5);
\draw[dashed] (-1.5,0) arc (180:360:1.5 and 0.45);
\draw[dashed] (1.5,0) arc (0:180:1.5 and 0.45);

\end{tikzpicture}
    \caption{Six-bump configuration. The two darker balls represent the two bumps carrying most of $\bfE$ and all of $\bfP$, while the four lighter balls represent the four bumps that give $\bfC$ and $\bfJ$. The six-bump solution is supported inside the union of the six cones and away from the inner dashed ball. The extension $(\td{g}_{in} - \delta, \td{k}_{in})$ of $(g_{in} - \delta, k_{in})$ will be designed to have disjoint support from the six-bump solution. The gluing annulus $A_{16}$ sits outside of all the bumps.}
    \label{fig:bump}
\end{figure}

\subsection{Construction of a single localized bump}
The main goal of this subsection is to prove the following.
\begin{proposition} \label{prop:bump-single}
Fix any $-1 < \dlt < - \frac{1}{2}$, $\tht \in (0, \frac{\pi}{2})$ and $\bfomg \in \bbS^{2}$. For any $s > \frac{3}{2}$, $b > 0$ and $\Gmm > 1$, there exists $\eps_{b, 0} = \eps_{b, 0}(b, \Gmm) > 0$ such that the following holds. Given $t > 0$, $\ell \in \bbR^{3}$ and $\xi \in \bbR^{3}$ obeying
\begin{equation*}
t < \eps_{b, 0}, \quad
1 \leq \gmm(\ell) := (1-\abs{\ell}^{2})^{-\frac{1}{2}} < \Gmm, \quad
\br{B_{b}(\xi)} \subseteq C_{\tht}(\bfomg) \cap A_{8},
\end{equation*}
there exists a solution $(g_{t, \ell, \xi}, k_{t, \ell, \xi}) \in C^{\infty}(\bbR^{3})$ to \eqref{eq:constraint} satisfying
\begin{equation} \label{eq:bump-single-supp}
	\supp (g_{t, \ell, \xi} - \dlt, k_{t, \ell, \xi}) \subseteq C_{\tht}(\bfomg) \cap \set{\abs{x} > 8}, \quad
\end{equation}
and, for $n = 0, 1$,
\begin{align}
\nrm{(t \rd_{t}, \rd_{\ell}, \rd_{\xi})^{(n)}(g_{t, \ell, \xi} - \dlt, k_{t, \ell, \xi})}_{\Hb^{s, \dlt} \times \Hb^{s-1, \dlt+1}} &\aleq_{s, b, \Gmm} t, \label{eq:bump-single-est} \\
	(t \rd_{t}, \rd_{\ell}, \rd_{\xi})^{(n)} \big[ \bfE[(g_{t, \ell, \xi}, k_{t, \ell, \xi}); A_{16}] - \gmm(\ell) t^{2} \big] &=  \calO_{b, \Gmm}(t^{3}), \label{eq:bump-single-E} \\
	(t \rd_{t}, \rd_{\ell}, \rd_{\xi})^{(n)} \big[ \bfP_{i}[(g_{t, \ell, \xi}, k_{t, \ell, \xi}); A_{16}] - \gmm(\ell) t^{2} \ell_{i} \big] &= \calO_{b, \Gmm}(t^{3}), \label{eq:bump-single-P} \\
	(t \rd_{t}, \rd_{\ell}, \rd_{\xi})^{(n)} \big[ \bfC_{i}[(g_{t, \ell, \xi}, k_{t, \ell, \xi}); A_{16}] - \gmm(\ell) t^{2} \xi_{i} \big] &= \calO(b t^{2}) + \calO_{b, \Gmm}(t^{3}), \label{eq:bump-single-C} \\
	(t \rd_{t}, \rd_{\ell}, \rd_{\xi})^{(n)} \big[ \bfJ_{i}[(g_{t, \ell, \xi}, k_{t, \ell, \xi}); A_{16}] - \gmm(\ell) t^{2} \tensor{\eps}{_{i}^{jk}} \xi_{j} \ell_{k} \big] &= \calO(b t^{2}) + \calO_{b, \Gmm}(t^{3}). \label{eq:bump-single-J}
\end{align}
Moreover, the following improved bound holds in $\set{\abs{x} > 16}$ (for $n = 0, 1$):
\begin{equation} \label{eq:bump-single-imp}
	\nrm{(t \rd_{t}, \rd_{\ell}, \rd_{\xi})^{(n)}(g_{t, \ell, \xi} - \dlt, k_{t, \ell, \xi})}_{\Hb^{s, \dlt} \times \Hb^{s-1, \dlt+1}(\set{\abs{x} > 16})} \aleq_{s, b, \Gmm} t^{2}.
\end{equation}
\end{proposition}
We remark that $B_{b}(\xi)$ contains the support of the seed bump, or more precisely the $\calO(t)$ part of $(g_{t, \ell, \xi}, k_{t, \ell, \xi})$; the reason behind the improvement \eqref{eq:bump-single-imp} is that $\set{\abs{x} > 16}$ is disjoint from $B_{b}(\xi)$.

To establish Proposition~\ref{prop:bump-single}, we need some preliminary lemmas.
\begin{lemma}[Charge computation for special initial data] \label{lem:bump-en}
Given $r_{0} \in 2^{\bbN_{0}}$, $\xi_{0} \in \bbR^{3}$, $t, b_{0}, A > 0$ and $\tau \in (-\dlt_{0}, \dlt_{0})$ for some $\dlt_{0} > 0$, let $(g, k) = (g(\tau), k(\tau)) \in C^{2} \times C^{1}(B_{2r_{0}}(\xi_{0}))$ be a solution to \eqref{eq:constraint} satisfying, for $n = 0, 1$,
\begin{equation*}
	\nrm{\rd_{\tau}^{n} (g_{ij} - \dlt_{ij} - \mathring{\alp}_{ij}, k_{ij})}_{C^{2} \times C^{1}(B_{2r_{0}}(\xi_{0}))} \leq A t^{2},
\end{equation*}
where $\mathring{\alp}_{ij} = \mathring{\alp}(\tau)_{ij}$ obeys the special condition
\begin{equation} \label{eq:bump-en-alp-special}
	\rd^{i} (\mathring{\alp}_{ij} - \dlt_{ij} \tr_{\dlt} \mathring{\alp}) = 0,
\end{equation}
as well as the size and support properties (for $n = 0, 1$ and $n' = 0, 1, 2$)
\begin{equation} \label{eq:bump-en-alp}
\abs{\rd_{\tau}^{n} (b_{0} \rd_{x})^{(n')} \mathring{\alp}} \leq A b_{0}^{\frac{3}{2}} t, \quad \supp \mathring{\alp} \subseteq B_{b_{0}}(0).
\end{equation}
Provided that $\br{B_{b_{0}}(0)} \subseteq B_{r_{0}}(\xi_{0})$, we have, for $n = 0, 1$,
\begin{align*}
	\rd_{\tau}^{n} \bb[ \bfE[(g, k); A_{r_{0}}(\xi_{0})] - \tfrac{1}{16} \int \sum_{j, k, \ell} (\rd_{k} \mathring{\alp}_{j \ell} - \rd_{\ell} \mathring{\alp}_{j k})^{2} \, \ud x \bb] &=  \calO_{A, b_{0}}(t^{3}), \\
	\rd_{\tau}^{n} \bfP_{i}[(g, k); A_{r_{0}}(\xi_{0})] &= \calO_{A, b_{0}}(t^{3}), \\
	\rd_{\tau}^{n} \bfC_{i}[(g, k); A_{r_{0}}(\xi_{0})] &= \calO(A^{2} b_{0} t^{2}) + \calO_{A, b_{0}}(t^{3}), \\
	\rd_{\tau}^{n} \bfJ_{i}[(g, k); A_{r_{0}}(\xi_{0})] &= \calO_{A, b_{0}}(t^{3}).
\end{align*}
\end{lemma}
\begin{proof}
For simplicity, we drop the $\tau$-dependence; we leave the similar proof of the $C^{1}$-dependence on $\tau$ to the reader. Define $(h, \pi) := ((g - \dlt) - \dlt \tr_{\dlt} (g-\dlt), k - \dlt \tr_{\dlt} k)$ and $(h_{1}, \pi_{1}) := (\mathring{\alp} - \dlt \tr_{\dlt} \mathring{\alp}, 0)$. Using the notation in \eqref{eq:constraint-abbrev}, the charges are given by
\begin{align*}
	&\begin{pmatrix} \bfE \\  \bfC_{k} \end{pmatrix}[(g, k); A_{r_{0}}] =-\frac{1}{2}\int \chi_{r_{0}, 2 r_{0}}(x-\xi_{0}) \rd_{i} \rd_{j} h^{ij} \begin{pmatrix} 1 \\ x_{k} \end{pmatrix} \, \ud x=-\frac{1}{2}\int  \chi_{r_{0}, 2 r_{0}}(x - \xi_{0}) M(h, \pi)\begin{pmatrix} 1 \\ x_{k} \end{pmatrix} \, \ud x\\
 &=- \frac{1}{2} \int M(h_{1}, \pi_{1})\begin{pmatrix} 1 \\ x_{k} \end{pmatrix} \, \ud x +\mathcal{O}_{A, b_{0}} (t^3)=- \frac{1}{2} \int R[\dlt + \mathring{\alp}] \begin{pmatrix} 1 \\ x_{k} \end{pmatrix} \, \ud x +\mathcal{O}_{A, b_{0}}(t^3),
\end{align*}
\begin{align*}
 &\begin{pmatrix} \bfP_{k} \\  \bfJ_{k} \end{pmatrix}[(g, k); A_{r_{0}}]=
 -\int \sum_{j}  \chi_{r_{0}, 2r_{0}}(x - \xi_{0}) \rd_{i} \pi^{ij} \begin{pmatrix} \bfe_{k}^{j}  \\ \bfY_{k}^{j} \end{pmatrix} \, \ud x =
 -\int \sum_{j}  \chi_{r_{0}, 2 r_{0}}(x - \xi_{0})  N^{j}(h, \pi) \begin{pmatrix} \bfe_{k}^{j} \\ \bfY_{k}^{j} \end{pmatrix} \, \ud x \\
 &=-\int \sum_{j}  N^j(h_{1}, \pi_{1}) \begin{pmatrix} \bfe_{k}^{j} \\ \bfY_{k}^{j} \end{pmatrix} \, \ud x+\mathcal{O}_{A, b_{0}}(t^3)= \mathcal{O}_{A, b_{0}}(t^3),
\end{align*}
where we used the fact that $\rd_{i} \rd_{j} h_{1}^{i j} = 0$, $\pi_{1} = 0$ and $\supp (h_{1}, \pi_{1}) \subseteq B_{r_{0}}(\xi_{0})$.

The above already implies the desired assertions for $\bfP_{k}$ and $\bfJ_{k}$. For $\bfE$ and $\bfC_{k}$, recall the general formulae
\begin{align*}
    R[g]=(g^{-1})^{ij}(\partial_k\Gamma_{ij}^k-\partial_j\Gamma_{ik}^k+\Gamma_{k \ell}^k\Gamma_{ij}^{\ell}-\Gamma_{i \ell}^k\Gamma_{jk}^{\ell}),\quad \Gamma_{ij}^k=\frac{1}{2} (g^{-1})^{k \ell}(\partial_i g_{j \ell}+\partial_j g_{i \ell}-\partial_{\ell} g_{ij}).
\end{align*}
Let $(g_{1})_{ij} :=\delta_{ij}+\mathring{\alpha}_{ij}$ and write $(g_{1}^{-1})^{ij}=\delta^{ij}+\td{\alpha}^{ij}$, where we observe that $\td{\alp}^{ij} = -\mathring{\alp}_{ij} + \calO_{A, b_{0}}(t^{2})$. We will abbreviate terms that are linear in $\mathring{\alp}$ by ${\rm lin}$ -- note that they will all vanish since $(\mathring{\alp}, 0)$ is chosen to solve the linearized constraint equation. We have
\begin{align*}
    R[g_{1}]&=\frac{1}{2}(g_{1}^{-1})^{ij}\partial_k\td{\alp}^{k \ell}(\partial_i \mathring{\alp}_{j \ell}+\partial_j \mathring{\alp}_{i \ell}-\partial_{\ell} \mathring{\alp}_{ij})+\frac{1}{2}(g_{1}^{-1})^{ij}(g_{1}^{-1})^{k \ell}\partial_k(\partial_i \mathring{\alp}_{j \ell}+\partial_j \mathring{\alp}_{i \ell}-\partial_{\ell} \mathring{\alp}_{ij})\\
    &\peq -\frac{1}{2}(g_{1}^{-1})^{ij}\partial_j\td{\alp}^{k \ell}\partial_i \mathring{\alp}_{k \ell}-\frac{1}{2}(g_{1}^{-1})^{ij}(g_{1}^{-1})^{k \ell}\partial_j\partial_i \mathring{\alp}_{k \ell}+(g_{1}^{-1})^{ij}(\Gamma_{k \ell}^k\Gamma_{ij}^{\ell}-\Gamma_{i \ell}^k\Gamma_{jk}^{\ell}) \\
    &=\frac{1}{2}\partial_k(\td{\alp}^{k \ell}(2\partial_i \mathring{\alp}_{i \ell}-\partial_{\ell} \mathring{\alp}_{ii}))+\frac{1}{2}\td{\alp}^{ij}\partial_k(2\partial_i \mathring{\alp}_{jk}-\partial_k \mathring{\alp}_{ij})-\frac{1}{2}\partial_i(\td{\alp}^{k \ell}\partial_i \mathring{\alp}_{k \ell})-\frac{1}{2}\td{\alp}^{ij}\partial_i\partial_j \mathring{\alp}_{kk}
    \\
    &\peq +\frac{1}{4}\partial_{\ell} \mathring{\alp}_{kk}(2\partial_i \mathring{\alp}_{i \ell}-\partial_{\ell} \mathring{\alp}_{ii})-\frac{1}{4}(\partial_i \mathring{\alp}_{k \ell})^2+\frac{1}{4}(\partial_k \mathring{\alp}_{i \ell}-\partial_{\ell} \mathring{\alp}_{ik})^2+{\rm lin}+\calO_{A, b_{0}}(t^3).
\end{align*}
To simplify the nonlinear terms, we use the special condition \eqref{eq:bump-en-alp-special}, which implies
\begin{align*}
    \partial^i \mathring{\alp}_{ij}=\partial_j \tr_\delta  \mathring{\alp},\quad  \partial^i\partial^j   \mathring{\alp}_{ij}=\Delta\tr_\delta \mathring{\alp}.
\end{align*}
Recall also that ${\rm lin} = \uD_{\dlt} R[\mathring{\alp}] = 0$, $g_{1} = \dlt + \mathring{\alp}$ and $\td{\alp}^{ij} = - \alp_{ij} + \calO_{A, b_{0}}(t^{2})$. Thus
\begin{align*}
    R[\dlt + \mathring{\alp}]&=\left(1-\frac12+\frac{1}{2}-\frac{1}{4}-\frac{1}{2}\right)(\partial_j\tr_\delta\mathring{\alp})^2+\left(-\frac{1}{2}-\frac{1}{4}+\frac{1}{2}\right)(\partial_k \mathring{\alp}_{ij})^2+\partial(\mathring{\alp},\partial\mathring{\alp})+ \calO_{A, b_{0}}(t^3)\\
    &=\frac{1}{4}(\partial_j\tr_\delta\mathring{\alp})^2-\frac{1}{4}(\partial_k\mathring{\alp}_{ij})^2+\partial(\mathring{\alp},\partial\mathring{\alp})+ \calO_{A, b_{0}}(t^3)\\
    &=-\frac{1}{8}(\partial_k \mathring{\alp}_{i \ell}-\partial_{\ell} \mathring{\alp}_{ik})^2+\partial(\mathring{\alp},\partial\mathring{\alp})+ \calO_{A, b_{0}}(t^3).
\end{align*}
Plugging in this formula into the above formulae for $\bfE$ and $\bfC_{k}$, and using the fact that $\abs{x_{k}} \leq b_{0}$ in $\supp \mathring{\alp} \subseteq B_{b_{0}}(0)$, the desired assertions for $\bfE$ and $\bfC_{k}$ follow. \qedhere
\end{proof}

To state the next lemma, consider a Lorentzian metric $\bfg \in C^{2}(\bfOmg)$, where $\bfOmg \subseteq \bbR^{1+3}$. Given  $(\Lmb, \bfxi) \in SO^{+}(1, 3) \times \bbR^{1+3}$, let $y^{\mu} = \tensor{\Lmb}{^{\mu}_{\mu'}} x^{\mu'} + \bfxi^{\nu}$. Then, with respect to the coordinates $(y^{0}, \ldots, y^{3})$, define
\begin{equation} \label{eq:induced-data}
	g^{(\Lmb, \bfxi)}_{ij} = \bfg_{ij} \bb|_{\set{y^{0} = 0} \cap \bfOmg}, \quad
	k^{(\Lmb, \bfxi)}_{ij} = \frac{-(\bfg^{-1})^{0 \mu}}{2 \sqrt{-(\bfg^{-1})^{00}}} \left( \rd_{\mu} \bfg_{i j} - \rd_{i} \bfg_{j \mu} - \rd_{j} \bfg_{i \mu} \right) \bb|_{\set{y^{0} = 0} \cap \bfOmg}.
\end{equation}
Motivated by Lemma~\ref{lem:charge-spt}, we introduce the notation
\begin{equation} \label{eq:spt-averge-charges}
	(\bbP_{0}, \bbP_{i}, \bbJ_{i 0}, \bbJ_{jk})[(g, k); A_{r_{0}}(\xi_{0})] = (\bfE, \bfP_{i}, \bfC_{i}, \tensor{\eps}{^{i}_{jk}} \bfJ_{i})[(g, k); A_{r_{0}}(\xi_{0})].
\end{equation}

\begin{lemma}[Effect of boost and translation] \label{lem:bump-poincare}
Given $t, A > 0$ and $\tau \in (-\dlt_{0}, \dlt_{0})$ for some $\dlt_{0} > 0$, let $\bfg = \bfg(\tau) \in C^{2}(\bfOmg)$ be a solution to \eqref{eq:eve} satisfying, for $n = 0, 1$,
\begin{equation*}
	\nrm{\rd_{\tau}^{n}(\bfg_{\mu \nu} - \bfeta_{\mu \nu} - \bfalp_{\mu \nu})}_{C^{2}(\bfOmg)} \leq A t^{2},
\end{equation*}
where $\bfalp_{\mu \nu} = \bfalp(\tau)_{\mu \nu}$ obeys
\begin{equation*}
	\mathrm{D}_{\bfeta} \bfG [\bfalp] = 0, \quad \nrm{\rd_{\tau} \bfalp}_{C^{2}(\bfOmg)} \leq A t.
\end{equation*}
\begin{enumerate}
\item For any $C^{1}$-curve $(\Lmb(\tau), \bfxi(\tau)) \in SO^{+}(1, 3) \times \bbR^{1+3}$ with $\abs{\rd_{\tau} (\Lmb, \bfxi)} \leq A$, we have, for $n = 0, 1$,
\begin{equation*}
	\nrm{\rd_{\tau}^{n} (g^{(\Lmb, \bfxi)}_{ij} - \dlt_{ij} - \alp^{(\Lmb, \bfxi)}, k^{(\Lmb, \bfxi)}_{ij} - \bt^{(\Lmb, \bfxi)})_{ij} )}_{C^{2}(\bfOmg)} = \calO_{A}(t^{2}),
\end{equation*}
where
\begin{equation*}
\alp^{(\Lmb, \bfxi)}_{i j} = \bfalp_{ij} \bb|_{\set{y^{0} = 0}}, \quad
\bt^{(\Lmb, \bfxi)}_{i j} = \frac{1}{2} \left( \rd_{0} \bfalp_{i j} - \rd_{i} \bfalp_{j 0} - \rd_{j} \bfalp_{i 0} \right) \bb|_{\set{y^{0} = 0}}
\end{equation*}
with respect to $y^{\mu} = \tensor{\Lmb}{^{\mu}_{\nu}} x^{\nu} + \bfxi^{\mu}$.

\item Let $(g, k) = (g^{(I, 0)}, k^{(I, 0)})$. Assume that, for each $\sgm \in [0, 1]$, there exists hypersurface $U_{\sgm} \subseteq \bfOmg$ such that $\rd U_{\sgm} = \set{x^{0} = 0, \abs{x} = (1+\sgm) r_{0}} \cup \set{y^{0} = 0, \abs{y} = (1+\sgm) r}$, $U_{\sgm} \cap \supp \bfalp = \0$, and $\mathrm{Vol}(U_{\sgm}) \leq A$, with the volume induced from the auxiliary Riemannian metric $(\ud x^{0})^{2} + \cdots (\ud x^{3})^{2}$. Then for $n = 0, 1$,
\begin{gather*}
	\rd_{\tau}^{n} \left[ \bbP_{\mu}[(g^{(\Lmb, \bfxi)}, k^{(\Lmb, \bfxi)}); A_{r}] - \tensor{\Lmb}{_{\mu}^{\mu'}} \bbP_{\mu'}[(g, k); A_{r_{0}}] \right] = \calO_{A}(t^{4}), \\
	\rd_{\tau}^{n} \left[ \bbJ_{\mu \nu}[(g^{(\Lmb, \bfxi)}, k^{(\Lmb, \bfxi)}); A_{r}] - \left(\tensor{\Lmb}{_{\mu}^{\mu'}} \tensor{\Lmb}{_{\nu}^{\nu'}} \bbJ_{\mu' \nu'} + \bfxi_{\mu} \tensor{\Lmb}{_{\nu}^{\nu'}} \bbP_{\nu'} - \bfxi_{\nu} \tensor{\Lmb}{_{\mu}^{\mu'}} \bbP_{\mu'} \right)[(g, k); A_{r_{0}}] \right] = \calO_{A}(t^{4}).
\end{gather*}
\end{enumerate}
\end{lemma}
\begin{proof}
For simplicity, we drop the $\tau$-dependence; we leave the similar proof of the $C^{1}$-dependence on $\tau$ to the reader. Statement~(1) is a simple consequence of extracting the terms of linear order in \eqref{eq:induced-data} around $\bfeta$. For Statement~(2), observe first that $\bfg = \calO_{A}(t^{2})$ on $U_{\sgm}$ since $\br{U_{\sgm}} \cap \supp \bfalp = \0$, and hence $\uD_{\bfeta} \bfG[\bfg] = \calO_{A}(t^{4})$ on $U_{\sgm}$. Applying \eqref{eq:U-X-stokes} to $U = U_{\sgm}$, we obtain
\begin{equation*}
	\int_{\set{y^{0} = 0, \abs{y} = (1+\sgm) r}} \star {}^{(X)} \bbU[\bfg] - \int_{\set{x^{0} = 0, \abs{x} = (1+\sgm) r_{0}}} \star {}^{(X)} \bbU[\bfg] = \calO_{A}(t^{4}),
\end{equation*}
for any Killing vector field $X$ with respect to $\bfeta$. In view of \eqref{eq:LT-killing-vf}, we obtain
\begin{align*}
\bbP_{\mu}[\bfg; \set{y^{0} = 0, \, \abs{y} = (1+\sgm) r}] &= \tensor{\Lmb}{_{\mu}^{\mu'}} \bbP_{\mu'}[\bfg; \set{x^{0} = 0, \, \abs{x} = (1+\sgm) r_{0}}] + \calO_{A}(t^{4}).
\end{align*}
Using Lemma~\ref{lem:charge-spt} and smoothly averaging in $\sgm$, we obtain the desired assertion for $\bbP_{\mu}$. The case of $\bbJ_{\mu \nu}$ is similar. \qedhere
\end{proof}

\begin{lemma} \label{lem:lin-einstein}
Consider an open subset $\bfOmg \subseteq \bbR^{1+3}$ such that $\Sgm_{0} := \set{x^{0} = 0} \cap \bfOmg$ is a Cauchy hypersurface in $\bfOmg$. Consider a symmetric covariant two-tensor field $\bfF_{\alp \bt}$ satisfying $\nb^{\alp} \bfF_{\alp \bt} = 0$ and a $1$-form $\bfJ_{\bt}$ defined on $\bfOmg$, as well as a pair of symmetric covariant two-tensor fields $(\alp_{ij}, \bt_{ij})$, a function $n$ and a $1$-form $N_{j}$ defined on $\Sgm_{0}$. The following initial value problem posed on $\bfOmg$
\begin{gather*}
\mathrm{D}_{\bfeta} \bfG [\bfalp]_{\alp \bt} = \bfF_{\alp \bt}, \quad \nb^{\alp} \bfH[\bfalp]_{\alp \bt} = \bfJ_{\bt}, \\
 (\bfalp_{ij}, \rd_{0} \bfalp_{ij}) |_{\Sgm_{0}} = (\alp_{ij}, 2 \bt_{ij} + \rd_{i} N_{j} + \rd_{j} N_{i}), \quad \bfalp_{00} |_{\Sgm_{0}} = n, \quad \bfalp_{0 j} |_{\Sgm_{0}} = N_{j},
\end{gather*}
is equivalent to the following reduced problem posed on $\bfOmg$:
\begin{gather*}
\Box \bfalp_{\alp \bt} = - 2 (\bfF_{\alp \bt} - \tfrac{1}{2} \bfeta_{\alp \bt} \tr_{\bfeta} \bfF) + \nb_{\alp} \bfJ_{\bt} + \nb_{\bt} \bfJ_{\alp},\\
 (\bfalp_{ij}, \rd_{0} \bfalp_{ij}) |_{\Sgm_{0}} = (\alp_{ij}, 2 \bt_{ij} + \rd_{i} N_{j} + \rd_{j} N_{i}), \quad \bfalp_{00} |_{\Sgm_{0}} = n, \quad \bfalp_{0 j} |_{\Sgm_{0}} = N_{j}, \\
(\nb^{\alp} \bfH[\bfalp]_{\alp \bt} - \bfJ_{\bt}) |_{\Sgm_{0}} = 0, \quad
(\mathrm{D}_{\bfeta} \bfG [\bfalp]_{\alp 0} - \bfF_{\alp 0}) |_{\Sgm_{0}} = 0.
\end{gather*}
\end{lemma}
We remark that $\nb^{\alp} \bfH[\bfalp]_{\alp \bt} = 0$ is the linearized wave coordinate condition, and that $(\nb^{\alp} \bfH[\bfalp]_{\alp \bt} - \bfJ_{\bt}) |_{\Sgm_{0}}$ induces initial conditions for $\rd_{0} \bfalp_{0 \bt} |_{\Sgm_{0}}$. The reduced problem, being an inhomogeneous initial value problem for the classical wave equation, is well-posed (in, say, $H^{s} \times H^{s-1}(\set{x^{0} = \tau} \cap \bfOmg)$) by our hypothesis on $\bfOmg$.
\begin{proof}
This is a well-known computation essentially dating back to \cite{ChBr}. From \eqref{eq:lin-bfG}, it follows that
\begin{align*}
&\Box \bfalp_{\alp \bt} = - 2 (\bfF_{\alp \bt} - \tfrac{1}{2} \bfeta_{\alp \bt} \tr_{\bfeta} \bfF) + \nb_{\alp} \bfJ_{\bt} + \nb_{\bt} \bfJ_{\alp} \\
	&\impmi \uD_{\bfeta} \bfG[\bfalp]_{\alp \bt} = \bfF_{\alp \bt} + \frac{1}{2} \left( \nb_{\alp} (\nb^{\gmm} \bfH[\bfalp]_{\gmm \bt} - \bfJ_{\bt}) + \nb_{\bt} (\nb^{\gmm} \bfH[\bfalp]_{\gmm \alp} - \bfJ_{\alp}) - \bfeta_{\alp \bt} \nb^{\dlt} (\nb^{\gmm} \bfH[\bfalp]_{\gmm \dlt} - \bfJ_{\dlt}) \right).
\end{align*}
From this equivalence, the derivation of the reduced problem from the first problem is immediate. To show that the solution $\bfalp$ to the reduced problem solves the first problem, it suffices to show that $\nb^{\alp} \bfH[\bfalp]_{\alp \bt} = \bfJ_{\bt}$. Observe that the second equation combined with $\nb^{\alp} \uD_{\bfeta} \bfG[\bfalp]_{\alp \bt} = 0$ (linearized Bianchi) and $\nb^{\alp} \bfF_{\alp \bt} = 0$ implies $\Box (\nb^{\alp} \bfH[\bfalp]_{\alp \bt} - \bfJ_{\bt}) = 0$. Moreover, the same equation combined with $(\mathrm{D}_{\bfeta} \bfG [\bfalp]_{\alp 0} - \bfF_{\alp 0}) |_{\Sgm_{0}} = 0$ implies $\rd_{0} (\nb^{\alp} \bfH[\bfalp]_{\alp \bt} - \bfJ_{\bt}) |_{\Sgm_{0}} = 0$. Since $(\nb^{\alp} \bfH[\bfalp]_{\alp \bt} - \bfJ_{\bt}) |_{\Sgm_{0}} = 0$ as well, it follows from uniqueness of the initial value problem for the classical wave equation that $\nb^{\alp} \bfH[\bfalp]_{\alp \bt} = \bfJ_{\bt}$.
\end{proof}

\begin{proof}[Proof of Proposition~\ref{prop:bump-single}]
{\bf Step~1.~Choice of $\mathring{\alp}$.} We begin by choosing $\mathring{\alp}_{ij} \in C^{\infty}_{c}(\bbR^{3})$, to which we shall apply Lemma~\ref{lem:bump-en} in the end (in particular, for our motivation behind the choices $b_{0} = \frac{1}{2} \Gmm^{-2} b$ and \eqref{eq:alp0-size}, see Steps~2 and 5 below, respectively). Fix a function $\chi \in C^{\infty}_{c}(B_{1})$ such that
\begin{equation} \label{eq:alp0-size}
	\frac{1}{16} \int \abs{\rd (\rd_{11}^{2} + \rd_{22}^{2}) \chi}^{2} \, \ud x = 1.
\end{equation}
Following \cite[Sec.~4.1]{MaoTao}, we set
    \begin{align*}\mathring{\alp}=t \begin{pmatrix}
        \frac{1}{2}(\partial^2_{22}\chi_{b, \Gmm}-\partial^2_{11}\chi_{b, \Gmm})&-\partial^2_{12}\chi_{b, \Gmm}&0\\
        -\partial^2_{12}\chi_{b, \Gmm}&\frac{1}{2}(\partial^2_{11}\chi_{b, \Gmm}-\partial^2_{22}\chi_{b, \Gmm})&0\\
        0&0&-\frac{1}{2}(\partial^2_{11}+\partial^2_{22})\chi_{b, \Gmm}
    \end{pmatrix}
    \end{align*}
where $\chi_{b, \Gmm}(x) = (\tfrac{1}{2} \Gmm^{-2} b)^{\frac{3}{2}} \chi(2 \Gmm^{2} b^{-1} x)$ so that $\nrm{\rd^{3}_{i j k}\chi_{b, \Gmm}}_{L^{2}}^{2} = \nrm{\rd^{3}_{i j k}\chi}_{L^{2}}^{2}$. From the definition, it is straightforward to verify that $\rd^{i} \mathring{\alp}_{ij} = \rd_{j} \tr_{\dlt} \mathring{\alp}$ and \eqref{eq:bump-en-alp} hold with $b_{0} = \tfrac{1}{2} \Gmm^{-2} b$.

\smallskip
\noindent {\bf Step~2.~Construction of $\bfalp$ and $(\alp_{t, \ell, \xi}, \bt_{t, \ell, \xi})$.} We use Lemma~\ref{lem:lin-einstein} to define $\bfalp \in C^{\infty}(\bbR^{1+3})$ to be the unique solution to the problem
\begin{equation} \label{eq:bump-bfalp}
\begin{gathered}
\mathrm{D}_{\bfeta} \bfG [\bfalp]_{\alp \bt} = 0, \quad \nb^{\alp} \bfH[\bfalp]_{\alp \bt} = 0, \\
 (\bfalp_{ij}, \rd_{0} \bfalp_{ij}) |_{\set{x^{0} = 0}} = (\mathring{\alp}_{ij}, 0), \quad \bfalp_{00} |_{\set{x^{0} = 0}} = 0, \quad \bfalp_{0 j} |_{\set{x^{0} = 0}} = 0,
\end{gathered}
\end{equation}
where $\rd_{0} \bfalp_{00} |_{\set{x^{0} = 0}}$, $\rd_{0} \bfalp_{0j} |_{\set{x^{0} = 0}}$ are induced by $\nb^{\alp} \bfH[\bfalp]_{\alp \bt} = 0$. Recall also from Lemma~\ref{lem:lin-einstein} that, in fact, $\Box \bfalp_{\mu \nu} = 0$ and $\supp (\bfalp_{\mu \nu}, \rd_{t} \bfalp_{\mu \nu}) |_{\set{x^{0} = 0}} \subseteq \set{x^{0} = 0, \, \abs{x} < \frac{1}{2} \Gmm^{-2} b}$ for each $\mu, \nu$. Hence, by finite speed of propagation,
\begin{equation} \label{eq:bump-bfalp-supp}
	\supp \bfalp \subseteq \bfOmg_{\bfalp} := \set{\abs{x} < \abs{x^{0}} + \tfrac{1}{2} \Gmm^{-2} b}.
\end{equation}
Using the parameters $\ell$ and $\xi$, define $\Lmb_{\ell} \in SO^{+}(1, 3)$ (Lorentz boost by velocity $\ell$) and $\bfxi \in \bbR^{1+3}$ by
\begin{equation*}
	(\Lmb_{\ell} x)^{0} = \gmm(\ell) (x^{0} - \ell \cdot P_{\ell} x), \quad
	(\Lmb_{\ell} x)^{j} = \gmm(\ell) (- \ell^{j} x^{0} + (P_{\ell} x)^{j}) + (P^{\perp}_{\ell} x)^{j}, \quad
	\bfxi^{0} = 0, \quad
	\bfxi^{j} = \xi^{j},
\end{equation*}
where $\gmm(\ell) = (1 - \abs{\ell}^{2})^{-\frac{1}{2}}$ (Lorentz contraction factor), and $P_{\ell}, P_{\ell}^{\perp} : \bbR^{3} \to \bbR^{3}$ are orthogonal projections to $(\ell)$ and $(\ell)^{\perp}$, respectively. Define
\begin{equation*}
	y^{\mu} = \tensor{(\Lmb_{\ell})}{^{\mu}_{\nu}} x^{\nu} + \bfxi^{\mu},
\end{equation*}
which constitute another set of canonical coordinates. With respect to the coordinates $(y^{0}, \ldots, y^{3})$, we define $(\alp_{t, \ell, \xi}, \bt_{t, \ell, \xi})$ on $\set{y^{0} = 0}$ as (cf.~\eqref{eq:induced-data-lin})
\begin{equation} \label{eq:bump-alp'}
(\alp_{t, \ell, \xi})_{i j} = \bfalp_{ij} \big|_{\set{y^{0} = 0}}, \quad
(\bt_{t, \ell, \xi})_{i j} = \tfrac{1}{2} \left( \rd_{0} \bfalp_{i j} - \rd_{i} \bfalp_{j 0} - \rd_{j} \bfalp_{i 0} \right) \big|_{\set{y^{0} = 0}}.
\end{equation}
Note that this pair solves the linearized constraint equation, in the sense that $(h_{t, \ell, \xi},\pi_{t, \ell \xi}) = (\alp_{t, \ell, \xi} - \dlt \tr_{\dlt} \alp_{t, \ell, \xi}, \alp_{t, \ell, \xi} - \dlt \tr_{\dlt} \bt_{t, \ell, \xi})$, as in \eqref{eq:hpi-def}, solves $\vec{P}(h_{t, \ell, \xi},\pi_{t, \ell, \xi})=0$. We claim that
\begin{gather}
	\nrm{(t \rd_{t}, \rd_{\ell}, \rd_{\xi})^{n} (\alp_{t, \ell, \xi}, \bt_{t, \ell, \xi})}_{H^{s} \times H^{s-1}(\bbR^{3})}
	+ \nrm{\nb (t \rd_{t}, \rd_{\ell}, \rd_{\xi})^{n} \bfalp_{0 \mu}}_{H^{s-1}(\bbR^{3})} \aleq_{n, s, b, \Gmm} t, \label{eq:bump-alp'-est} \\
	\supp (\alp_{t, \ell, \xi}, \bt_{t, \ell, \xi}) \subseteq B_{b}(\xi). \label{eq:bump-alp'-supp}
\end{gather}
The $H^{s} \times H^{s-1}$ bound \eqref{eq:bump-alp'-est} follows from the standard energy estimate applied to $\Box (t \rd_{t}, \rd_{\ell}, \rd_{\xi})^{n} \bfalp_{\mu \nu} = 0$ in the region between the two planes $\set{x^{0} = 0}$ and $\set{y^{0} = 0}$. To verify \eqref{eq:bump-alp'-supp}, observe first that, since $\abs{x^{0}} \leq \abs{\ell}\abs{x}$ on $\set{y^{0} = 0}$, we have
\begin{equation*}
\set{y^{0} = 0} \cap \set{\abs{x} > \tfrac{1}{2} (1 - \abs{\ell})^{-1} \Gmm^{-2} b} \subseteq \set{\abs{x} > \abs{x^{0}} + \tfrac{1}{2} \Gmm^{-2} b} = \bbR^{1+3} \setminus \br{\bfOmg_{\bfalp}}.
\end{equation*}
Note also that, on $\set{y^{0} = 0}$, the coordinates $(y^{1}, y^{2}, y^{3})$ see length contraction by $\gmm^{-1}$ in the direction of $\ell$ relative to $(x^{1}, x^{2}, x^{3})$, i.e.,
\begin{equation} \label{eq:length-cont}
	 (y^{j} - \xi^{j}) \big|_{\set{y^{0} = 0}}
	= \gmm^{-1} (P_{\ell} x)^{j} + (P_{\ell}^{\perp} x)^{j} \big|_{\set{y^{0} = 0}}.
\end{equation}
In particular, $\abs{x} \geq \abs{y - \xi}$. Putting these facts together, and using $\frac{1}{2} (1-\abs{\ell})^{-1} \leq \gmm^{2} < \Gmm^{2}$, we have
\begin{equation*}
	\set{y^{0} = 0, \, \abs{y - \xi} \geq b} \subseteq \set{y^{0} = 0, \, \abs{x} \geq b} \subseteq \set{y^{0} = 0, \, \abs{x} > \tfrac{1}{2} (1-\abs{\ell})^{-1} \Gmm^{-2} b} \subseteq \bbR^{1+3} \setminus \br{\bfOmg_{\bfalp}}.
\end{equation*}
In view of \eqref{eq:bump-bfalp-supp}, $\bfalp = 0$ in an open neighborhood of $\set{y^{0} = 0, \, \abs{y - \xi} \geq b}$, which implies the desired support statement.

\smallskip
\noindent {\bf Step~3.~Construction of $(g_{t, \ell, \xi}, k_{t, \ell, \xi})$.}
We work on the hyperplane $\set{y^{0} = 0}$ using the coordinates $(y^{1}, y^{2}, y^{3})$. By \eqref{eq:bump-alp'-est}, \eqref{eq:bump-alp'-supp} and hypothesis, note that
\begin{equation*}
\supp (h_{t, \ell, \xi}, \pi_{t, \ell, \xi}) \subseteq C_{\tht}(\bfomg) \cap A_{8}, \quad
\nrm{(t \rd_{t}, \rd_{\ell}, \rd_{\xi}) (h_{t, \ell, \xi}, \pi_{t, \ell, \xi})}_{\Hb^{s, \dlt} \times \Hb^{s-1, \dlt+1}} = \calO_{s, b, \Gmm}(t),
\end{equation*}
where the weights do not play any role in the second assertion by the support property. Let $\vec{S}_{c} = (S_{c}, T_{c})$ be the conic solution operators in Lemma~\ref{lem:conic} adapted to $C_{\theta}(\bfomg)$. By the computation in Section~\ref{subsec:linearize}, in order to construct a solution to \eqref{eq:constraint} of the form $(g_{t, \ell, \xi}, k_{t, \ell, \xi}) = (\dlt + \alp_{t, \ell, \xi} + \td{\alp}, \bt_{t, \ell, \xi} + \td{\bt})$, it suffices to solve
\begin{align*}
    (\td{h},\td{\pi})= \vec{S}_{c} \vec{N}(h_{t, \ell, \xi} + \td{h}, \pi_{t, \ell, \xi} + \td{\pi}),
\end{align*}
where $(\td{h},\td{\pi}) = (\td{\alp} - \dlt \tr_{\dlt} \td{\alp}, \td{\bt} - \dlt \tr_{\dlt} \td{\bt})$.
By standard Picard iteration, there is a unique solution $(\td{h}, \td{\pi}) \in (\Hb^{s, \dlt} \times \Hb^{s-1, \dlt+1})_{0}(C_{\tht}(\bfomg))$  as long as $t$ is sufficiently small depending on $s, b, \Gmm$, such that
\begin{equation*}
\nrm{(t \rd_{t}, \rd_{\ell}, \rd_{\xi})^{n} (\td{h}, \td{\pi})}_{\Hb^{s, \dlt} \times \Hb^{s-1, \dlt+1}} \aleq_{n, s, b, \Gmm} t^{2}.
\end{equation*}
By \eqref{eq:hpi2gk}, assertions \eqref{eq:bump-single-supp} and \eqref{eq:bump-single-est} immediately follow. Moreover, since $(g_{t, \ell, \xi} - \dlt, k_{t, \ell, \xi}) = (\td{\alp}, \td{\bt})$ outside of $\supp (\alp_{t, \ell, \xi}, \bt_{t, \ell, \xi})$, \eqref{eq:bump-single-imp} follows as well.

\smallskip
\noindent {\bf Step~4.~Construction of $\bfg$ and identification with $\bfalp$ up to $\calO(t^{2})$.}
In preparation for computation of the charges for $(g_{t, \ell, \xi}, k_{t, \ell, \xi})$ on $\set{y^{0} = 0}$, we consider its Cauchy development $\bfg$. More precisely, we solve the vacuum Einstein equation $\bfG(\bfg) = 0$ in the domain $\bfOmg = \set{(y^{0}, y^{1}, y^{2}, y^{3}) \in \bbR^{1+3} : \abs{y^{0}} < 64, \, \abs{y} + 2 \abs{y^{0}} < 256}$ under the wave coordinate gauge condition $\Box_{\bfg} y^{\mu} = 0$ with the following initial data:
\begin{equation} \label{eq:bump-bfg}
	\bfg_{00} |_{\set{y^{0} = 0}} = -1 + \bfalp_{00} |_{\set{y^{0} = 0}}, \quad
	\bfg_{0j} |_{\set{y^{0} = 0}} = \bfalp_{0j} |_{\set{y^{0} = 0}},
\end{equation}
as well as $\bfg_{ij} |_{\set{y^{0} = 0}}$, $\rd_{0} \bfg_{ij} |_{\set{y^{0} = 0}}$, $\rd_{0} \bfg_{0 j} |_{\set{y^{0} = 0}}$ and $\rd_{0} \bfg_{0 0} |_{\set{y^{0} = 0}}$ determined (in order) from the initial data set $(g_{t, \ell, \xi}, k_{t, \ell, \xi})$ using \eqref{eq:induced-data} and $\Box_{\bfg} y^{\mu}  = 0$. Assuming that $t < \eps_{b, 0}$ with $\eps_{b, 0}$ sufficiently small depending on $b, \Gmm$, we may ensure the existence of $\bfg$ in $\bfOmg$ by the standard local well-posedness theory \cite[Ch.~14--15]{Rin} (see also \cite{ChBr}). Observe also that $\rd \bfOmg$ is spacelike with respect to $\bfeta$, and hence with respect to $\bfg$ provided that $t$ is sufficiently small. By \eqref{eq:bump-alp'-est}, boundedness, persistence of regularity, Cauchy stability and  Sobolev embedding, we have
\begin{equation} \label{eq:bump-bfg-est}
	\nrm*{(t \rd_{t}, \rd_{\ell}, \rd_{\xi})^{n}(\bfg - \bfeta)}_{C^{s}(\bfOmg)} \aleq_{n, s, b, \Gmm} t.
\end{equation}
To conclude this step, it remains to identify $\bfg$ with $\bfalp$ up to terms of order $t^{2}$. Observe, from \eqref{eq:bump-alp'} and our choice of initial data for $\bfg$, that $\td{\bfalp} = \bfg - \bfeta - \bfalp$ solves
\begin{gather*}
\mathrm{D}_{\bfeta} \bfG [\td{\bfalp}]_{\alp \bt} = \calO(t^{2}), \quad \nb^{\alp} \bfH[\td{\bfalp}]_{\alp \bt} = \calO(t^{2}), \\
 (\td{\bfalp}_{ij}, \rd_{0} \td{\bfalp}_{ij}) |_{\set{y^{0} = 0}} = \calO(t^{2}), \quad (\td{\bfalp}_{00}, \td{\bfalp}_{0j}) |_{\set{y^{0} = 0}} = 0, \quad (\rd_{0} \td{\bfalp}_{00}, \rd_{0} \td{\bfalp}_{0j}) |_{\set{y^{0} = 0}} = \calO(t^{2}),
\end{gather*}
where $g = \calO(t^{2})$ is a shorthand for $\nrm{(t \rd_{t}, \rd_{\ell}, \rd_{\xi})^{n} g}_{C^{s}(\bfOmg)} \aleq_{n, s, b, \Gmm} t^{2}$. It follows from Lemma~\ref{lem:lin-einstein} (and Sobolev embedding) that
\begin{equation} \label{eq:bump-tdbfalp-est}
	\nrm{(t \rd_{t}, \rd_{\ell}, \rd_{\xi})^{n} (\bfg - \bfeta - \bfalp)}_{C^{s}(\bfOmg)} \aleq_{n, s, b, \Gmm} t^{2}.
\end{equation}

\smallskip
\noindent {\bf Step~5.~Computation of charges.}
Thanks to \eqref{eq:bump-tdbfalp-est}, we may now use Lemma~\ref{lem:bump-poincare} to compute the charges of $(g_{t, \ell, \xi}, k_{t, \ell, \xi})$, which equals $(g^{(\Lmb_{\ell}, \bfxi)}, k^{(\Lmb_{\ell}, \bfxi)})$ in the notation of Lemma~\ref{lem:bump-poincare}, in terms of those of $(g, k)$, which is the induced data on $\set{x^{0} = 0}$. Moreover, in view of \eqref{eq:induced-data} with $(\Lmb, \bfxi) = (I, 0)$, \eqref{eq:bump-bfalp} and \eqref{eq:bump-tdbfalp-est}, we may appeal to Lemma~\ref{lem:bump-en} with $\mathring{\bfalp}$ as in Step~1 to compute the charges of $(g, k)$. By \eqref{eq:alp0-size}, we have
\begin{align*}
\frac{1}{16} \int \sum_{j, k, \ell} (\rd_{k} \mathring{\alp}_{j \ell} - \rd_{\ell} \mathring{\alp}_{j k})^{2} \, \ud x
&= \frac{1}{8} \int \left[ \sum_{j, k, \ell} (\rd_{j} \mathring{\alp}_{k, \ell})^{2} - \sum_{j} (\rd_{j} \tr_{\dlt} \mathring{\alp})^{2} \right] \, \ud x \\
&= t^{2} \int \left[ \frac{1}{4} \abs{\rd \rd^{2}_{12} \chi_{b, \Gmm}}^{2} + \frac{1}{16} \abs{\rd (\rd^{2}_{11} - \rd^{2}_{22}) \chi_{b, \Gmm}}^{2} \right] \, \ud x \\
&= \frac{t^{2}}{16} \int \abs{\rd (\rd^{2}_{11} + \rd^{2}_{22}) \chi}^{2}  \, \ud x  = t^{2}.
\end{align*}
Since $\set{x^{0} = 0,\, 16 \leq \abs{x} \leq 32}$ and $\set{y^{0} = 0,\, 16 \leq \abs{y} \leq 32}$ are contained in $\bfOmg$ yet disjoint from $\br{\bfOmg_{\bfalp}}$, for each $\sgm$ we may find a hypersurface $U_{\sgm}$ such that $\rd U_{\sgm} = \set{x^{0} = 0, \, \abs{x} = 16 (1+\sgm)} \cup \set{y^{0} = 0, \, 16 (1+\sgm)}$, $\br{U_{\sgm}} \cap \supp \bfalp = \0$ while $\mathrm{Vol} (U_{\sgm}) = \calO(1)$.
By Lemma~\ref{lem:bump-en}, we have
\begin{equation*}
(\bbP_{0}, \bbP_{i}, \bbJ_{\mu \nu})[(g, k); A_{16}] = (t^{2}, 0, \calO(\Gmm^{-2} b t^{2})) + \calO_{b, \Gmm} (t^{3}),
\end{equation*}
and similarly after applying of $(t \rd_{t}, \rd_{\ell}, \rd_{\xi})$. Applying Lemma~\ref{lem:bump-poincare}, \eqref{eq:spt-averge-charges} and $\nrm{\Lmb_{\ell}}_{\bbR^{4} \to \bbR^{4}} \aleq \gmm \aleq \Gmm$, we immediately arrive at \eqref{eq:bump-single-E}--\eqref{eq:bump-single-P} for the charges of $(g_{t, \ell, \xi}, k_{t, \ell, \xi})$ on $\set{y^{0} = 0}$. Moreover, \eqref{eq:bump-single-C}--\eqref{eq:bump-single-J} follow as well by estimating $\tensor{(\Lmb_{\ell})}{_{\mu}^{\mu'}} \tensor{(\Lmb_{\ell})}{_{\nu}^{\nu'}} \calO(\Gmm^{-2} b t^{2}) = \calO(b t^{2})$. \qedhere
\end{proof}

\subsection{Multi-bump configurations with prescribed charges}
Our first goal is to establish the following.
\begin{proposition}[Almost time-symmetric four-bump configuration] \label{prop:bump-four}
Fix $-1 < \dlt < - \frac{1}{2}$ and $\tht \in (0, \frac{\pi}{2})$, and define $C^{(4)}_{\tht} = C_{\tht}(\bfe_{1}) \cup C_{\tht}(-\bfe_{1}) \cup C_{\tht}(\bfe_{2}) \cup C_{\tht}(-\bfe_{2})$. For $s > \frac{3}{2}$, there exists $\eps_{b, 1} = \eps_{b, 1}(s) > 0$ such that the following holds. Consider
\begin{equation*}
	\calU = \set{Q \in \bbR^{10} : 0 < \bfE < 2 \eps_{b, 1}^{2}, \, \abs{\bfC} + \abs{\bfP} + \abs{\bfJ} < 2 \eps_{b, 1}^{2} \bfE}.
\end{equation*}
For any $Q \in \calU$, there exists $(g_{Q}^{\rm bump}, k_{Q}^{\rm bump}) \in C^{\infty}(\bbR^{3})$ such that
    \begin{align*}
        \supp (g_Q^{\rm bump}-\delta,k_Q^{\rm bump})\subseteq C^{(4)}_{\tht} \cap\{|x|>8\}\quad \text{and}\quad \bfQ[(g_Q^{\rm bump},k_Q^{\rm bump}); A_{16}]=Q.
        \end{align*}
       Moreover, the map $Q\mapsto (g_Q^{\rm bump}-\delta,k_Q^{\rm bump})$ is $C^{1}$ in the $H^{s,\delta}_{\rm b}\times H_{\rm b}^{s-1,\delta+1}$ topology, and
\begin{align}\label{eq:size-bump}
    \|(g_Q^{\rm bump}-\delta,k_Q^{\rm bump})\|_{H^{s,\delta}_{\rm b}\times H_{\rm b}^{s-1,\delta+1}}&\aleq_s \sqrt{\bfE}, \\
    \|\rd_{Q} (g_Q^{\rm bump},k_Q^{\rm bump})\|_{H^{s,\delta}_{\rm b}\times H_{\rm b}^{s-1,\delta+1}} &\aleq_{s} \frac{1}{\sqrt{\bfE}}. \label{eq:Lip-bump}
\end{align}
Moreover, in $\set{\abs{x} > 16}$, we have the improved bounds
\begin{align}
    \|(g_Q^{\rm bump} - \dlt,k_Q^{\rm bump})\|_{H^{s,\delta}_{\rm b}\times H_{\rm b}^{s-1,\delta+1}(\set{\abs{x} > 16})} & \aleq_{s} \bfE, \label{eq:size-extr-bump} \\
    \|\rd_{Q} (g_Q^{\rm bump},k_Q^{\rm bump})\|_{H^{s,\delta}_{\rm b}\times H_{\rm b}^{s-1,\delta+1}(\set{\abs{x} > 16})}&\lesssim_s 1. \label{eq:Lip-extr-bump}
\end{align}
\end{proposition}
\begin{proof}
We would like to consider a superposition of $4$ bumps with disjoint supports and use the inverse function theorem around $\bfE>0, \bfC=\bfP=\bfJ=0$. For $\br{B_{b}(\xi)} \subseteq C_{\tht}(\bfe_1) \cap A_{8}$, $\br{B_{b}(\eta)} \subseteq C_{\tht}(\bfe_2) \cap A_{8}$, let
$$\supp(g_{t,\ell,\xi}-\delta,k_{t,\ell,\xi})\subseteq C_\tht(\bfe_1),\quad\supp(g_{t,m,\eta}-\delta,k_{t,m,\eta})\subseteq C_\tht(\bfe_2)$$
be the single bumps constructed in Proposition~\ref{prop:bump-single}. We define
\begin{align*}
    (g^*_{t,\ell^*,\xi^*}(x),k^*_{t,\ell^*,\xi^*}(x))&:=(g_{t,\ell^*,-\xi^*}(-x),-k_{t,\ell^*,-\xi^*}(-x))\in C^\infty(C_\tht(-\bfe_1)),\\
    (g^*_{t,m^*,\eta^*}(x),k^*_{t,m^*,\eta^*}(x))&:=(g_{t,m^*,-\eta^*}(-x),-k_{t,m^*,-\eta^*}(-x))\in C^\infty( C_\tht(-\bfe_2) )
\end{align*}
be the dual bumps defined by inversion symmetry.
Since the solutions have disjoint supports, we define the superposition of $4$ bumps
\begin{equation}
    (g^{(4)}-\delta, k^{(4)}):=(g_{t,\ell,\xi}-\delta,k_{t,\ell,\xi})+(g_{t,m,\eta}-\delta,k_{t,m,\eta})+(g^*_{t,\ell^*,\xi^*}-\delta,k^*_{t,\ell^*,\xi^*})+(g^*_{t,m^*,\eta^*}-\delta,k^*_{t,m^*,\eta^*}),
\end{equation}
which solves \eqref{eq:constraint}. Now we get a family with $25$ parameters and we are interested in the charges
\begin{equation}
    \bfQ:(t,\ell,\ell^*,\xi,\xi^*,m,m^*,\eta,\eta^*)\mapsto \bfQ[(g^{(4)},k^{(4)});A_{16}].
\end{equation}
We need to choose $10$ of them with a uniform non-degenerate Jacobian. Intuitively, we will choose $t$ for the energy, $\xi$ for the central charge, $\ell$ for the momentum and three parameters among $(\ell^*,m)$ for the angular momentum.

We note at $\ell=0$ we have $k_{t,\ell,\xi}=0$ and thus $$\bfP[(g_{t,0,\xi},k_{t,0,\xi});A_{16}]=\bfJ[(g_{t,0,\xi},k_{t,0,\xi});A_{16}]=0.$$
Moreover, by the inversion symmetry, for $\xi+\xi^*=0$, $\eta+\eta^*=0$, we have
$$\bfC[(g^{(4)},k^{(4)});A_{16}]=0.$$

Now we fix $\xi^*,\eta,\eta^*,m^*$ and compute the Jacobian of the family $(t,\xi,\ell,\ell^*,m)\mapsto \bfQ$ with  $13$ parameters at the point $\ell=\ell^*=m=m^*=0$ and $\xi+\xi^*=\eta+\eta^*=0$:
\begin{align}\label{eq:bump-jacobian}
        \begin{pmatrix}
    t\partial_t\bfE&\partial_\xi \bfE&\partial_\ell\bfE&\partial_{\ell^*} \bfE&\partial_{m}\bfE\\
    t\partial_t\bfC&\partial_\xi \bfC&\partial_\ell\bfC&\partial_{\ell^*} \bfC&\partial_{m}\bfC\\
    t\partial_t\bfP&\partial_\xi \bfP&\partial_\ell\bfP&\partial_{\ell^*} \bfP&\partial_{m}\bfP\\
    t\partial_t\bfJ&\partial_\xi \bfJ&\partial_\ell\bfJ&\partial_{\ell^*} \bfJ&\partial_{m}\bfJ
    \end{pmatrix}= \begin{pmatrix}
    8t^2  & 0 &0 &0 & 0 \\
    0&t^2&0 &0&0\\
    0&0 & t^2&t^2 & t^2\\
     0& 0 & t^2\tensor{\eps}{_i^{jk}}\xi_j & t^{2} \tensor{\eps}{_i^{jk}}\xi^*_j&t^2\tensor{\eps}{_i^{jk}}\eta_j \end{pmatrix}+\begin{pmatrix}
         \mathcal{O}_{b}(t^3)\\ \mathcal{O}(bt^2)+\mathcal{O}_{b}(t^3)\\
         \mathcal{O}_{b}(t^3)\\ \mathcal{O}(bt^2)+\mathcal{O}_{b}(t^3)
     \end{pmatrix}.
\end{align}
If we subtract the $\ell$ columns from the $\ell^*$ columns, and then half the $\ell^*$ column and add to the $\ell$ column, and finally subtract the $\ell$ column from the $m$ column, we would get
\begin{align*}
\begin{pmatrix}
    8t^2  & 0 &0 &0 & 0 \\
    0&t^2&0 &0&0\\
    0&0 & t^2&0 & 0\\
     0& 0 & 0 & t^{2} \tensor{\eps}{_i^{jk}}\xi^*_j&t^2\tensor{\eps}{_i^{jk}}\eta_j \end{pmatrix}+\begin{pmatrix}
         \mathcal{O}_{b}(t^3)\\ \mathcal{O}(bt^2)+\mathcal{O}_{b}(t^3)\\
         \mathcal{O}_{b}(t^3)\\ \mathcal{O}(bt^2)+\mathcal{O}_{b}(t^3)
     \end{pmatrix}.
\end{align*}
Let $\xi=(12,0,0)$ so that $\xi^*=(-12,0,0)$, and $\eta=(0,12,0)$, then the $3\times 6$ minor in the bottom-right corner of the first matrix is
\begin{align*}
    t^2(\tensor{\eps}{_i^{jk}}\xi^*_j, \tensor{\eps}{_i^{jk}}\eta_j )=t^2\begin{pmatrix}
        0&0&0&0&0&12\\
        0&0&12&0&0&0\\
        0&-12&0&-12&0&0
    \end{pmatrix}
\end{align*}
We can then select the variables $\ell^*_2,\ell^*_3,m_3$ so that
$\displaystyle\frac{\partial(\bfE,\bfC,\bfP,\bfJ)}{\partial(t,\xi,\ell,\ell^*_2,\ell^*_3,m_3)}$
is non-degenerate once we choose $t\ll_b 1$ and $b\ll 1$.


Now we fix a sufficiently small $b>0$ and use the inverse function theorem to conclude the proposition. We denote $\Theta=(\xi,\ell,\ell^*_2,\ell^*_3,m_3)$ and fix $\xi^*=(-12,0,0)$, $\eta=(0,12,0)$, $\eta^*=(0,-12,0)$, $\ell^*_1=m_1=m_2=0$ and $m^*=0$, and apply the inverse function theorem over the region
\begin{align*}
    \mathcal{U}_{\eps_1}:=\{(1-\eps_{b, 1}^{2})\eps_1^2<\bfE<(1+\eps_{b, 1}^{2})\eps_1^2, \, |\bfC|+|\bfP|+|\bfJ|<10\eps_{b, 1}^{2} \eps_1^2\}.
\end{align*}

By \eqref{eq:bump-jacobian}, for $\eps_1$ sufficiently small, the map $(\tau,\Theta)\mapsto \eps_1^{-2}\bfQ(\eps_1 \tau,\Theta)$
is a diffeomorphism with uniformly non-degenerate Jacobian near the point $$\eps_1^{-2}\bfE= 1,\quad \eps_1^{-2}\bfC=\eps_1^{-2}\bfP=\eps_1^{-2}\bfJ=0,\quad \eps_1\ll_b 1, b\ll 1.$$
More precisely, there is a uniform (i.e. independent of $\eps_1$) neighbourhood $\mathcal{V}$ near $\tau=1/2,\Theta=0$ such that
$$(\tau,\Theta)\mapsto \eps_1^{-2}\bfQ(\eps_1 \tau,\Theta):\mathcal{V}\to \eps_1^{-2}\mathcal{U}_{\eps_1}$$
is a diffeomorphism with uniformly bounded inverse in $C^1$.
Let $\td{Q}\mapsto (\tau,\Theta)$ be the inverse map so that
\begin{align*}
    \td{Q}\mapsto (\tau,\Theta)\mapsto \eps_1^{-2}\bfQ(\eps_1 \tau, \Theta)=\td{Q},\quad \td{Q}\in\eps_1^{-2}\mathcal{U}_{\eps_1}.
\end{align*}
Then we have
\begin{align*}
    \bfQ[(g,k)_{(\eps_1\tau,\Theta)_{\td{Q}}};A_{16}]=\eps_1^2\td{Q},\quad \td{Q}\in \eps_1^{-2}\mathcal{U}_{\eps_1}.
\end{align*}
Taking $Q=\eps_1^2\td{Q}$, we conclude the construction of a reparametrized family of $(g_Q^{\rm bump}, k_{Q}^{\rm bump})$ such that $\bfQ[(g_Q^{\rm bump}, k_{Q}^{\rm bump});A_{16}]=Q$ for any $Q\in\mathcal{U}_{\eps_1}$. It also follows from the construction that $Q\mapsto(g_Q^{\rm bump}-\delta, k_{Q}^{\rm bump})$ is smooth in the $H^{s,\delta}_{\rm b}\times H_{\rm b}^{s-1,\delta+1}$ topology.
Note that by varying $\eps_1$, we can cover $\mathcal{U}$ with
$$\mathcal{U}\subset \bigcup\limits_{0<\eps_1<10\eps_{b,1}}\mathcal{U}_{\eps_1}.$$
By the uniqueness of the inverse function theorem, the $\eps_1$-family $\{g_Q:Q\in \mathcal{U}_{\eps_1}\}$ glues together to a smooth family $g_Q^{\rm bump}$ for $Q\in \mathcal{U}$.

Moreover, by Proposition~\ref{prop:bump-single}, we have for $n=0,1$ (recall $t=\eps_1\tau \aeq \eps_1$),
\begin{align*}
\|(t\partial_t,\partial_{\Theta})^{(n)}(g_Q^{\rm bump}-\delta, k_{Q}^{\rm bump})\|_{H_{\rm b}^{s,\delta}\times H_{\rm b}^{s-1,\delta+1}}&=\mathcal{O}(t)=\mathcal{O}(\sqrt{\bfE}), \end{align*}
as well as the following chain of estimates:
\begin{align*}
    \|\partial_Q(g_Q^{\rm bump}, k_{Q}^{\rm bump})\|_{H_{\rm b}^{s,\delta}\times H_{\rm b}^{s-1,\delta+1}}&= \eps_1^{-2}\|\partial_{\td{Q}}(g_Q^{\rm bump}, k_{Q}^{\rm bump})\|_{H_{\rm b}^{s,\delta}\times H_{\rm b}^{s-1,\delta+1}}\\
    &\lesssim \eps_1^{-2}\|\partial_{(\tau, \Tht)}(g_Q^{\rm bump}, k_{Q}^{\rm bump})\|_{H_{\rm b}^{s,\delta}\times H_{\rm b}^{s-1,\delta+1}}\\
    &=\mathcal{O}(\eps_{1}^{-1}) = \mathcal{O}(\sqrt{\bfE}^{-1}).
\end{align*}
We are left to check the uniform $C^{1}$-dependence on $Q$ in $\{|x|>16\}$. By \eqref{eq:bump-single-imp}, for $n=0,1$,
\begin{align*}
    \|(t\partial_t,\partial_{\Theta})^{(n)}(g_Q^{\rm bump}-\delta,k_Q^{\rm bump})\|_{H_{\rm b}^{s,\delta}\times H_{\rm b}^{s-1,\delta+1}(\set{\abs{x} > 16})}=\mathcal{O}(t^2)=\mathcal{O}(\bfE).
\end{align*}
Thus
\begin{align*}
    \|\partial_Q(g_Q^{\rm bump},k_Q^{\rm bump})\|_{H_{\rm b}^{s,\delta}\times H_{\rm b}^{s-1,\delta+1}(\set{\abs{x} > 16})}&= \eps_1^{-2}\|\partial_{\td{Q}}(g_Q^{\rm bump},k_Q^{\rm bump})\|_{H_{\rm b}^{s,\delta}\times H_{\rm b}^{s-1,\delta+1}(\set{\abs{x} > 16})}\\
    &\lesssim\eps_{1}^{-2} \|\partial_{(\tau,\Theta)}(g_Q^{\rm bump},k_Q^{\rm bump})\|_{H_{\rm b}^{s,\delta}\times H_{\rm b}^{s-1,\delta+1}(\set{\abs{x} > 16})} = \mathcal{O}(1). \qedhere
\end{align*}
\end{proof}

To extend the range of allowed $\bfP$, we adjoin to the four-bump configuration two more bumps that carry most of the total linear momentum.
\begin{proposition}[Six-bump configuration] \label{prop:bump-six}
Fix $-1 < \dlt < - \frac{1}{2}$ and $\tht \in (0, \frac{\pi}{2})$, and define $C^{(6)}_{\tht} = C^{(4)}_{\tht} \cup C_{\tht}(\bfe_{3}) \cup C_{\tht}(-\bfe_{3})$. Given $s > \frac{3}{2}$ and $\Gmm \geq 1$, there exist $\eps_{b} = \eps_{b}(s, \Gmm) > 0$ and $\mu_{b} = \mu_{b}(s, \Gmm)$ such that the following holds. Consider
\begin{gather*}
\calU_{\Gmm} = \set*{Q \in \bbR^{10}:
	\bfE > \abs{\bfP}, \quad
	\tfrac{\bfE}{\sqrt{\bfE^{2} - \abs{\bfP}^{2}}} < 2 \Gmm, \quad
	\bfE < \eps_{b}^{2}, \quad
	\abs{\bfC} + \abs{\bfJ} < \mu_{b} \bfE}.
\end{gather*}
For each $Q \in \calU_{\Gmm}$, there exists $(g_Q^{\rm bump; \Gmm},k_Q^{\rm bump; \Gmm})\in C^\infty(\bbR^{3})$ such that
    \begin{align*}
 \supp       (g_Q^{\rm bump; \Gmm} - \dlt,k_Q^{\rm bump; \Gmm}) \subseteq C_{\tht}^{(6)} \cap \set{\abs{x} > 8} \quad \text{and}\quad
       \bfQ[(g_Q^{\rm bump; \Gmm},k_Q^{\rm bump; \Gmm}); A_{16}]=Q.
    \end{align*}
Moreover,
\begin{align}\label{eq:size-bump-Gmm}
    \|(g_Q^{\rm bump; \Gmm}-\delta,k_Q^{\rm bump; \Gmm})\|_{H^{s,\delta}_{\rm b}\times H_{\rm b}^{s-1,\delta+1}}&\aleq_s \sqrt{\bfE}, \\
    \|\rd_{Q} (g_Q^{\rm bump; \Gmm},k_Q^{\rm bump; \Gmm})\|_{H^{s,\delta}_{\rm b}\times H_{\rm b}^{s-1,\delta+1}} &\aleq_{s} \frac{1}{\sqrt{\bfE}}. \label{eq:Lip-bump-Gmm}
\end{align}
 Moreover, in $\set{\abs{x} > 16}$, we have the improved bounds
\begin{align}
    \|(g_Q^{\rm bump; \Gmm} - \dlt,k_Q^{\rm bump; \Gmm})\|_{H^{s,\delta}_{\rm b}\times H_{\rm b}^{s-1,\delta+1}(\set{\abs{x} > 16})} & \aleq_{s} \bfE, \label{eq:size-extr-bump-Gmm} \\
    \|\rd_{Q} (g_Q^{\rm bump; \Gmm},k_Q^{\rm bump; \Gmm})\|_{H^{s,\delta}_{\rm b}\times H_{\rm b}^{s-1,\delta+1}(\set{\abs{x} > 16})}&\lesssim_s 1. \label{eq:Lip-extr-bump-Gmm}
\end{align}
\end{proposition}

\begin{proof}
We first apply Proposition~\ref{prop:bump-single} to place two boosted bumps in $C_{\tht}(\bfe_{3}) \cup C_{\tht}(-\bfe_{3})$ that carries most of $\bfE$ and all of $\bfP$, then use Proposition~\ref{prop:bump-four} to correct the remaining charges. To be more precise, let $(g_{t,\ell,\xi},k_{t,\ell,\xi})$ be the boosted bump supported in $C_\tht(\bfe_3)\cap \{|x|>8\}$ constructed in Proposition~\ref{prop:bump-single} and define the dual bump by
$$(g^*_{t,\ell^*,\xi^*}(x),k^*_{t,\ell^*,\xi^*}(x)):=(g_{t,\ell^*,-\xi^*}(-x),-k_{t,\ell^*,-\xi^*}(-x))\in C^\infty(C_\tht(-\bfe_3)).$$
Since the two bumps have disjoint supports, the superposition
\begin{align}
    (g^{(2)}_{t,\ell,\xi}-\delta,k^{(2)}_{t,\ell,\xi}):=(g_{t,\ell,\xi}-\delta,k_{t,\ell,\xi})+(g^*_{t,\ell,-\xi}-\delta,k^*_{t,\ell,-\xi})
\end{align}
solves \eqref{eq:constraint}. By inversion symmetry, we have
\begin{align*}
    \bfC[(g^{(2)}_{t,\ell,\xi},k^{(2)}_{t,\ell,\xi});A_{16}]=\bfJ[(g^{(2)}_{t,\ell,\xi},k^{(2)}_{t,\ell,\xi});A_{16}]=0.
\end{align*}
Moreover, we have
\begin{align*}
    (t\partial_t,\partial_\ell,\partial_\xi)^{(n)}((\bfE,\bfP)[(g^{(2)}_{t,\ell,\xi},k^{(2)}_{t,\ell,\xi});A_{16}]-2\gamma(\ell)t^2(1,\ell))=\mathcal{O}_{b,\Gamma}(t^3).
\end{align*}
Using inverse function theorem over the region
$$\mathcal{U}^{\bfE,\bfP}_\Gamma:=\set*{(\bfE,\bfP)\in \mathbb{R}^4:|\bfP|<\bfE<10\eps_b^2, \tfrac{\bfE}{\sqrt{\bfE^2-|\bfP|^2}}<10\Gamma},$$
it is easy to see we have a $4$-parameter family $(g^{(2)}_{\bfE,\bfP},k^{(2)}_{\bfE,\bfP})$ for $(\bfE,\bfP)\in \mathcal{U}^{\bfE,\bfP}_\Gamma$ with
\begin{align}
    (\bfE,\bfP)[(g^{(2)}_{\bfE,\bfP},k^{(2)}_{\bfE,\bfP});A_{16}]=(\bfE,\bfP),\quad (\bfC,\bfJ)[(g^{(2)}_{\bfE,\bfP},k^{(2)}_{\bfE,\bfP});A_{16}]=0.
\end{align}
Now we want to add the almost time-symmetric four-bump configuration in Proposition~\ref{prop:bump-four} to $(g^{(2)},k^{(2)})$:
\begin{equation}
    (g^{(6)}-\delta,k^{(6)}):=(g^{(2)}_{\td{E},\td{P}}-\delta,k^{(2)}_{\td{E},\td{P}})+(g^{\rm bump}_Q-\delta,k^{\rm bump}_Q),
\end{equation}
which again solves \eqref{eq:constraint}. The charges are given by
\begin{equation}
    (\bfE,\bfP)[(g^{(6)}-\delta,k^{(6)});A_{16}]=(\td{E}+E,\td{P}+P),\quad (\bfC,\bfJ)[(g^{(6)}-\delta,k^{(6)});A_{16}]=(C,J),
\end{equation}
for $(\td{E},\td{P})\in \mathcal{U}_\Gamma^{\bfE,\bfP}$ and $Q\in\mathcal{U}$. For $(\bfE,\bfP,\bfC,\bfJ)\in\mathcal{U}_\Gamma$, we take
\begin{equation}
\td{E}=\bfE\sqrt{1-(10\Gamma)^{-2}},\, E=\bfE-\td{E},\,\td{P}=\bfP,\,P=0,\,C=\bfC,\, J=\bfJ.
\end{equation}
Then it is easy to check $(\td{E},\td{P})\in \mathcal{U}_\Gamma^{\bfE,\bfP}$ and $Q=(E,P,C,J)\in\mathcal{U}$ for $\eps_b>0,\mu_b>0$ small enough. The estimates follow from Proposition~\ref{prop:bump-single} and Proposition~\ref{prop:bump-four}.
\end{proof}

\subsection{Extension procedures}
In this subsection, we prove extension lemmas for $(g_{in}, k_{in})$ and $(g_{out}, k_{out})$.
\begin{lemma} \label{lem:ext-outgoing}
Let $(g_{in}, k_{in}) \in H^{s} \times H^{s-1}(A_{1})$ be a pair solving \eqref{eq:constraint}. Fix any $-1 < \dlt < -\frac{1}{2}$, $\tht > 0$ and $\bfomg_{0} \in \bbS^{2}$. Provided that $\nrm{(g_{in}-\dlt, k_{in})}_{H^{s} \times H^{s-1}(A_{1})}$ is sufficiently small, there exists $(\td{g}_{in}, \td{k}_{in})$ solving \eqref{eq:constraint} on $\bbR^{3}$ such that
\begin{gather*}
(\td{g}_{in}, \td{k}_{in}) = (g_{in}, k_{in}) \hbox{ in } A_{1}, \quad
\supp (\td{g}_{in} - \dlt, \td{k}_{in}) \subseteq B_{4} \cup C_{\tht}(\bfomg_{0})  \\
\nrm{(\td{g}_{in}-\dlt, \td{k}_{in})}_{\Hb^{s, \dlt} \times \Hb^{s-1, \dlt+1}} \aleq_{s, \dlt, \tht} \nrm{(g_{in}-\dlt, k_{in})}_{H^{s} \times H^{s-1}(A_{1})}.
\end{gather*}
\end{lemma}

\begin{proof} {\bf Step~1.} The first step is to construct a solution operator $\vec{S}_{out}$ for $\vec{P}$ (i.e., $\vec{P} \vec{S}_{out} = I$) with the mapping property
\begin{equation*}
\vec{S}_{out} : (\Hb^{s-2,\delta+2})_{0}(A_2\cup C_\tht(\bfomg_0))\to (\Hb^{s,\delta}\times \Hb^{s-1,\delta+1})_{0}(A_2\cup C_\tht(\bfomg_0)).
\end{equation*}

We begin by constructing a Bogovskii-type operator $\vec{S}_B$ for the region $\Omega:=(A_2\cup C_\tht(\bfomg_0)) \cap \{|x|<8\}$ such that $\vec{P}\vec{S}_B f =f $ in $\Omega\setminus C_\tht(\bfomg_0)$ for $f\in C_c^\infty(\Omega)$. To achieve this, we carry out an argument similar to the proof of Lemma~\ref{lem:bogovskii} above. We begin with covering $\Omg$ by finitely many regions $\Omg_j, j=0,1,\cdots,J$ star-shaped with respect to small balls $B_j\subset \Omg_j$ such that $B_0\subset C_\tht(\bfomg_0)$ and the following joins of consecutive balls lie in $\Omg$:
\begin{align*}
        {\rm ch}(B_j,B_{j+1}):=\{tx+(1-t)y:x\in B_j, y\in B_{j+1}, t\in[0,1]\}\subset \Omg,\quad j=0,1,\cdots,J-1.
    \end{align*}
    Let $\vec{S}_j$ be the Bogovskii-type operator on $\Omg_j$ defined in Lemma~\ref{lem:bogovskii-0} and $\chi_j$ be a partition of unity with respect to $\Omega_j$. First of all, $u_j=\vec{S}_j(\chi_j f)$ solves the equation $\vec{P}u_j=\chi_j f$ outside $B_j$.  Note $\vec{S}_j$ is also well-defined on the star-shaped region ${\rm ch}(B_j,B_{j+1})$,
    we correct the  errors inductively:
    \begin{equation*}
    \begin{aligned}
        &v_J:=0, \quad w_J:=0,\quad v_{j-1}:=\vec{S}_{j-1} (\chi_j f-\vec{P}u_j+w_j)\in C_c^\infty({\rm ch}(B_{j-1},B_j)),\\
        &w_{j-1}:=(\chi_j f-\vec{P}u_j+w_j)-\vec{P}v_{j-1}\in C_c^\infty(B_{j-1}),\quad j=1,2,\cdots,J.
    \end{aligned}
    \end{equation*}
Once we define $u:=\sum\limits_{j=0}^J(u_j+v_j)$, it is easy to see $\vec{P} v_k=0$ and $\vec{P}u=f$ outside $\cup B_j$. For $x\in B_k, k=1,2,\cdots,J$, we note $\vec{P}v_j(x)=0$ for $j\neq k, k-1$ and thus
\begin{align*}
    \vec{P}u(x)&=\sum\limits_{j\neq k} \chi_j f+\vec{P}u_k(x)+\vec{P}v_k(x)+\vec{P}v_{k-1}(x)\\
    &=\sum\limits_{j\neq k} \chi_j f+\vec{P}u_k(x)+\vec{P}v_k(x)+\chi_k f-\vec{P}u_k(x)+w_k(x)\\
    &=f(x)+(\chi_{k+1}f-\vec{P}u_{k+1}+w_{k+1})(x)=f(x).
\end{align*}
We conclude that $\vec{P}u=f$ outside $B_0\subset C_\tht(\bfomg_0)$.

Next, let $\vec{S}_{c}$ be the conic solution operator adapted to the cone $C_\tht(\bfomg_0)$ so that $\vec{P}\vec{S}_{c}f=f$ for $f\in C_c^\infty( C_\tht(\bfomg_0) )$. Let $\chi\in C_c^\infty(\{|x|<8\})$ be a cutoff function such that $\chi(x)=1$ near $A_2$. We then define
$$\vec{S}_{out} f := \vec{S}_c(f-P\vec{S}_B (\chi f))+\vec{S}_B (\chi f): (\Hb^{s-2,\delta+2})_{0}(A_2\cup C_\tht(\bfomg_0))\to (\Hb^{s,\delta}\times \Hb^{s-1,\delta+1})_{0}(A_2\cup C_\tht(\bfomg_0)),$$
which obeys $\vec{P}\vec{S}_{out} f =f$ for any $f\in C_c^\infty(A_2\cup C_\tht(\bfomg_0))$ as desired.

\smallskip
\noindent {\bf Step~2.} Let $(h_{in},\pi_{in})$ be the initial data corresponding to $(g_{in},k_{in})$ under the transformation \eqref{eq:hpi-def}, and we can freely extend $(h_{in},\pi_{in})$ to $H^s_{0}\times H^{s-1}_{0}(\set{|x|<4})$ (with control on the norm). We then consider the following fixed point problem:
\begin{align*}
    (\td{h},\td{\pi})=\vec{S}_{out}(\vec{N}( h_{in}+\td{h}, \pi_{in}+\td{\pi})-\vec{P}(h_{in},\pi_{in})).
\end{align*}
By the Banach fixed point theorem, there exists a unique solution $(\td{h},\td{\pi})\in (\Hb^{s,\delta}\times \Hb^{s-1,\delta+1})_{0}(A_2\cup C_\tht(\bfomg_0))$ such that $\|(\td{h},\td{\pi})\|_{H_{\rm b}^{s,\delta}\times H_{\rm b}^{s-1,\delta+1}}\lesssim \|(h_{in},\pi_{in})\|_{H^s\times H^{s-1}(A_1)}$.
Define $(\td{g}_{in},\td{k}_{in})$ to be the solution corresponding to $(h_{in}+\td{h},\pi_{in}+\td{\pi})$. Since $(\td{h},\td{\pi})$ is supported outside $A_{1}$, we have $(\td{g}_{in},\td{k}_{in})=(g_{in},k_{in})$ in $A_{1}$.
\end{proof}

\begin{lemma} \label{lem:ext-ingoing}
Let $(g_{out}, k_{out}) \in H^{s} \times H^{s-1}(A_{2 R})$ be a pair solving \eqref{eq:constraint}. If $\nrm{(g_{out}-\dlt, k_{out})}_{H^{s} \times H^{s-1}(A_{2 R})}$ is sufficiently small, there exists $(\td{g}_{out}, \td{k}_{out})$ solving \eqref{eq:constraint} on $\td{A}_{R}$ such that
\begin{gather*}
(\td{g}_{out}, \td{k}_{out}) = (g_{out}, k_{out}) \hbox{ in } A_{2R}, \quad
\nrm{(\td{g}_{out}-\dlt, \td{k}_{out})}_{\dot{\calH}^{s} \times \dot{\calH}^{s-1}(\td{A}_{R})} \aleq_{s} \nrm{(g_{out}-\dlt, k_{out})}_{H^{s} \times H^{s-1}(A_{2 R})}.
\end{gather*}
\end{lemma}
\begin{proof}
    We may assume $R=1$ by rescaling. First we may freely extend $(g_{out},k_{out})$ to $H^s(\set{|x|\leq 4})$ (with control on norm). Let $\vec{S}_{in}$ be the Bogovskii operator on $B_2$ as in Lemma \ref{lem:bogovskii-0} with $\eta\in C_c^\infty(\{|x|<1/4\})$. We consider the following fixed point problem:
    \begin{align*}
        (\td{h},\td{\pi})=\vec{S}_{in}(\vec{N}(h_{in}+\td{h},\pi_{in}+\td{\pi})-\vec{P}(h_{in},\pi_{in})).
    \end{align*}
By the Banach fixed point theorem, there exists a unique solution $(\td{h},\td{\pi})\in (H^s\times H^{s-1})_{0}(B_2)$. Define $(\td{g}_{out},\td{k}_{out})$ to be the solution corresponding to $(h_{out}+\td{h},\pi_{out}+\td{\pi})$. By the support of $\eta$, $(\td{g}_{out},\td{k}_{out})$ solves the \eqref{eq:constraint} on $\td{A}_1$. Since $\supp (\td{h},\td{\pi}) \subseteq B_{2}$, we have $(\td{g}_{out},\td{k}_{out})=(g_{out},k_{out})$ in $A_2$.
\end{proof}

\subsection{Proof of Theorem~\ref{thm:obs-free-unit}} \label{subsec:obs-free-unit}

Fix $-1 < \dlt < - \frac{1}{2}$ and $\tht \in (0, \frac{\pi}{8})$. Let $\calU_{\Gmm}(\Dlt \bfE) = \set{\Dlt Q' \in \calU_{\Gmm} : \frac{1}{2} \Dlt \bfE \leq \bfE(\Dlt Q') \leq 2 \Dlt \bfE}$. Given $\Dlt Q' \in \calU_{\Gmm}(\Dlt \bfE)$, by Proposition~\ref{prop:bump-six}, there exists a six-bump initial data $(g_{\Dlt Q'}^{\rm bump; \Gmm}-\delta,k_{\Dlt Q'}^{\rm bump; \Gmm})$ supported in $C_{\tht}^{(6)} \setminus \br{B_{4}}$. By Lemma~\ref{lem:ext-outgoing}, given any $\bfomg_{0} \in \bbS^{2}$, there also exists $(\td{g}_{in}-\delta,\td{k}_{in}) \in \Hb^{s, \dlt} \times \Hb^{s-1, \dlt+1}$ extending $(g_{in} - \dlt ,k_{in})$ which is supported in $B_{4} \cup C_{\tht}(\bfomg_{0})$. By an appropriate choice of $\bfomg_{0}$, we may ensure that $C_{\tht}^{(6)}$ and $C_{\tht}(\bfomg_{0})$ have disjoint closures, so that the same holds for $C_{\tht}^{(6)} \setminus \br{B_{4}}$ and $C_{\tht}(\bfomg_{0}) \cup B_{4}$. We may thus define
\begin{align*}
    (g_{\Dlt Q'}-\delta,k_{\Dlt Q'}):=(\td{g}_{in}-\delta,\td{k}_{in})+(g_{\Dlt Q'}^{\rm bump; \Gmm}-\delta,k_{\Dlt Q'}^{\rm bump; \Gmm}),
\end{align*}
which extends $(g_{in} - \dlt, k_{in})$, belongs to $\Hb^{s, \dlt} \times \Hb^{s-1, \dlt+1}(\bbR^{3})$ and solves \eqref{eq:constraint} by the disjoint support property.
Moreover, by Lemma~\ref{lem:charge}, we have
\begin{align*}
    \bfQ[(g_{\Dlt Q'},k_{\Dlt Q'});A_{16}]
    &=\bfQ[(g_{\Dlt Q'}^{\rm bump; \Gmm},k_{\Dlt Q'}^{\rm bump; \Gmm});A_{16}]+\bfQ[(\td{g}_{in},\td{k}_{in});A_{16}] \\
    &=\Dlt Q'+\bfQ[(g_{in},k_{in});A_1]+\mathcal{O}_{\Gmm}(\mu_o \Delta\bfE),
 \end{align*}
while
 \begin{align*}
   \bfQ[(g_{\Dlt Q'},k_{\Dlt Q'});A_{16}] - \bfQ[(g_{\Dlt Q''},k_{\Dlt Q''});A_{16}]
   = \Dlt Q' - \Dlt Q''.
\end{align*}
In particular, the map $T: \calU_{\Gmm}(\Dlt \bfE) \to \bbR^{10}$, $\Dlt Q' \mapsto \bfQ[(g_{\Dlt Q'}, k_{\Dlt Q'}); A_{16}]$ is bi-Lipschitz.

Next, by Lemma~\ref{lem:ext-ingoing}, there exists $(\td{g}_{out},\td{k}_{out})$ on $\td{A}_{16\Gamma}$ extending $(g_{out},k_{out})$. Again by Lemma~\ref{lem:charge},
\begin{align*}
    \bfQ[(\td{g}_{out},\td{k}_{out});A_{16}]=\bfQ[(g_{out},k_{out});A_{32}]+\mathcal{O}_{\Gmm}(\mu_o \Delta\bfE).
\end{align*}

By the preceding identities and the definition of $\Dlt Q$, we have
\begin{align*}
T(\Dlt Q') = \bfQ[(g_{\Dlt Q'},k_{\Dlt Q'});A_{16}] = \bfQ[(\td{g}_{out},\td{k}_{out});A_{16}] + (\Dlt Q' -\Dlt Q) +\mathcal{O}_{\Gmm}(\mu_o \Delta\bfE).
\end{align*}
Choosing $\mu_{o}$ even smaller depending on $\mu_{b}$ and $\Gmm$, we may apply the inverse function theorem to ensure that the image of the bi-Lipschitz map $T(\calU_{\Gmm})$ covers a ball of radius $\calO_{\Gmm}(\mu_{o} \Dlt \bfE)$ around $\bfQ[(\td{g}_{out},\td{k}_{out});A_{16}]$; here, it is important that $\calU_{\Gmm}$ in Proposition~\ref{prop:bump-six} is larger than the subset of $\bbR^{10}$ defined by \eqref{eq:obs-free-unit:EP}--\eqref{eq:obs-free-unit:CJ}. By inverting this bi-Lipschitz map and composing with $(g_{\Dlt Q'}, k_{\Dlt Q'})$, we may produce an admissible family on $A_{16}$ verifying \eqref{eq:gluing-q}--\eqref{eq:gluing-adm-lip} with $\eps = \calO_{\Gmm}(\mu_{o}^{\frac{1}{2}} (\Dlt \bfE)^{\frac{1}{2}})$ and $K = \calO_{\Gmm}(1)$; that $(\td{g}_{out}, \td{k}_{out})$ verifies \eqref{eq:gluing-epss} is also clear. Choosing $\eps_{o}$ sufficiently small, we can apply Theorem~\ref{thm:gluing-annulus} and conclude the proof. \hfill \qedsymbol

\subsection{Proof of Theorem~\ref{thm:obs-free-af}} \label{subsec:obs-free-af}
Given the two asymptotically flat solutions, we first rescale them using \eqref{eq:scaling}. Arguing as in the proof of \eqref{eq:extr-r} in the proof of Theorem~\ref{thm:corvino-schoen}, we have
\begin{align*}
    (\bfE,\bfP)[(g_{\Box}^{(r)},k^{(r)}_{\Box});A_{R_{\Box}}] &= r^{-1}(\bfE,\bfP)[(g_{\Box},k_{\Box});A_{r R_{\Box}}]=r^{-1}(\bfE^{ADM},\bfP^{ADM})[(g_{\Box},k_{\Box})]+\mathcal{O}(r^{-2\alpha}), \\
    (\bfC,\bfJ)[(g_{\Box}^{(r)},k^{(r)}_{\Box});A_{R_{\Box}}]&= r^{-2}(\bfC,\bfJ)[(g_{\Box},k_{\Box});A_{rR_{\Box}}]=\mathcal{O}(r^{-2\alpha}),
\end{align*}
for $\Box = in$ or $out$, where $R_{in} = 1$ and $R_{out} = 32$.
Therefore, for sufficient large $r\geq R_0$, we verify the conditions \eqref{eq:obs-free-unit:EP}--\eqref{eq:obs-free-unit:CJ}.
The asymptotic flatness condition \eqref{eq:w-af} implies \eqref{eq:obs-free-unit:data}, since
\begin{align*}
    \|(g_{in}^{(r)}-\delta,k_{in}^{(r)})\|_{H^s\times H^{s-1}(A_1)}^2+\|(g_{out}^{(r)}-\delta,k_{out}^{(r)})\|_{H^s\times H^{s-1}(A_{32})}^2\lesssim r^{-2\alpha}<\mu_o \Delta\bfE
\end{align*}
for sufficiently large $r \geq R_{0}$. We can then apply Theorem~\ref{thm:obs-free-unit} to conclude the proof. \hfill \qedsymbol

\appendix
\section{Proof of Lemma~\ref{lem:kerr-extr}} \label{sec:kerr-extr}
For simplicity, we omit the superscript ${}^{Kerr}$ below. Following \cite{ChrCorIse}, we use the Kerr-Schild coordinates $(x^{0}, x^{1}, x^{2}, x^{3})$, with respect to which
\begin{equation} \label{eq:kerr-spt}
	(\bfg_{M, a})_{\mu \nu} = \bfeta_{\mu \nu} + \frac{2M r^{3}}{r^{4} + a^{2} (x^{3})^{2}} \bfk_{\mu} \bfk_{\nu}, \quad
	\bfk = \ud x^{0} + \frac{r x^{1} + a x^{2}}{r^{2} + a^{2}} \ud x^{1} + \frac{r x^{2} - a x^{1}}{r^{2} + a^{2}} \ud x^{2} + \frac{x^{3}}{r} \ud x^{3}.
\end{equation}
where $r$ is implicitly defined by $r^{2} ((x^{1})^{2} + (x^{2})^{2}) + (r^{2} + a^{2}) (x^{3})^{2} = r^{2} (r^{2} + a^{2})$.

According to \cite[Appendix~F]{ChrDel}, there exists a map
\begin{equation*}
\calE \to \bbR \times SO^{+}(1, 3) \times \bbR \times SO(3) \times \bbR^{3}, \quad Q = (\bfE, \bfP, \bfJ, \bfC) \mapsto (M(\bfE, \bfP), \Lmb(\bfE, \bfP), a(Q), R(Q), \xi(Q))
\end{equation*}
such that $(g_{Q}, k_{Q}) = (g_{M, a}^{(R, \xi, \Lmb)}, k_{M, a}^{(R, \xi, \Lmb)})$, where the right-hand side is the initial data set induced from $\bfg_{M, a}$ on the hypersurface $\set{y^{0} = 0}$, where $y^{\mu} = \tensor{\Lmb}{^{\mu}_{\nu}}(\tensor{R}{^{\nu}_{\lmb}} x^{\lmb} + \bfxi^{\nu})$ (i.e., rotate, translate, then boost). In fact,
\begin{equation} \label{eq:kerr-para}
	M = \sgn \bfE \sqrt{\bfE^{2} - \abs{\bfP}^{2}}, \quad
	\Lmb = \Lmb_{\ell} \hbox{ with } \ell = \bfE^{-1} \bfP, \quad
	a = M^{-1} \abs{\bfJ(\Lmb_{\ell}^{-1} Q)}, \quad
	\xi = - M^{-1} \bfC(\Lmb_{\ell}^{-1} Q)
\end{equation}
where $Q' = \Lmb(Q)$ is defined by the relations 
\begin{equation*}
    \bbP_{\mu}(Q') = \tensor{\Lmb}{_{\mu}^{\mu'}} \bbP_{\mu'}(Q), \quad 
    \bbJ_{\mu \nu}(Q') = \tensor{\Lmb}{_{\mu}^{\mu'}} \tensor{\Lmb}{_{\nu}^{\nu'}} \bbJ_{\mu' \nu'}(Q),
\end{equation*} 
where $(\bbP_{0}, \bbP_{i}, \bbJ_{i 0}, \bbJ_{jk})(Q) = (\bfE, \bfP_{i}, \bfC_{i}, \tensor{\eps}{^{i}_{jk}} \bfJ_{i})(Q)$ (cf.~\eqref{eq:spt-averge-charges}, Lemma~\ref{lem:bump-poincare} and \cite[Prop.~E.1]{ChrDel}). Moreover, $R$ is a rotation in the plane spanned by $\sum_{j} \bfJ(\Lmb_{\ell}^{-1} Q)_{j} \rd_{x^{j}}$ and $\rd_{x^{3}}$ that maps the former to the latter. Note also that the Lorentz factor $\gmm$ of $\Lmb$ equals $(\bfE^{2} - \abs{\bfP}^{2})^{-\frac{1}{2}} \abs{\bfE}$.

In what follows, we shall freely make use of the following basic facts:
\begin{equation} \label{eq:kerr-some-facts}
\gmm^{-1} \aleq \abs{\Lmb}_{\bbR^{4} \to \bbR^{4}} \aleq \gmm, \quad
\abs{(a, \xi)} \aeq \gmm^{-2} M^{-1} (\bfC, \bfJ), \quad
\abs{\abs{x}^{-\alp} - \abs{x - \xi}^{-\alp}} \aleq_{\alp} \abs{x}^{-\alp-1} \hbox{ for } \abs{x} \geq 10 \abs{\xi}.
\end{equation}
Moreover, to simplify the exposition, we now adopt the following conventions: (1) $M, a \geq 0$ (the general case is similar), (2)~all bounds we state in the variables $y$ are to be satisfied for $\abs{y} \geq C_{R}(M + M^{-1} \abs{(\bfC, \bfJ)})$ for some $C_{R} = C_{R}(\gmm)$, which may vary line by line, and (3)~unless otherwise specified, all implicit constants may depend on $n$ and $\gmm$.

We begin with
\begin{align*}
    \partial_M \bfg_{M,a}=\frac{2r^3}{r^4+a^2(x^3)^2} \bfk_\mu \bfk_\nu,\quad \partial_a \bfg_{M,a}=\rd_{a} \left(\frac{2M r^{3}}{r^{4} + a^{2} (x^{3})^{2}} \right) \bfk_\mu \bfk_\nu+\frac{2Mr^3}{r^4+a^2(x^3)^2}(\partial_a \bfk_\mu \bfk_\nu+\bfk_\mu\partial_a \bfk_\nu).
\end{align*}
Notice that $|\partial^{(n)} \bfk_\mu|\lesssim \abs{x}^{-n}$, so we have
$|x|^n|\partial^{(n)}\partial_M \bfg_{M,a}|\lesssim \abs{x}^{-1}$.
Moreover, notice that
\begin{align*}
    \rd_{a} r = - \frac{r ((x^{1})^{2} + (x^{2})^{2})}{(r^{2}+a^{2})^{2} - a ((x^{1})^{2} + (x^{2})^{2})} = \calO(\abs{x}^{-1}),
\end{align*}
from which it follows that
\begin{align*}
    |x|^n\left|\partial^{(n)} \rd_{a} \left(\frac{2M r^{3}}{r^{4} + a^{2} (x^{3})^{2}} \right)\right|\lesssim \frac{M a}{\abs{x}^3} + \frac{M}{\abs{x}^{2}}, \quad
 \abs{x}^{n} \abs*{ \rd^{(n)} \partial_a \bfk_\mu} \aleq \mathcal{O}(\abs{x}^{-1}).
\end{align*}
We conclude that
\begin{align*}
    |x^n||\partial^{(n)}\partial_a(\bfg_{M,a})|\lesssim \frac{M a}{\abs{x}^3}+\frac{M}{\abs{x}^2}.
\end{align*}
Recall $y^{\mu} = \tensor{\Lmb}{^{\mu}_{\nu}}(\tensor{R}{^{\nu}_{\lmb}} x^{\lmb} + \bfxi^{\nu})$, as well as \eqref{eq:induced-data}. Using the preceding identities, \eqref{eq:kerr-spt} and \eqref{eq:kerr-some-facts}, we have
\begin{align*}
    |y|^n\left|\partial_y^{(n)} g_{M,a}^{(R, \xi, \Lmb)}\right|\lesssim M|y|^{-1},\quad |y|^n\left|\partial_y^{(n)}\partial_Mg_{M,a}^{(R, \xi, \Lmb)}\right|\lesssim |y|^{-1}, \quad
       |y|^n\left|\partial_y^{(n)}\partial_a g_{M,a}^{(R, \xi, \Lmb)}\right|\lesssim M|y|^{-2}.
\end{align*}
We may also compute\footnote{The exact choices of the coordinates $\Lmb$ and $R$ on $SO^{+}(1, 3)$ and $SO(3)$, respectively, are not too important. For instance, they may be regarded as local coordinates near given $\Lmb(\bfE, \bfP)$ and $R(Q)$ obtained by left-translation of normal coordinates near the identity in $SO^{+}(1, 3)$ and $SO(3)$, respectively.}
\begin{align*}
    |y|^n\left|\partial_y^{(n)}\partial_\xi g_{M,a}^{(R, \xi, \Lmb)}\right|\lesssim M|y|^{-2},\quad |y|^n\left|\partial_y^{(n)}\partial_\Lambda g_{M,a}^{(R, \xi, \Lmb)}\right|\lesssim M|y|^{-1}, \quad
    |y|^n\left|\partial_y^{(n)}\partial_{R} g_{M,a}^{(R, \xi, \Lmb)}\right|\lesssim Ma|y|^{-2},
\end{align*}
where we used the rotation invariance of the leading term of the Kerr metric for the last bound. Moreover, for $g_{M, A}^{(R, \xi, \Lmb), -}$, the leading term is cancelled and $\abs{y}^{-1}$ improve to $\abs{y}^{-2}$ in the above bounds. Finally, by \eqref{eq:induced-data}, similar bounds follow for $k_{M, a}^{(R, \xi, \Lmb)}$ and $k_{M, a}^{(R, \xi, \Lmb), +}$.

From the preceding assertions, \eqref{eq:Kerr-af}--\eqref{eq:Kerr-parity} immediately follow. To establish the remaining bounds, observe also that, by \eqref{eq:kerr-para} and the definition of $R$ (which is locally well-defined),
\begin{align*}
    \abs*{\frac{\partial \Lmb}{\partial(\bfE,\bfP)}} \lesssim_{\gamma} \frac{1}{M},\,
    \abs*{\frac{\partial R}{\partial (\bfC, \bfJ)}} \lesssim_{\gamma} \frac{1}{M a},\,
    \abs*{\frac{\partial R}{\partial(\bfE, \bfP)}} \lesssim_{\gamma} \frac{\abs{(\bfC, \bfJ)}}{M} \frac{1}{M a},\,
    \abs*{\frac{\partial (a,\xi)}{\partial(\bfC, \bfJ)}} \lesssim_{\gamma} \frac{1}{M},\,
    \abs*{\frac{\partial(a,\xi)}{\partial(\bfE,\bfP)}} \lesssim_{\gamma} \frac{\abs{(\bfC, \bfJ)}}{M^{2}}.
\end{align*}
For \eqref{eq:Kerr-Lip-EP}, we have
\begin{align*}
    &|y|^n|\partial^{(n)}\partial_{\bfE,\bfP}g_{M, a}^{(R, \xi, \Lmb)}|
    \lesssim \abs*{\tfrac{\partial M}{\partial(\bfE,\bfP)}} |y|^n|\partial^{(n)}\partial_{M} g_{M, a}^{(R, \xi, \Lmb)}|
    +\abs{\tfrac{\partial \Lmb}{\partial(\bfE,\bfP)}} |y|^n|\partial^{(n)}\partial_{\Lmb} g_{M, a}^{(R, \xi, \Lmb)}|\\
    &+\abs{\tfrac{\partial a}{\partial(\bfE,\bfP)}} |y|^n|\partial^{(n)}\partial_{a}g_{M, a}^{(R, \xi, \Lmb)}|
    +\abs{\tfrac{\partial R}{\partial(\bfE,\bfP)}} |y|^n|\partial^{(n)}\partial_{R} g_{M, a}^{(R, \xi, \Lmb)}|
    +\abs{\tfrac{\partial \xi}{\partial(\bfE,\bfP)}} |y|^n|\partial^{(n)}\partial_{\xi} g_{M, a}^{(R, \xi, \Lmb)}| \\
    &\lesssim |y|^{-1}+M^{-1}M|y|^{-1}
    +\tfrac{\abs{(\bfC, \bfJ)}}{M^{2}}M|y|^{-2}+\tfrac{\abs{(\bfC, \bfJ)}}{a M^{2}}aM|y|^{-2}
    \lesssim |y|^{-1}
\end{align*}
and we have a similar bound for $k_{M, a}^{(R, \xi, \Lmb)}$. The proof of \eqref{eq:Kerr-Lip-EP-parity} is similar. For \eqref{eq:Kerr-Lip-CJ}, we have
\begin{align*}
    &|y|^n|\partial^{(n)}\partial_{\bfC,\bfJ}g_Q^{Kerr}| \\
    &\lesssim \abs{\tfrac{\partial a}{\partial(\bfC,\bfJ)}}|y|^n|\partial^{(n)}\partial_{a}g_{M, a}^{(R, \xi, \Lmb)}|
    +\abs{\tfrac{\partial R}{\partial(\bfC,\bfJ)}} |y|^n|\partial^{(n)}\partial_{R} g_{M, a}^{(R, \xi, \Lmb)}|
    +\abs{\tfrac{\partial \xi}{\partial(\bfC,\bfJ)}} |y|^n|\partial^{(n)}\partial_{\xi} g_{M, a}^{(R, \xi, \Lmb)}|\\
    &\lesssim M^{-1}M|y|^{-2}+(Ma)^{-1}Ma|y|^{-2}+M^{-1}M|y|^{-2}\lesssim|y|^{-2}
\end{align*}
and we have a similar bound for $k_{M, a}^{(R, \xi, \Lmb)}$. \hfill \qedsymbol

\bibliographystyle{siam}
\bibliography{gr-gluing}

\end{document}